\documentclass[a4paper]{article}

\usepackage[pdftex]{graphicx}
\usepackage{amsmath,amssymb,amsthm}
\usepackage{mathtools,thm-restate}
\usepackage{thmtools}
\usepackage{cleveref}
\usepackage{array}
\usepackage[all]{xy}
\usepackage{enumitem}
\usepackage{comment}
\usepackage{accents}
\usepackage{mymacro}

\declaretheorem[numberwithin=section]{theorem}
\declaretheorem[sibling=theorem]{lemma}
\declaretheorem[sibling=theorem]{proposition}
\declaretheorem[sibling=theorem]{corollary}

\declaretheorem[style=definition,numbered=no]{definition}

\declaretheorem[style=remark,sibling=theorem]{example}
\declaretheorem[style=remark,sibling=theorem]{remark}
\declaretheorem[style=remark,numbered=no]{notation}

\numberwithin{equation}{section}
\numberwithin{figure}{section}

\title{Transversality theorem in highly relative situations and its application}
\date{\today}
\author{Jun Yoshida}

\begin{document}

\maketitle
\tableofcontents

\section*{Introduction}
\addcontentsline{toc}{section}{Introduction}
\label{intro}

In this paper, we aim to provide a notion of ``relative objects'', i.e. objects equipped with some sort of subobjects, in differential topology.
In this area, we often encounter them.
For example, people study the theory of (higher) knots and links, embeddings of codimension two in the sphere of a certain dimension.
In other words, knots and links are kinds of relative objects in the category $\mathbf{Mfd}$ of manifolds and smooth maps, and the knot theory is nothing but the study of them.
Another example is a manifold with boundaries.
Every manifold $X$ with boundaries has a canonical submanifold $\partial X\subset X$.
In addition, for another manifold $Y$ with boundaries, a smooth map $F:X\to Y$ is often required to map $\partial X$ into $\partial Y$.
This requirement is equivalent to saying that $F$ is a map $(X,\partial X)\to (Y,\partial Y)$ of relative objects.
The reader will find other numerous examples.
In spite of these active researches, there seem to be poor notions to deal with them in differential topology in comparison with, for instance, in homotopy theory.
Even if we have an embedding $X'\hookrightarrow X$ in a differentiable context, people often think of the pair $(X,X')$ as a homotopical object.
We want more direct differential calculus on pairs and to extend classical notions and theories to relative situations; e.g. functions, vector fields, jet bundles, singularities, and so on.

One motivation comes from a variant of topological field theory.
Recall that we have a symmetric monoidal category $\mathbf{Cob}_{d+1}$ whose objects are closed oriented manifolds of dimension $d$ and whose morphisms are (diffeomorphism classes of) $(d+1)$-dimensional bordisms.
The composition is just gluing bordisms, and the symmetric monoidal structure is given by the disjoint unions of manifolds.
Atiyah pointed out in \cite{Atiyah88} that a topological quantum field theory (TQFT) can be formulated as a symmetric monoidal functor from $\mathbf{Cob}_{d+1}$ to the category of vector spaces.
The theory of TQFTs is a recent hot topic, and some researchers have introduced variants of the notion.
We here mention, especially, TQFT with defects; it is a theory on bordisms with ``pictures'' drawn.
Each bordism is equipped with a submanifold.
Moreover, one sometimes consider recursive version: submanifolds has submanifolds which has submanifolds...
Hence, we need to treat with ``relative'' bordisms.
It is expected that $2$-dimensional TQFTs with defects are highly related to planar algebras introduced by Jones \cite{Jon99} and other kinds of ``graphical calculi'' in monoidal categories.
To formulate these relations, the classification of bordisms is important.
This is why we attempt to develop a relative notion in differential topology.

To establish the notion, what should we do?
In modern perspective, one of the best ways to understand an object is to investigate the function space on it.
Hence, the main object in this paper is the space $C^\infty(\mathcal X,\mathcal Y)$ of relative smooth maps between relative objects $\mathcal X$ and $\mathcal Y$ in $\mathbf{Mfd}$ which we will call arrangements of manifolds.
Notice that $C^\infty(\mathcal X,\mathcal Y)$ is topologized as a subspace of the usual space $C^\infty(X,Y)$ of smooth maps with Whitney $C^\infty$-topology.
The reader will find that the space $C^\infty(\mathcal X,\mathcal Y)$ is as good as $C^\infty(X,Y)$ and that many arguments in the usual case are still valid in relative cases.
In particular, the goal of this paper is a relative version of Transversality Theorem, which is a foundamental result in the singularity theory of smooth maps originally proved by Thom \cite{Thom56} (and ``multi'' version by Mather \cite{Mat70V}).
The formal statement is as follows:

\begin{restatable*}{theorem}{thmJetTransversal}
\label{thm:multijet-transv}
Let $\mathcal X$ and $\mathcal Y$ be excellent arrangements of manifolds of shape $S$.
Suppose we are given a submanifold $W\subset J^r_{\mathbf n}(\mathcal X,\mathcal Y)$ for a map $\mathbf n:S^{[1]}\to\mathbb Z_{\ge 0}$ and a non-negative integer $r\ge 0$.
Then the subset
\[
\mathcal T_W :=\{F\in C^\infty(\mathcal X,\mathcal Y)\mid j^r_{\mathbf n}(F)\pitchfork W\} \subset C^\infty(\mathcal X,\mathcal Y)
\]
is a residual subset.
Moreover, if $W$ is compact, then $\mathcal T_W$ is open.
\end{restatable*}

Here, $J^r_{\mathbf n}(\mathcal X,\mathcal Y)$ is a relative version of multi-jet bundles defined in \cite{Mat70V}, and $j^r_{\mathbf n}F$ is the multijet of a smooth map $F$.
We note that in the simplest relative case, just pairs of manifolds, the theorem was proved by Ishikawa \cite{Ish98}.
It is, however, still unenough to consider complicated situation; for example, Ishikawa's result does not cover manifolds with corners, and we need it.

Now, let us explain the idea more precisely.
Our main question is simple: what are relative objects in a given category $\mathcal C$.
If one take $\mathcal C=\mathbf{Mfd}$, this is nothing but the question considered above.
One naive answer is as follows: suppose the category $\mathcal C$ is equipped with a notion of ``embeddings''.
An object $X'\in\mathcal C$ is called a subobject of $X$ if it is equipped with an embedding $X'\hookrightarrow X$.
Then, a relative object in $\mathcal C$ is a pair $(X,\{X_\alpha\}_\alpha)$ of an object $X$ and a family $\{X_\alpha\}_\alpha$ of subobjects of $X$.
Writing $\mathbf{Emb}_{\mathcal C}\subset\mathcal C$ the subcategory consisting of embeddings, one may notice that a relative object $(X,\{X_\alpha\}_\alpha)$ gives rise to a functor $\mathcal X:S\to\mathbf{Emb}_{\mathcal C}$ for a poset $S$ with a maximum element such that $\mathcal X(\max S)=X$.
This is a primitive form of the notion, and we will call $\mathcal X$ a pre-arrangement in $\mathcal C$ of shape $S$.
In this point of view, a relative morphism $\mathcal X\to\mathcal Y$ between pre-arrangements of the same shape $S$ is defined as a natural transformation between functors $\mathcal X,\mathcal Y:S\to\mathbf{Emb}_{\mathcal C}\hookrightarrow\mathcal C$.
In particular, in the case $\mathcal C=\mathbf{Mfd}$, this is the definition of the space $C^\infty(\mathcal X,\mathcal Y)$.
If $S$ is a finite lattice, and if $\mathcal X$ sends meets in $S$ to pullbacks in $\mathcal C$, we call $\mathcal X$ an arrangement.
This is important when we are interested in the intersections of subobjects.

For example, consider the category $\mathbf{MkFin}$ whose objects are marked finite sets, i.e. finite sets with some elements ``marked''.
The morphisms are usual maps preserving marked elements.
For a marked finite set $I$, we denote by $I_0\subset I$ the set of marked elements.
We say a morphism $I\to J\in\mathbf{MkFin}$ is an embedding if it is injective (while some of the readers feel uncomfortable).
Then, an arrangement $\mathcal I$ in $\mathbf{MkFin}$ of shape $S$ is equivalent to a marked finite set $I$ equipped with lattice homomorphisms $\mathcal I:S\to 2^I$ and $\mathcal I_0:S\to 2^{I_0}$ preserving infimums and such that $\mathcal I(s)_0\mathcal I(s)$ for each $s\in S$.
Note that this arrangement gives rise to an arrangement of manifolds of shape $S$; we define a functor $\mathcal E^{\mathcal I}:S\to\mathbf{Emb}$ by
\[
\mathcal E^{\mathcal I}(s) := \mathbb R^{\mathcal I_0(s)}\times\mathbb R_+^{\mathcal I(s)\setminus\mathcal I_0(s)}\,.
\]
The arrangement $\mathcal E^{\mathcal I}$ plays an important role when we define relative jet bundles and consider Transversality Theorem.

We here sketch the structure of this paper.
First, in \Cref{sec:prelim}, we review the definition and elementary results on manifolds with corners.
Throughout the paper, we mainly talk about manifolds with corners rather than just manifolds.
We decided to give as many definitios and notations as possible because other authors used different and inequivalent conventions especially on treatment of corners; e.g. what is the tangent space at a corner?
One can find major ones in the paper \cite{Joy09}.
As an unfortunate result, our convention might becomes different from any other literatures even in the reference list.
We, however, believe that ours is the most natural one from the algebraic point of view.
Indeed, we employed algebraic definitions of jet bundles and (co)tangent bundles.
They are parallel to those in algebraic geometry.
The author learned the idea in \cite{ScP09}, but the most comprehensive textbook is \cite{MoerdijkReyes1991}.
Note that, except for notations and definitions, we mainly follow \cite{Mic80} for manifolds with corners and \cite{GG73} for jet bundles and Whitney $C^\infty$-topology.

In \Cref{sec:mult-relative}, we discuss relative objects in the category $\mathbf{Mfd}$.
We will formalize the idea given above, and some elementary constructions will be mentioned.
We also show basic properties of the space $C^\infty(\mathcal X,\mathcal Y)$ including that the Baire property.
In addition, we will define a special class of arrangements of manifolds, called excellent arrangements.
They are, roughly, manifolds modeled on the arrangements $\mathcal E^{\mathcal I}$ defined above.
An important example is a manifold with faces, wihch we will discuss in \Cref{sec:corner-arr}.
This notion was originally introudced by \cite{Jan68}.
For them, we see most of the behaviors of corners can be described in terms of arrangements.
The new notion of edgings will be introduced, which stands for a condition for maps to send each corners to designated one.
This and the relative notions enable us to consider the space $\mathcal F^\beta(X,Y)$ of smooth maps between manifolds with faces along an edging $\beta$.
Furthermore, we will prove, so-called, Collar Neighborhood Theorem in a certain form.
Our version is essentially the same as Laures' \cite{Lau00} but a bit generalized.
In particular, we will consider collarings along general edgings.
This result together with the uniqueness theorem is used to make bordisms into a (higher) category.

The proof of \Cref{thm:multijet-transv} will be given in \Cref{sec:rel-transversality} after defining a relative version of jet bundles.
Thanks to the algebraic definition, the definition is almost straight forward.
For excellent arrangements $\mathcal X$ and $\mathcal Y$ of shape $S$, the relative jet bundle $J^r_\kappa(\mathcal X,\mathcal Y)$ is defined for each non-negative integer $r\ge 0$ and each interval $\kappa=[\kappa_0,\kappa_1]\subset S$.
It consists of ``relative jets'', which are roughly jets of smooth maps preserving arrangements, while we have to be careful around corners.
Obviously $J^r_\kappa(\mathcal X,\mathcal Y)$ is a smooth fiber bundle over the manifold $\mathcal X\double(\kappa_0\double)\times\mathcal Y(\kappa_0)$, here $\mathcal X\double(\kappa_0\double):=\mathcal X(s)\setminus\bigcup_{t<s}\mathcal X(t)$.
In addition, the ``multi'' version will be also considered.
Notice that the relative jet bundles are really related to the space of polynomial maps.
For two arrangements $\mathcal I$ and $\mathcal J$, consider the space $P^r(\mathcal I,\mathcal J)$ of the polynomial maps between arrangements $\mathcal E^{\mathcal I}$ and $\mathcal E^{\mathcal J}$.
It is easily verified that $P^r(\mathcal I,\mathcal J)$ is diffeoomrphic to a finite dimensional Euclidean space.
Since excellent arrangements are modeled on arrangements of the form $\mathcal E^{\mathcal I}$, the relative jet bundles should be modeled on $P^r(\mathcal I,\mathcal J)$.
Along this idea, after certain discussion on corners, we will see that smooth maps $F:\mathcal X\to\mathcal Y$ between excellent arrangements admit perturbations with parameters in a subspace of $P^r(\mathcal I,\mathcal J)$.
This observation is a key step for the proof of \Cref{thm:multijet-transv}.
Indeed, the actual proof is a combination of it and the classical Parametric Transversality Theorem.

In the final section, \Cref{sec:appemb}, we give an application of \Cref{thm:multijet-transv}, Embedding Theorem of manifolds with faces.
In compact cases, this theorem was proved in the paper \cite{Lau00}.
Our result is stronger than his result; e.g. it covers non-compact cases and embeddings into general convex polyhedra.
Moreover, it involves the residuality of embeddings in the space $C^\infty(X,\mathbb R^k\times\mathbb R^n_+)$ rather than just the existence.
Precisely, we prove the following

\begin{restatable*}{theorem}{thmImmersion}
\label{thm:imm-thm}
Let $X$ and $Y$ be two manifolds with finite faces, and let $\beta$ be an edging of $X$ with $Y$.
Assume we have $2\cdot \dim X\le\dim Y$.
Then, immersions along $\beta$ form a residual and, hence, dense subset of the space $\mathcal F^\beta(X,Y)$.
\end{restatable*}

\begin{restatable*}{theorem}{thmEmbedding}
\label{thm:emb-thm}
Let $X$ and $Y$ be manifolds with finite faces, and let $\beta$ be an edging of $X$ with $Y$ so that the space $\mathcal F^\beta(X,Y)$ is non-empty.
Then, for any sufficiently large integer $n>0$, there is a proper embedding $F:X\to Y\times\mathbb R^n$ such that
\begin{enumerate}[label={\rm(\roman*)}]
  \item for each $\tau\in\Gamma_Y$, $F(\partial^\beta_\tau X)\subset\partial_\tau Y$;
  \item for each connected face $D\in\operatorname{bd}Y$ of $Y$, $F\pitchfork D$.
\end{enumerate}
\end{restatable*}

If $Y$ is a polyhedron in the Euclidean space, it is verified that the space $\mathcal F^\beta(X,Y)$ is non-empty for any edging $\beta$.
Hence, we always have an embedding $X\to Y\times\mathbb R^n$ along $\beta$ for a sufficiently large integer $n$.

The idea of the proof of \Cref{thm:emb-thm} is simple.
We will see that some characteristic properties of embeddings can be described in terms of the transversality of jets with certain submanifolds.
In particular, \Cref{thm:imm-thm} is a consequence of this observation and \Cref{thm:multijet-transv}, and it is also follows that injective maps also form a residual subset as soon as they has enough codimensions.
Finally, \Cref{thm:emb-thm} follows from the observation that proper maps form a non-empty open subset in $\mathcal F^\beta(X,Y\times\mathbb R^n)$ as soon as $\mathcal F^\beta(X,Y)\neq\varnothing$.

\subsection*{Acknowledgements}
This paper is written under the supervision of Toshitake Kohno.
I would like to thank him for great encouragement and advice.
I would also like to appreciate my friends, Yuta Nozaki, Satoshi Sugiyama, and Shunsuke Tsuji.
They pointed out some errors and helped me improve the paper.
Finally, I thank the mathematical research institute MATRIX in Australia where part of this research was performed.
This work was supported by the Program for Leading Graduate Schools, MEXT, Japan.
This work was supported by JSPS KAKENHI Grant Number JP15J07641.

\section{Preliminaries}
\label{sec:prelim}

In this first section, we prepare elementary notions.

\subsection{Manifolds with corners}
First of all, we give a quick review on the theory of manifolds with corners.
The notion was oringally introduced by Cerf \cite{Cer61} and Douady \cite{Dou61}.
Nowadays, however, there are some inequivalent definitions and conventions.
A good survey of them was presented by Joyce in \cite{Joy09}.
Nevertheless, the author believes that the definitions in this section are standard ones.

Briefly, a manifolds with corners is a manifold modeled on the space
\[
\mathbb R^n_+ := \{(x_1,\dots,x_n)\in\mathbb R^n\mid 1\le\forall i\le n: x_i\ge 0\}
\]
instead of the Euclidean space $\mathbb R^n$.
Hence, a manifold $X$ with corners is a second countable Hausdorff space equipped with a family $\{(U_\alpha,\varphi_\alpha)\}_\alpha$ of open cover $\{U_\alpha\}_\alpha$ and open embeddings $\varphi_\alpha:U_\alpha\to\mathbb R^n_+$ so that it defines a \emph{smooth structure} on $X$.
In this case, the number $n$ is called the dimension of $X$.
Note that, for the smoothness, we use the convention that if $A\subset\mathbb R^m$ and $B\subset\mathbb R^n$ are subsets of the Euclidean spaces, then a map $F:A\to B$ is smooth if and only if it extends to a map $\widehat F:U\to\mathbb R^n$ on an open neighborhood $U$ of $A$ which is smooth in the usual sense.
We denote by $\mathbf{Mfd}$ the category of manifolds with corners and smooth maps.

\begin{remark}
Some authors require additional conditions for smooth maps.
For example, in \cite{Joy09}, the terminology weakly smooth maps is used for smooth maps in our convention.
Of course, different definitions give rise to different categories.
\end{remark}

To investigate manifolds with corners, it is convenient to introduce the notion of marked sets.
A marked set is a set $I$ with some elements marked; equivalently, a marked set is a pair $(I,I_0)$ os a set $I$ and a subset $I_0\subset I$ whose elements are ``marked''.
A map of marked finite sets is a map preserving marked elements.
We denote by $\mathbf{MkFin}$ the category of marked finite sets and maps of them.
We define a map $\mathbb H^\bullet:\mathbf{MkFin}\to\mathbf{Mfd}$ in the following way: for a marked finite set $I$, we set
\[
\mathbb H^I := \mathbb R^{I_0}\times\mathbb R^{I_+}_+\subset\mathbb R^{\#I}\ .
\]
If $\varphi:I\to J$ is a map of marked finite sets, then $\varphi_\ast:\mathbb H^I\to\mathbb H^J$ is defined by
\[
\varphi_\ast\left((x_i)_i\right)
= \left(\sum_{i\in\varphi^{-1}\{j\}} x_i\right)_j\ .
\]
We give special notations to frequently used marked finite sets: for natural numbers $0\le k\le m$, we write
\[
\langle m|k\rangle := \{1,\dots,m\}
\ \text{with}\ 1,2,\dots,k\ \text{marked}\ .
\]
For example, we have $\mathbb H^{\langle m|k\rangle}\cong\mathbb R^k\times\mathbb R^{m-k}_+$.
As a consequence of the smooth invariance of domain, we have the following result.

\begin{lemma}
\label{lem:mfd-chart}
Let $X$ be a manifold with corners.
Then, each point $p\in X$ admits an open neighborhood $U\in X$, a marked finite set $I$, and an open embedding $\varphi:U\to\mathbb H^I$ which is a local diffeomorphism such that $\varphi(p)=0$.
Moreover, the marked finite set $I$ is unique up to isomorphisms in $\mathbf{MkFin}$.
\end{lemma}

In the case of \Cref{lem:mfd-chart}, we call $(U,\varphi)$ a coordinate chart, or chart briefly, on $X$ centered at $p$.
By virtue of the last assertion, we may assume $\varphi:U\to\mathbb H^{\langle m|k\rangle}$ for unique natural numbers $0\le k\le m$.
We call $p$ a corner of codimension $m-k$.
Notice that $m$ is just the dimension of $X$.
We set $\partial_c X\subset X$ to be the subset of corners of codimension $c$.
It is verified that
\[
\partial_c(X\times Y) = \coprod_{c_1+c_2=c}(\partial_{c_1}X)\times(\partial_{c_2}Y)\ .
\]

Similarly to the usual manifolds, each manifolds with corners admits a canonical sheaf on it, namely the sheaf $U\mapsto C^\infty(U)$ of smooth (real-valued) functions.
Obviously, it is a sheaf of $\mathbb R$-algebras so that we may regard manifolds with corners as locally ringed spaces.
For a manifold $X$ with corners and a point $p\in X$, we denote by $C^\infty_p(X)$ the ring of germs of smooth functions defined near $p$ and by $\mathfrak m_p(X)$ its maximal ideal, namely
\[
\mathfrak m_p(X) := \{f\in C^\infty_p(X)\mid f(p)=0\}\ .
\]
Moreover, each chart $(U,\varphi)$ on $X$ centered at $p$ gives rise to an $\mathbb R$-algebra homomorphism
\[
\varphi_!:C^\infty(U)\ni f \mapsto \sum_\alpha \frac{\partial^{|\alpha|}f\varphi^{-1}}{\partial x^\alpha}(0) x^\alpha \in \mathbb R[\![x_1,\dots,x_n]\!]
\]
into the $\mathbb R$-algebra of formal power series, which sends $\mathfrak m_p(X)$ onto the maximal $(x_1,\dots,x_n)$ of the local ring $\mathbb R[\![x_1,\dots,x_n]\!]$.
Hence, a smooth $F:X\to Y$ induces an $\mathbb R$-algebra homomorphism $C_{F(p)}(Y)\to C_p(X)$ which preserves the maximal ideals.

The above algebraic aspect allows us to describe some geometric notions on manifolds in a more unified ways.
For an $\mathbb R$-algebra $A$, we denote by $\operatorname{Der}(A,\mathbb R)$ the set of $\mathbb R$-valued $\mathbb R$-derivations on $A$, which has a canonical structure of $\mathbb R$-vector space.

\begin{definition}
Let $X$ be a manifold with corners and $p\in X$ be a point.
We define two $\mathbb R$-vector spaces
\[
\begin{gathered}
T_pX := \operatorname{Der}(C^\infty_p(X),\mathbb R) \\
T_p^\ast X := \mathfrak m_p(X) / \mathfrak m_p(X)^2\,,
\end{gathered}
\]
and we call them the tangent and cotangent spaces at $p$ respectively.
\end{definition}

For the space $\mathbb H^I$ with the standard coordinate $(x_i)_{i\in I}$, we have canonical isomorphisms
\[
\begin{gathered}
T_0\mathbb H^I \cong \mathbb R\{\frac\partial{\partial x_i}\mid i\in I\}\ , \\
T_0^\ast\mathbb H^I \cong \mathbb R\{d_0x_i\mid i\in I\}\ .
\end{gathered}
\]
If $X$ is a general manifold with corners, a coordinate function $\varphi:U\to\mathbb H^I$ centered at $p\in X$ gives rise to isomorphisms
\[
\begin{gathered}
\varphi_\ast:T_p X \cong T_0\mathbb H^I \cong \mathbb R^{\#I}\ , \\
\varphi^\ast:T_0^\ast\mathbb H^I \cong T^\ast_pX \cong \mathbb R^{\#I}\ .
\end{gathered}
\]
Notice that the right hand sides depend only on the number of elements of $I$, so the tangent and cotangent spaces are \emph{the same} even on corners.
This allwos us to define the tangent bundle $TX$ and the cotangent bundle $T^\ast X$ on each manifold $X$ with corners in a canonical way, namely
\[
TX = \coprod_{p\in X} T_pX
\ ,\quad T^\ast X = \coprod_{p\in X} T^\ast_pX\ .
\]

In addition, the standard argument shows that there are canonical isomorphisms
\[
T^\ast_pX \cong\mathrm{Hom}_{\mathbb R}(T_pX,\mathbb R)
\ ,\quad T_p X \cong \mathrm{Hom}_{\mathbb R}(T^\ast_pX,\mathbb R)
\]
of $\mathbb R$-vector spaces so that $T^\ast X \cong\mathpzc{Hom}_X(TX,\mathbb R_X)$ and $TX\cong\mathpzc{Hom}_X(T^\ast X,\mathbb R_X)$, where $\mathbb R_X=\mathbb R\times X$ is the trivial bundle.
The following formulas are convenient in the practical computation:
\[
\begin{gathered}
T_{(p,q)}(X\times Y) \simeq T_pX\oplus T_qY \\
T^\ast_{(p,q)}(X\times Y) \simeq T^\ast_pX\oplus T^\ast_qY
\end{gathered}
\]

Smooth sections of the bundle $TX$ are given a special name.

\begin{definition}
Let $X$ be a manifold with corners.
Then a vector fields on $X$ is a smooth section of the tangent bundle $TX\to X$.
We denote by $\mathfrak X(X)$ the $C^\infty(X)$-module of vector fields on $X$.
\end{definition}

Recall that we have $T_pX=\operatorname{Der}(C^\infty_p(X),\mathbb R)$.
Hence, for $\xi\in\mathfrak X(X)$, and for $f\in C^\infty(X)$, we have a map
\[
\xi(f) : X \ni p \mapsto \xi_p(f) \in \mathbb R\ ,
\]
which is smooth.
Then, the space $\mathfrak X(X)$ admits a structure of a Lie algebra over $\mathbb R$ so that $\mathfrak X(X)\cong\operatorname{Der}(C^\infty(X))$ as Lie algebras.

Since we consider manifolds with corners, there is a notion special to them.

\begin{definition}
Let $X$ be a manifold with corners.
\begin{enumerate}[label={\rm(\arabic*)}]
  \item For a point $p\in X$, a vector $v\in T_pX$ in the tangent space is said to be inward-pointing if there is a smooth map $\gamma:\mathbb R_+\to X$ such that $\gamma(0)=p$ and $v=\gamma_\ast(\left.\frac{d}{dt}\right|_0)$.
  \item A vector field $\xi$ on $X$ is said to be inward-pointing if $\xi(p)\in T_pX$ is inward-pointing for each $p\in X$.
\end{enumerate}
\end{definition}

For an inward-pointing vector field $\xi$ on $X$, the condition above guarantees that it locally admits a one-parameter family $\varphi=\{\varphi_t\}_{t\ge 0}$.
It is often the case that $\varphi$ actually defines an smooth map $\varphi:X\times\mathbb R_+\to X$ such that each $\varphi_t=\varphi(\blankdot,t):X\to X$ is an embedding.

Finally, we introduce the notion of transversality of vectors, which admits the most variations.
The version we introduce here is the weakest one.

\begin{definition}
Let $X$ and $Y$ be manifolds with corners, and let $F:X\to Y$ be a smooth map.
\begin{enumerate}[label={\rm(\arabic*)}]
  \item We say that $F$ intersects a submanifold $W\subset Y$ transversally, denoted by $F\pitchfork W$, at $p\in X$ for a point with $F(p)\in W$ if the map
\[
T_pX\oplus T_{F(p)}W \to T_{F(p)}Y
\]
is an epimorphism.
  \item For an arbitrary subset $A\subset W$, we say $F$ intersects $W$ transversally on $A$ if we have $F\pitchfork W$ at every point $p\in X$ with $F(p)\in A$.
In particular the case of $A=W$, we will simply say $F$ intersects $W$ transversally.
\end{enumerate}
\end{definition}

In other words, we have $F\pitchfork W$ at $p\in F^{-1}(W)$ if and only if the composition map
\[
T_pX \xrightarrow{d_pF} T_pY \twoheadrightarrow T_pY / T_pW
\]
is an epimorphism.
In particular, we have two obvious cases:
if $\dim X < \codim W$, we have $F\pitchfork W$ if and only if $F(X)\cap W=\varnothing$.
On the other hand, if $\dim W = \dim Y$, we always have $F\pitchfork W$.

\subsection{Submanifolds}
\label{sec:submfds}

\begin{definition}
Let $X$ and $Y$ be manifolds with corners.
Then a smooth map $F:X\to Y$ is called an immersion (resp. submersion) at $p\in X$ if the induced map
\[
F_\ast:T_pX\to T_{f(p)}Y
\]
is a monomorphism (resp. epimorphism) of $\mathbb R$-vector spaces.
We say $F$ is an immersion (resp. submersion) on a subset $U\subset X$ if $F$ is an immersion (resp. submersion) at any point of $U$.
In particular, we will say, simply, $F$ is an immersion (resp. submersion) if it is immersion (resp. submersion) on whole $X$.
\end{definition}

Note that $F:X\to Y$ is an immersion (resp. a submersion) at $p\in X$ if and only if the induced map
\[
F^\ast:T^\ast_{f(p)}Y\to T^\ast_pX
\]
is an epimorphism (resp. a monomorphism) of $\mathbb R$-vector spaces.
With a little care about corners, we have the following standard criterion:

\begin{proposition}
\label{prop:immersion-chart}
Let $X$ be an $m$-dimensional manifold with corners, and let $Y$ be an $n$-dimensional manifold without boundaries.
Then for a smooth map $F:X\to Y$ and $p\in X$, the following statements are equivalent:
\begin{enumerate}[label={\rm(\alph*)}]
  \item\label{subprop:immersion-chart:immersion} $F$ is an immersion at $p$ (hence $m\le n$).
  \item\label{subprop:immersion-chart:goodchart} There are charts $(U,\varphi)$ on $X$ centered at $p$ and $(V,\psi)$ on $Y$ centered at $F(p)$ such that we have
\[
\psi F \varphi^{-1}(x_1,\dots,x_m) = (x_1,\dots,x_m,0,\dots,0)\,.
\]
\end{enumerate}
Moreover, if $F$ is an immersion, the chart $(U,\varphi)$ in the condition \ref{subprop:immersion-chart:goodchart} can be taken to be a restriction of an arbitrary chart centered at $p$.
\end{proposition}

\begin{definition}
Let $X$ and $Y$ be manifolds with corners.
Then a smooth map $F:X\to Y$ is called an embedding if it is both an immersion and a topological embedding.
\end{definition}

Similarly to topological embeddings, it follows immediately from the definition that, for a fixed embedding $i:Y\to Z$ of manifolds with corners, a smooth map $F:X\to Y$ is an embedding if and only if so is the composition $iF:X\to Z$.

\begin{corollary}
\label{cor:pbchart}
Let $F:X\to Y$ be a smooth map between manifolds, and suppose $\partial Y=\varnothing$.
Then the following are equivalent:
\begin{enumerate}[label={\rm(\roman*)}]
  \item The map $F$ is an embedding.
  \item For each $p\in X$, there are a chart $(V,\psi)$ on $Y$ centered at $F(p)$ and a pullback square:
\[
\xymatrix{
  F^{-1}(V) \ar[r]^F \ar[d] \ar@{}[dr]|(.4){\pbcorner} & V \ar[d]^{\psi} \\
  \mathbb H^{\langle m|k\rangle} \ar[r] & \mathbb R^n }
\]
\end{enumerate}
\end{corollary}

The existence of corners often make situation complicated.
To avoid such difficulties, we use the following result throughout the paper.

\begin{proposition}[2.7 in \cite{Mic80}]
\label{prop:emb-bdless}
Every $n$-dimensional manifold $X$ with corners can be embedded into an $n$-dimensional manifold $\widehat X$ without boundaries as a closed subset.
\end{proposition}

In boundaryless cases, submanifolds are locally defined as linear subspaces in appropriate coordinates, and we can discuss many local algebraic properties on submanifolds.
We cannot, however, take such coordinates in general manifolds with corners because charts around corners are more ``rigid'' than those around internal points.
Fortunately, we have another way to describe submanifolds algebraically in general cases:

\begin{proposition}[Whitney, cf \cite{BL75}]
\label{prop:whitney-zero}
Every closed subset of a manifold $X$ with corners is the zero-set of a smooth function which values in $[0,1]$.
\end{proposition}

\begin{corollary}
\label{cor:submfd-defzero}
Let $X$ be a manifold with corners, and let $X'\subset X$ be a submanifold.
Then, each point $p\in X'$ admits an open neighborhood $U$ in $X$ together with a smooth function $\lambda:U\to[0,1]$ such that $X'\cap U = \{\lambda=0\}\subset U$.
\end{corollary}
\begin{proof}
Notice that every point $p\in X'$ of the submanifold admits a neighborhood $U\subset X$ such that $X'\cap U$ is closed in $U$; indeed, embedding $X\hookrightarrow\widehat X$ into a manifold $\widehat X$ without boundary using \Cref{prop:emb-bdless}, we have a neighborhood $\widehat U\subset\widehat X$ of $p$ such that $X'\cap\widehat U$ is closed in $\widehat U$ by \Cref{cor:pbchart}, and $U=\widehat U\cap X$ is a required open neigborhood.
Then the result immediately follows from \Cref{prop:whitney-zero}.
\end{proof}

\Cref{cor:submfd-defzero} means that every submanifold arise from a pullback of smooth maps.
On the other hand, a cospan in the category $\mathbf{Mfd}$ does not have pullbacks in general.
A typical case where it does is related to the transversality.
The following is a well-known result for manifolds without boundaries.

\begin{proposition}[e.g. see Theorem II.4.4 in \cite{GG73}]
\label{prop:transv-invim}
Let $F:X\to Y$ be a smooth map between manifolds without boundaries, and let $W\subset Y$ is a submanifold without boundaries of codimention $k$.
Then if $F\pitchfork W$, the subspace $F^{-1}(W)\subset X$ is a submanifold of codimension $k$.
\end{proposition}

If manifolds have corners, the situation becomes a little bit more complicated.
One can easily find counter examples of \Cref{prop:transv-invim} if we allow manifolds with corners.
For them, we only have weaker assertion, which is still enough for our purpose.

\begin{proposition}[cf. Lemma 6.3 in \cite{Mic80}]
\label{prop:transinv-cover}
Let $X$ and $Y$ be manidols with corners, and let $W\subset Y$ be a submanifold of codimension $k$ possibly with corners.
Let $F:X\to Y$ be a smooth map such that $F\pitchfork W$.
Suppose we have an embedding $X\hookrightarrow\widehat X$ into a manifold without boundary and of the same dimension as $X$.
Then, there is a countable family $\{N_i\}_{i=1}^\infty$ of codimension $k$ submanifolds of $\widehat X$ without boundaries such that $F^{-1}(W)\subset\bigcup_i N_i$ in $\widehat X$.
Moreover, for each $p\in X\cap N_i$, the induced map $F_\ast:T_pX\to T_pY$ maps $T_pN_i\subset T_p\widehat X\simeq T_pX$ into $T_{F(p)}W\subset T_{F(p)}Y$.
\end{proposition}

To end this section, we introduce a notion regarding how vector fields interact with submanifolds.

\begin{definition}
Let $X$ be a manifold with corners, and $N\subset X$ a submanifolds.
\begin{enumerate}[label={\rm(\arabic*)}]
  \item A vector field $\xi$ on $X$ is said to be along $N$, written $\xi\parallel N$, if for each $p\in N$, the vector $\xi(p)\in T_pX$ belongs to the image of $T_pN$.
  \item A vector field $\xi$ on $X$ is said to be transversal to $N$, written $\xi\pitchfork N$, if for each $p\in N$, the tangent space $T_pX$ is generated, as an $\mathbb R$-vector space, by $\xi(p)$ and $T_pN$.
\end{enumerate}
\end{definition}

It is easily verified that if $\xi$ is a vector field along $N$, then the one-parameter family associated to $\xi$ restricts to $N$.

\subsection{Jet bundles}
\label{sec:jet-bdl}

In this section, we review the notion of jets.
We mainly follow the literature \cite{GG73} and \cite{ScP09} with careful attention to corners (see also \cite{Mic80}).
Omitting some details, we refer the reader to them.

\begin{definition}
Let $X$ be a manifold with corners, and let $p\in X$.
Then we define an $\mathbb{R}$-algebra $J^r_p(X)$ by
\[
J^r_p(X) := C^\infty_p(X)/\mathfrak m_p(X)^{r+1}\,.
\]
For each $f\in C^\infty_p(X)$, we denote by $j^rf(p)$ its image in $J^r_p(X)$ and call it the $r$-th jet of $f$.
\end{definition}

\begin{lemma}
\label{lem:jet-germ}
Let $X$ be an manifold with corners, and let $p\in X$.
Then, every coordinate function $\varphi:U\to\mathbb H^I$ centered at $p$ for a marked finite set $I$ gives rise to an isomorphism
\[
J^r_p(X)
\xrightarrow{\substack{\widetilde\varphi_p\cr\sim}} P^r(I)
:= \mathbb R[x_i\mid i\in I]/(x_i\mid i\in I)^{r+1}
\]
of $\mathbb R$-algebras for each $r\ge 0$, which is given by the Taylor expansion.
\end{lemma}

We will often identify $P^r(I)$ with the vector space of polynomials over $I$ of degree at most $r$.
If $I\simeq\langle m|k\rangle$, $P^r(I)$ is, as an $\mathbb R$-vector space, of dimension
\[
\begin{pmatrix} m+r \\ m \end{pmatrix}\,.
\]
One can use \Cref{lem:jet-germ} to \emph{paste jets together} and give a canonical smooth structure on the set
\[
J^r(X) := \coprod_{p\in X} J^r_p(X)
\]
so that the map $J^r(X)\to X$ is an $\mathbb R$-algebra bundle over $X$, or an algebra object in the category $\mathbf{Vect}_{\mathbb R}(X)$.

More generally, for manifolds $X$ and $Y$ with corners, we define a smooth fiber bundle
\begin{equation}
\label{eq:rjet-bundle}
J^r(X,Y) := \mathpzc{Hom}_{\mathbb R\mathchar`-\mathrm{Alg}}(J^r(Y),J^r(X))
\end{equation}
over $X\times Y$.
Indeed, for finite sets $I$ and $J$, we denote by $P^r(I,J)_0$ the set of polynomial mappings $f:\mathbb R^I\to\mathbb R^J$ of degree at most $r$ with $f(0)=0$.
Then we have a canonical bijection
\[
\mathrm{Hom}_{\mathbb R\mathchar`-\mathrm{Alg}}(P^r(J),P^r(I))
\simeq P^r(I,J)_0\,.
\]
One can easily verify that $P^r(I,J)_0$ is diffeomorphic to the Euclidean space of dimension
\[
n\cdot\left(\begin{pmatrix} m+r \\ m \end{pmatrix}-1\right)\,.
\]
By \Cref{lem:jet-germ}, coordinates $\varphi:U\to\mathbb H^I$ on $X$ and $\psi:V\to\mathbb H^J$ on $Y$ gives rise to an isomorphism
\[
J^r(X,Y)_{(p,q)}
\simeq \mathrm{Hom}_{\mathbb R\mathchar`-\mathrm{Alg}}(P^r(J),P^r(I))\,.
\]
Thus, we obtain a fiber bundle $J^r(X,Y)$ over $X\times Y$ with fiber $P^r(I,J)_0$ for certain (marked) finite set $I$ and $J$.
Furthermore, it is also verified that the composition $J^r(X,Y)\twoheadrightarrow X\times Y\twoheadrightarrow X$ is also a locally trivial fibration with fiber $Y\times P^r(I,J)_0$.
Notice that we have an isomorphism
\[
J^r(X,\mathbb R) \simeq J^r(X)
\]
of fiber bundles over $X$.

We next define jets of smooth maps.
Let $F:X\to Y$ be a smooth map between manifolds with corners.
Then for each point $p\in X$ with $q=f(p)\in Y$, we have an induced homomorphism
\[
F^\ast_p:C^\infty_q(Y)\to C^\infty_p(X)
\]
of $\mathbb R$-algebras.
Since their residue fields are $\mathbb R$, $F^\ast_p$ preserves the maximal ideals, and it induces a homomorphism
\[
j^r F(p):J^r_q(Y) \to J^r_p(X)
\]
of $\mathbb R$-algebras for each non-negative integer $r$.
In other words, we can assign each $p\in X$ to a homomorphism $j^rF(p)$, which gives rise to a map
\[
j^rF: X\ni p \mapsto (p,F(p),j^rF(p))\in X\times Y\times J^r(X,Y)_{(p,q)}\hookrightarrow J^r(X,Y)\,.
\]
For the smoothness, we have the followinng result.

\begin{lemma}
\label{lem:jet-smooth}
Let $X$, $B$, and $Y$ be manifolds with corners, and suppose we have a smooth map $F:X\times B\to Y$.
Then for each non-negative integer $r\ge 0$, the map
\[
X\times B \ni (p,b) \mapsto j^r F_b(p) \in J^r(X,Y)
\]
is smooth, where we write $F_b(p) = F(b,p)$.
\end{lemma}

In particular, taking $B$ to be a point in \Cref{lem:jet-smooth}, we obtain a well-defined map
\[
j^r:C^\infty(X,Y) \mapsto C^\infty(X,J^r(X,Y))\,.
\]
We see in the next section that $j^r$ ``induces'' a topology on $C^\infty(X,Y)$ so that the map is continuous.

The construction of jet bundles, moreover, gives rise to a functor.
Indeed, for a smooth map $F:X\to Y$, and for each $p\in X$, we have an $\mathbb R$-algebra homomorphism
\[
F^\ast:J^r_{F(p)}(Y) \to J^r_p(X)\,.
\]
Hence, if $W$ is another manifold, we obtain an $\mathbb R$-linear map
\[
\mathrm{Hom}_{\mathbb R\mathchar`-\mathrm{Alg}}(J^r_p(X),J^r_w(W))
\to \mathrm{Hom}_{\mathbb R\mathchar`-\mathrm{Alg}}(J^r_{F(p)}(Y),J^r_w(W))
\]
for each $w\in W$.
This defines a well-define map
\begin{equation}
\label{eq:jet-functor}
F_\ast:J^r(W,X)\to J^r(W,Y)\,.
\end{equation}

\begin{lemma}
\label{lem:jet-functor}
Let $F:X\to Y$ be a smooth map between manifolds with corners.
Then, for every manifold $W$ with corners, the map \eqref{eq:jet-functor} is smooth.
Moreover, $J^r(W,\blankdot)$ gives rise to a functor $\mathbf{Mfd}\to\mathbf{Mfd}$ which preserves embeddings of manifolds.
\end{lemma}

\begin{lemma}
\label{lem:jet-proj}
Let $X$, $Y$, and $Z$ be manifolds with corners.
Then, for each non-negative integer $r\ge 0$, the projections $Y\times Z\to Y,Z$ induce the following pullback square in the category $\mathbf{Mfd}$:
\[
\xymatrix{
  J^r(X,Y\times Z) \ar[r] \ar[d] \ar@{}[dr]|(.4)\pbcorner & J^r(X,Y) \ar[d] \\
  J^r(X,Z) \ar[r] & X }
\]
\end{lemma}
\begin{proof}
It is easily verified that the problem reduces to the local one.
In particular, we may assume $X=\mathbb H^I$, $Y=\mathbb H^J$, and $Z=\mathbb H^K$.
Then, we have
\[
\begin{gathered}
J^r(X,Y\times Z) \simeq \mathbb H^I\times \mathbb H^J\times \mathbb H^K\times P^r(I,J\amalg K)_0 \\
J^r(X,Y) \simeq \mathbb H^I\times \mathbb H^J\times P^r(I,J)_0 \\
J^r(X,Z) \simeq \mathbb H^I\times \mathbb H^K\times P^r(I,K)_0 \ .
\end{gathered}
\]
Hence the result is now obvious.
\end{proof}

We saw, as above, the jet bundle $J^r(\blankdot,\blankdot)$ has some good functorial properties on the second variable.
On the other hand, we do not have the dual results on the first variable in genral.
A smooth map $X\to W$ might not even induce a smooth map $J^r(W,Y)\to J^r(X,Y)$.
We have a base-change result instead.
Recall that we have a locally trivial fibration $J^r(W,Y)\to W$ so that we obtain a fiber bundle
\[
X\times_W J^r(W,Y) \to X
\]
over $X$ with the same fiber as $J^r(W,Y)$.

\begin{lemma}
\label{lem:jet-basechange}
Let $W$, $X$, and $Y$ be manifolds with corners, and let $F:X\to W$ be a smooth map.
Then the precomposition with $F$ gives rise to a smooth bundle map
\begin{equation}
\label{eq:jet-basechange:F}
F^\ast:X\times_W J^r(W,Y) \to J^r(X,Y)
\end{equation}
over $X\times Y$.
Moreover, if $F$ is a submersion, then the induced map $F^\ast$ is an embedding.
\end{lemma}
\begin{proof}
For each $p\in X$ with $w=F(p)$, the precomposition with $F$ gives rise to an $\mathbb R$-algebra homomrophism $J^r_w(W)\to J^r_p(X)$.
Then, the map \eqref{eq:jet-basechange:F} is fiberwisely described as the map
\begin{equation}
\label{eq:prf:jet-basechange:homo}
\mathrm{Hom}_{\mathbb R\mathchar`-\mathrm{Alg}}(J^r_q(Y),J^r_w(W))
\to \mathrm{Hom}_{\mathbb R\mathchar`-\mathrm{Alg}}(J^r_q(Y),J^r_p(X))\,.
\end{equation}
Thus, the first assertion is obvious.
To verify the last, it suffices to see if $F$ is a submersion, the homomorphism \eqref{eq:prf:jet-basechange:homo} is a monomorphism for each $p\in X$ and $q\in Y$, which is not difficult.
\end{proof}

\begin{corollary}
\label{cor:jet-prod}
Let $\{X_i\}_{i\in I}$ and $\{Y_i\}_{i\in I}$ be families of manifolds with corners indexed by a finite set $I$.
Then the product of maps gives rise to a smooth map
\[
\prod_{i\in I} J^r(X_i,Y_i) \to J^r(\prod_{i\in I} X_i,\prod_{i\in I} Y_i )\,.
\]
\end{corollary}
\begin{proof}
Using the induction on the cardinality of $I$, one can notice that we only have to consider the case $I=\{1,2\}$:
\begin{equation}
\label{eq:jet-prod:bipr}
J^r(X_1,Y_1)\times J^r(X_2,Y_2) \to J^r(X_1\times X_2,Y_1\times Y_2)
\end{equation}
By \Cref{lem:jet-proj}, we have a natural diffeomorphism
\[
J^r(X_1\times X_2,Y_1\times Y_2)
\simeq J^r(X_1\times X_2,Y_1)\times_{X_1\times X_2} J^r(X_1\times X_2,Y_2)\,.
\]
On the other hand, we have a canonical diffeomorphism
\[
\begin{multlined}
J^r(X_1,Y_1)\times J^r(X_2,Y_2) \\
\simeq
\left( (X_1\times X_2)\times_{X_1}J^r(X_1,Y_1) \right)
\times_{X_1\times X_2} \left( (X_1\times X_2)\times_{X_2} J^r(X_2,Y_2) \right)\,.
\end{multlined}
\]
Under these diffeomorphisms, the map \eqref{eq:jet-prod:bipr} is induced by the map
\[
(X_1\times X_2)\times_{X_i}J^r(X_i,Y_i)
\to J^r(X_1\times X_2,Y_i)
\]
for $i=1,2$, which is smooth by \Cref{lem:jet-basechange}.
Thus, the result follows from \Cref{lem:jet-proj}.
\end{proof}

In the case $r=1$, the first jet bundle is strongly related to the (co)tangent bundle.

\begin{proposition}
\label{prop:tangent-J1}
For manifolds $X$ and $Y$ with corners, there are natural isomorphisms
\[
J^1(X,Y)
\cong \mathpzc{Hom}_{\mathbb R}(T^\ast Y,T^\ast X)
\cong \mathpzc{Hom}_{\mathbb R}(TX,TY)
\]
of bundles over $X\times Y$, which sends $j^1_pF\in J^1(X,Y)$ to $d_pF:T_pX\to T_{F(p)}Y$ for $F\in C^\infty(X,Y)$.
In particular, in the case $Y=\mathbb R$, we have a canonical isomorphism
\[
J^1(X) \cong \mathbb R_X\otimes T^\ast X
\]
of real vector bundles over $X$.
\end{proposition}

For each $f\in C^\infty(X)$, we have a smooth map
\[
X\xrightarrow{j^1f} J^1(X) \twoheadrightarrow T^\ast X\,,
\]
which we denote by $df$.

\subsection{Whitney $C^\infty$-topology on the space of smooth maps}
\label{sec:whitney-top}

Finally, we will topologize the set $C^\infty(X,Y)$.
Let $X$ and $Y$ be smooth manifolds with corners.
For a natural number $k\in\mathbb N$ and a subset subset $A\subset J^k(X,Y)$, we define a subset $M(A)\subset C^\infty(X,Y)$ by
\[
M(A) := \{ F\in C^\infty(X,Y) \mid j^kF(X)\subset A \}\,.
\]
Note that for every $A_1,A_2\in J^k(X,Y)$, we have
\begin{equation}
\label{eq:M-cap}
M(A_1\cap A_2) = M(A_1)\cap M(A_2)\,.
\end{equation}
It follows that for a fixed $k\in\mathbb N$, the subsets $M(U)$ for all open subsets $U\subset J^k(X,Y)$ form a basis for a topology on $C^\infty(X,Y)$, which we call the Whitney $C^k$-topology.

\begin{definition}
The Whitney $C^\infty$-topology on $C^\infty(X,Y)$ is the topology generated by all open subsets in the Whitney $C^k$-topology for all $k\in\mathbb N$.
\end{definition}

There is the smooth projection $\pi^l_k:J^l(X,Y)\to J^k(X,Y)$ for $k<l$, so that if $U\subset J^k(X,Y)$ and $V\subset J^l(X,Y)$, then $(\pi^k_l)^{-1}(U)\subset J^l(X,Y)$ is an open subset.
Since $j^k = \pi^k_l j^l$, we have
\[
M((\pi^k_l)^{-1}(U)) = M(U)\,.
\]
Hence, by \eqref{eq:M-cap}, we obtain that all subsets of $C^\infty(X,Y)$ which is open in the Whitney $C^k$-topology for some $k\in\mathbb N$ form a basis for the Whitney $C^\infty$-topology.

In what follows, we always suppose that $C^\infty(X,Y)$ is equipped with the Whitney $C^\infty$-topology unless otherwise noted.

We here review some important properties of the Whitney $C^\infty$-topology following \cite{GG73}.

\begin{lemma}
\label{lem:whitney-metric}
Let $X$ and $Y$ be manifolds with corners.
Choose a metric $d_s$ for each $J^s(X,Y)$.
For $k\in\mathbb N$, we define a map $\rho_k:C^\infty(X,Y)\times C^\infty(X,Y)\to\mathbb R$ by the formula
\[
\rho_k(F,G)
:= \sup_{p\in X} \frac{d_k(j^kF(p),j^kG(p))}{1+d_k(j^kF(p),j^kG(p))}\,.
\]
Then the following hold:
\begin{enumerate}[label={\rm(\arabic*)}]
  \item\label{sublem:whit-metric} $\rho_k$ is a metric on the set $C^\infty(X,Y)$, which is not necessarily compatible with the Whitney topologies.
  \item\label{sublem:whit-contin} $\rho_k$ is continuous with respect to the Whitney $C^k$-topology; equivalently, for every $F\in C^\infty(X,Y)$ and $\varepsilon>0$, the set
\[
B_k(F;\varepsilon):=\left\{G\in C^\infty(X,Y)\mid\rho_k(F,G)<\varepsilon\right\}
\]
is open in the Whitney $C^k$-topology.
  \item\label{sublem:whit-cauchy} Suppose $\{F_i\}\subset C^\infty(X,Y)$ is a sequence which is a Cauchy sequence with respect to $\rho_k$ for each $k\in\mathbb N$.
Then there is a smooth function $G:X\to Y$ such that for each $p\in X$ and $k\in\mathbb N$, we have
\[
j^kG(p) = \lim_{i\to\infty} j^kF_i(p)\in J^k(X,Y)\,.
\]
\end{enumerate}
\end{lemma}
\begin{proof}
See the discussion in the first part of Section II-$\S 3$ of \cite{GG73}.
\end{proof}

\begin{remark}
The topology induced by the metric $\rho_k$ does not coincide with the Whitney $C^k$-topology in general.
Actually, we have a good basis for the Whitney $C^k$-topology; for a smooth function $\delta:X\to (0,\infty)$, we set
\[
B(F;\delta)
:= \left\{ G\in C^\infty(X,Y)\mid\forall p\in X : d_k(j^kF(p),j^kG(p)) < \delta(p)\right\}\,.
\]
Then the family $\{B(F;\delta)\mid\delta:X\to(0,\infty)\}$ is a neighborhood basis around $F\in C^\infty(X,Y)$ in the Whitney $C^k$-topology.
Hence, if $X$ is compact, the topology induced by $\rho_k$ is precisely the Whitney $C^k$-topology.
\end{remark}

Finally, we see that the space $C^\infty(X,Y)$ satisfies some functorial properties.

\begin{proposition}
\label{prop:whit-comp}
Let $X$, $Y$, and $W$ be manifolds with corners.
\begin{enumerate}[label={\rm(\arabic*)}]
  \item\label{subprop:whit-comp:post} If $F:X\to Y$ is a smooth map, then the induced map
\[
F_\ast:C^\infty(W,X)\to C^\infty(W,Y)
\]
is continuous.
Moreover, if $F$ is an embedding, then $F_\ast$ is a topological embedding.
  \item\label{subprop:whit-comp:pre} If $F:X\to Y$ is a smooth proper map, then the induced map
\[
F^\ast:C^\infty(Y,W) \to C^\infty(X,W)
\]
is continuous.
\end{enumerate}
\end{proposition}
\begin{proof}
\ref{subprop:whit-comp:post}:
For every open subset $V\subset J^r(W,Y)$, by \Cref{lem:jet-functor}, we have an open subset $F_\ast^{-1}(V)\subset J^r(W,X)$.
We also have
\begin{equation}
\label{eq:prf:whit-comp:fMV}
F_\ast^{-1}(M(V)) = M(F_\ast^{-1}(V)) \subset C^\infty(W,X)\,,
\end{equation}
which directly implies the first result.
Moreover, if $F$ is an embedding, by \Cref{lem:jet-functor}, so is the smooth map $F_\ast:J^r(W,X)\to J^r(W,Y)$.
Hence, for every open subset $U\subset J^r(W,X)$, there is an open subset $V_U\subset J^r(W,Y)$ such that $F_\ast^{-1}(V_U)=U\subset J^r(W,X)$.
Then, by \eqref{eq:prf:whit-comp:fMV}, we obtain
\[
F_\ast^{-1}(M(V_U)) = M(F_\ast^{-1}(V_U)) = M(U) \subset C^\infty(W,X)\,.
\]
This implies that every basic open subset of $C^\infty(W,X)$ can be written as the inverse image of an open subset of $C^\infty(W,Y)$.
Thus, $F_\ast:C^\infty(W,X)\to C^\infty(W,Y)$ is a topological embedding.

\ref{subprop:whit-comp:pre}:
See Proposition II.3.9 of \cite{GG73}.
\end{proof}

\begin{proposition}
\label{prop:Cinf-prod-contin}
Let $\{X_i\}_{i\in I}$ and $\{Y_i\}_{i\in I}$ be two families of manifolds with corners indexed by the same finite set $I$.
Then the map
\begin{equation}
\label{eq:prop:Cinf-prod-contin}
\prod_{i\in I} C^\infty(X_i,Y_i) \to C^\infty(\prod_{i\in I} X_i,\prod_{i\in I} Y_i)\,;\quad (F_i)_i \mapsto \prod F_i
\end{equation}
is continuous with respect to the Whitney $C^\infty$-topology.
\end{proposition}
\begin{proof}
Suppose we have a smooth map $F_i:X_i\to Y_i$ for each $i\in I$.
Then we have a commutative diagram
\[
\xymatrix@C=1em{
  & \prod_{i\in I} X_i \ar[dl]_{\prod j^rF_i} \ar[dr]^{j^r(\prod F_i)} & \\
  \prod_{i\in I} J^r(X_i,Y_i) \ar[rr] && J^r(\prod_{i\in I} X_i, \prod_{i\in I} Y_i) }
\]
for each non-negative integer $r$, where the bottom arrow is smooth by \Cref{cor:jet-prod}.
Thus, for every neighborhood $\mathcal U$ of $j^r(\prod F_i)(\prod X_i)$, one may find an open nighborhood of $\prod j^rF_i(X_i)$ which is mapped into $U$.
This implies that the map \eqref{eq:prop:Cinf-prod-contin} is continuous.
\end{proof}

\begin{corollary}
\label{cor:Cinf-cartprod}
Let $X$, $Y$ and $W$ be three manifolds with corners.
Then, the projections $X\times Y\to X,Y$ induce the homeomorphism
\[
C^\infty(W,X\times Y)
\xrightarrow\simeq C^\infty(W,X)\times C^\infty(W,Y)\,.
\]
\end{corollary}
\begin{proof}
The map is continuous thanks to \ref{subprop:whit-comp:pre} in \Cref{prop:whit-comp}.
Its inverse is clearly given by the composition
\[
C^\infty(W,X)\times C^\infty(W,Y)
\xrightarrow{\times} C^\infty(W\times W,X\times Y)
\xrightarrow{\Delta_W} C^\infty(W,X\times Y)\,,
\]
where $\Delta_W:W\to W\times W$ is the diagonal map.
The first map is continuous by \Cref{prop:Cinf-prod-contin} while so is the second by \ref{subprop:whit-comp:post} in \Cref{prop:whit-comp}.
Hence, we obtain the result.
\end{proof}

\section{Multi-relative settings}
\label{sec:mult-relative}

In this section, we will make clear our relative setting.
This setting provides us a more precise observation for submanifolds than when we naively consider pairs $(M,L)$ of manifolds and submanifolds.

\subsection{Arrangements of submanifolds}
\label{sec:conf-submfd}

We first fix some elementary notations.
We denote by $\mathbf{Emb}$ the category of (smooth) manifolds with corners and smooth embeddings.

\begin{definition}
A pre-arrangement of manifolds is a pair $(S,\mathcal X)$ of a finite poset $S$ with maximum and a functor $\mathcal X:S\to\mathbf{Emb}$.
By abuse of notation, we often denote $(S,\mathcal X)$ by $\mathcal X$ simply.
We call $S$ the shape of $\mathcal X$.
We will write
\[
|\mathcal X| := \mathcal X(\max S)\,,
\]
and call it the ambient manifold of $\mathcal X$.
A pre-arrangement $\mathcal X$ is said to be closed if its image consists of closed embeddings; i.e. for each $s\le t\in S$, the embedding $\mathcal X(s)\to\mathcal X(t)$ is a closed embedding.
\end{definition}

If $\mathcal X$ is a pre-arrangement of manifolds, the functor $\mathcal X:S\to\mathbf{Emb}$ factors through $\mathcal X:S\to\mathbf{Emb}/|\mathcal X|$, the slice category over the ambient manifold.
Note that we can naturally see $\mathbf{Emb}/|\mathcal X|$ as the partially ordered set of submanifolds of the ambient manifold.
In this point of view, a pre-arrangement $\mathcal X$ can be regard as a $S$-indexed family $\{\mathcal X(s)\}_{s\in S}$ of (embedded) submanifolds of $X=|\mathcal X|$ such that $\mathcal X(1)=X$ and $\mathcal X(s)\subset\mathcal X(t)$ whenever $s\le t$ in $S$.

We can also consider relative maps:

\begin{definition}
For pre-arrangements $\mathcal X,\mathcal Y$ of manifolds of the same shape $S$, an $S$-map $\mathcal X\to\mathcal Y$ is a smooth map $|\mathcal X|\to |\mathcal Y|$ between the ambient manifolds which restricts to $\mathcal X(s)\to \mathcal Y(s)$ for each $s\in S$.
We denote by $C^\infty(\mathcal X,\mathcal Y)$ the subset of $C^\infty(|\mathcal X|,|\mathcal Y|)$ consisting of $S$-maps.
\end{definition}

\begin{example}
Take $S=\mathbf{pt}$ to be a one point set.
In this case, a pre-arrangement of shape $\mathbf{pt}$ is nothing but a manifold $X$ with corners, and a $\mathbf{pt}$-map is nothing but a smooth map.
\end{example}

\begin{example}
\label{ex:linear-arr}
Take $S=[n] := \{0\le 1\le 2\le\dots\le n\}$.
Then a pre-arrangement of shape $[n]$ can be seen as a sequence of submanifolds
\[
X_0\hookrightarrow X_2\hookrightarrow\dots\hookrightarrow X_n\,.
\]
If $\mathcal X,\mathcal Y:[n]\to\mathbf{Emb}$ are pre-arrangement, in the classical notation, we have
\[
C^\infty(\mathcal X,\mathcal Y) = \bigcap_{i=0}^{n-1} C^\infty(X_n,X_i;Y_n,Y_i)\subset C^\infty(X,Y)\,.
\]
The case $S=[1]$ is discussed in \cite{Ish98}.
\end{example}

\begin{example}
Let $S$ be a finite poset with maximum, and let $X$ be a manifold with corners.
Then there is a pre-arrangement of shape $S$ consisting only of $X$, which we also denote by $X$.
For any pre-arrangement $\mathcal W$ of shape $S$, we have
\[
C^\infty(\mathcal W, X) = C^\infty(|\mathcal W|,X)\,.
\]
\end{example}

\begin{example}
\label{ex:arr:restrict}
Let $\mathcal X$ be a pre-arrangement of shape $S$, and suppose we have a order-preserving map $\mu:T\to S$ from a poset $T$ with maximum.
Then we can define a pre-arrangement of shape $T$ by
\[
T \xrightarrow{\mu} S \xrightarrow{\mathcal X} \mathbf{Emb}\,,
\]
which we will denote by $\mathcal X_\mu$.
\end{example}

\begin{example}
\label{ex:arr:product}
Let $\mathcal X$ and $\mathcal Y$ be two pre-arrangements of manifolds of shape $S$ and $T$ respectively.
Then we can define their product arrangement $\mathcal X\times\mathcal Y$ to be a pre-arrangement of shape $S\times T$ as
\[
(\mathcal X\times\mathcal Y)(s,t) := \mathcal X(s)\times\mathcal Y(t)\,.
\]
\end{example}

We now have a notion which enables us to describe relative situations.
This notion, however, might cause some problematic situation; for example, two submanifolds $\mathcal X(s)$ and $\mathcal X(t)$ in a pre-arrangement might intersect with one another in a very complicated subset of $|\mathcal X|$.
To avoid these situations, we also need to control intersections of members of pre-arrangements, which can be done in very simple way as follows:

\begin{definition}
A pre-arrangement $\mathcal X$ is called an arrangement if its shape $S$ is a lattice, and the functor $\mathcal X:S\to\mathbf{Emb}$ preserves meets (or pullbacks); i.e. for each $s,t\in S$, we have
\begin{equation}
\label{eq:cond-complete}
\mathcal X(s\wedge t) = \mathcal X(s)\cap\mathcal X(t) \subset |\mathcal X|\,.
\end{equation}
\end{definition}

Note that the equation \eqref{eq:cond-complete} holds if and only if the square
\[
\xymatrix{
  \mathcal X(s\wedge t) \ar[r] \ar[d] \ar@{}[dr]|(.4){\pbcorner} & \mathcal X(s) \ar[d] \\
  \mathcal X(t) \ar[r] & |\mathcal X| }
\]
is a pullback square in the category $\mathbf{Emb}$.

As is often the case, it is a tedious work to describe arrangements completely.
Actually, we can obtain it as a completion of some sort of pre-arrangements.
For a finite poset $S$ with the maximum element, we denote by $S^\vee$ the set of non-empty upper subsets of $S$.
The set $S^\vee$ is ordered by the inclusions, and it is easily checked that $S^\vee$ is a finite lattice.
Moreover, there is a canonical embedding $S\hookrightarrow S^\vee$ of posets.

\begin{definition}
Let $\mathcal X:S\to\mathbf{Emb}$ be a pre-arrangement of manifolds.
Then the right Kan extension $\mathcal X^\vee:S^\vee\to\mathbf{Emb}$ of $\mathcal X$ along $S\hookrightarrow S^\vee$ is, if exists, called the completion of $\mathcal X$.
\end{definition}

The following is an obvious result:

\begin{lemma}
Suppose we have a completion $\mathcal X^\vee$ of a pre-arrangement $\mathcal X$ of manifolds.
Then $\mathcal X^\vee$ is an arrangement of shape $S^\vee$, and we have $|\mathcal X^\vee|=|\mathcal X|$.
Moreover, if we denote by $\iota:S\hookrightarrow S^\vee$ the embedding, we have a canonical isomorphism $\eta^\ast\mathcal X^\vee \simeq \mathcal X$.
\end{lemma}

Not all pre-arrangements have completions because the category $\mathbf{Emb}$ does not have all limits.
However, if the completion exists, we can describe it concretely as follows:
Let $\mathcal X$ be a pre-arrangement of shape $S$, and suppose we have the completion $\mathcal X^\vee$.
Then for each $F\in S^\vee$, we have
\begin{equation}
\label{eq:completion-arrangement}
\mathcal X^\vee(F) = \bigcap_{s\in F} \mathcal X(s)\,.
\end{equation}
This observation tells us that a pre-arrangement $\mathcal X$ has the completion if and only if the right-hand side of \eqref{eq:completion-arrangement} is again a submanifold of $|\mathcal X|$.

\subsection{Whitney $C^\infty$-topology on the space of relative maps}
\label{sec:baire-relative}

If $\mathcal X$ and $\mathcal Y$ are pre-arrangements of shape $S$, then the set $C^\infty(\mathcal X,\mathcal Y)$ of maps of arrangements is a subset of $C^\infty(|\mathcal X|,|\mathcal Y|)$ so that we can topologize $C^\infty(\mathcal X,\mathcal Y)$ with the subspace topology.
More precisely, we have a map
\[
C^\infty(\mathcal X,\mathcal Y)
\hookrightarrow C^\infty(|\mathcal X|,|\mathcal Y|)
\xrightarrow{j^r} C^\infty(|\mathcal X|,J^r(|\mathcal X|,|\mathcal Y|))\,,
\]
which we still denote by $j^r$.
For each subset $A\subset J^r(|\mathcal X|,|\mathcal Y|)$, by abuse of notation, we will also write
\[
M(A) := \{f\in C^\infty(\mathcal X,\mathcal Y)\mid j^r(|\mathcal X|)\subset A\}\,.
\]
Then the Whitney $C^\infty$-topology on $C^\infty(\mathcal X,\mathcal Y)$ is generated by subsets $M(U)\subset C^\infty(\mathcal X,\mathcal Y)$ for open subsets $U\subset J^r(|\mathcal X|,|\mathcal Y|)$ and for $r\in\mathbb N$.

\begin{lemma}
\label{lem:arr-whit-submfd}
Let $\mathcal X$ and $\mathcal Y$ be closed pre-arrangements of manifolds of shape $S$.
Then, for each $s\in S$, the restriction map
\[
\mathrm{res}_s:C^\infty(\mathcal X,\mathcal Y)
\to C^\infty(\mathcal X(s),\mathcal Y(s))
\]
is continuous (with respect to the Whitney $C^\infty$-topology).
\end{lemma}
\begin{proof}
To prove the result, it suffices to show that for every open subset $U'_s\subset J^r(\mathcal X(s),\mathcal Y(s))$, we can choose an open subset $U\subset J^r(|\mathcal X|,|\mathcal Y|)$ such that
\[
\mathrm{res}_s^{-1}(M(U'_s)) = M(U)\,.
\]
First notice that, for each $r\ge 0$, we have the following commutative square:
\[
\xymatrix{
  J^r(\mathcal X(s),\mathcal Y(s)) \ar[r] \ar@{->>}[d] & J^r(X,Y) \ar@{->>}[d] \\
  \mathcal X(s)\times\mathcal Y(s) \ar[r] & |\mathcal X|\times |\mathcal Y| }
\]
Hence, we have an embedding $\lambda:J^r(\mathcal X(s),\mathcal Y(s))\to J^r(|\mathcal X|,|\mathcal Y|)|_{\mathcal X(s)\times\mathcal Y(s)}$ of fiber bundles over $\mathcal X(s)\times \mathcal Y(s)$.
It follows that we can choose an open subset $U_s\subset J^r(|\mathcal X|,|\mathcal Y|)|_{\mathcal X(s)\times\mathcal Y(s)}$ such that $U'_s = \lambda^{-1}(U_s)$.
We put
\[
U = U_s\cup J^r(|\mathcal X|,|\mathcal Y|)|_{(|\mathcal X|\times |\mathcal Y|)\setminus(\mathcal X(s)\times\mathcal Y(s))}
\]
which is an open subset of $J^r(|\mathcal X|,|\mathcal Y|)$ since $\mathcal X$ and $\mathcal Y$ are closed arrangements.
Then, for $f\in C^\infty(\mathcal X,\mathcal Y)$, we have
\[
\begin{split}
j^r(f)(|\mathcal X|) \subset U
&\iff \forall p\in |\mathcal X|: j^r(f)(p) \in U \\
&\iff \forall p\in \mathcal X(s): j^r(f)(p) \in U_s \\
&\iff j^r(f)(\mathcal X(s))\subset U_s\,.
\end{split}
\]
Moreover, since $f(\mathcal X(s))\subset \mathcal Y(s)$, we always identify $j^r(f)(\mathcal X(s))$ with $j^r(\mathrm{res}_s f)(\mathcal X(s))\subset J^r(\mathcal X(s),\mathcal Y(s))$ via the inclusion $J^r(\mathcal X(s),\mathcal Y(s))\to J^r(|\mathcal X|,|\mathcal Y|)$.
These imply that we have
\[
M(U) = \mathrm{res}_s^{-1}(M(U'_s))\,,
\]
which is the required result.
\end{proof}

\begin{corollary}
\label{cor:arr-res}
Let $\mathcal X$ and $\mathcal Y$ be closed pre-arrangements of manifolds of shape $S$.
Suppose we have an order-preserving map $\mu:T\to S$ from a poset $T$ with maximum.
Then the canonical map
\[
\mathrm{res}_\mu:C^\infty(\mathcal X,\mathcal Y) \to C^\infty(\mathcal X_\mu,\mathcal Y_\mu)
\]
is continuous (see \Cref{ex:arr:restrict}).
\end{corollary}
\begin{proof}
Let us denote by $t\in S$ the image of the maximum element of $T$ under $\mu:T\to S$.
Since the map
\[
C^\infty(\mathcal X_\mu,\mathcal Y_\mu)
\hookrightarrow C^\infty(|\mathcal X_\mu|,|\mathcal Y_\mu|)
= C^\infty(\mathcal X(t),\mathcal Y(t))
\]
is a topological embedding, to see $\mathrm{res}_\mu$ is continuous, it suffices to show that
\[
C^\infty(\mathcal X,\mathcal Y)
\to C^\infty(\mathcal X(t),\mathcal Y(t))
\]
is continous, which is already proved in \Cref{lem:arr-whit-submfd}.
\end{proof}

\begin{corollary}
\label{cor:submfd-jet}
Let $\mathcal X$ and $Y$ be closed pre-arrangements of manifolds of shape $S$.
Then for each $s\in S$ and $r\ge 0$, the map
\[
C^\infty(\mathcal X,\mathcal Y)\ni f
\mapsto j^r(f|_{\mathcal X(s)}) \in C^\infty(\mathcal X(s),J^r(\mathcal X(s),\mathcal Y(s)))
\]
is continuous.
\end{corollary}
\begin{proof}
The result follows from \Cref{lem:arr-whit-submfd} and the fact that the map
\[
j^r:C^\infty(\mathcal X(s),\mathcal Y(s))\to C^\infty(\mathcal X(s),J^r(\mathcal X(s),\mathcal Y(s)))
\]
is continuous (e.g. see Proposition II.3.4 in \cite{GG73}).
\end{proof}

In the later sections, we will discuss density of smooth functions.
To do this, the following notion is very essential:

\begin{definition}
Let $X$ be an arbitrary topological space.
\begin{enumerate}[label={\rm(\arabic*)}]
  \item A subset $T\subset X$ is called residual if $T$ is written as a countable intersection of open dense subsets of $X$.
  \item $X$ is called a Baire space if every residual subset is dense in $X$.
\end{enumerate}
\end{definition}

Similarly to the case of the space of usual $C^\infty$-maps (e.g. see \cite{GG73}), we have the following result.

\begin{proposition}
\label{prop:Cinf-baire}
Let $\mathcal X$ and $\mathcal Y$ be closed pre-arrangements of manifolds of shape $S$.
Then the space $C^\infty(\mathcal X,\mathcal Y)$ is a Baire space.
\end{proposition}
\begin{proof}
We have to show that every countable intersection of open dense subsets is dense in $C^\infty(\mathcal X,\mathcal Y)$, so suppose  we have a countable family $\{G_n\}_{n\in\mathbb N}$ of open dense subsets in $C^\infty(\mathcal X,\mathcal Y)$, and let $U\subset C^\infty(\mathcal X,\mathcal Y)$ be an arbitrary non-empty open subset.
We show
\begin{equation}
\label{eq:smooth-baire}
U\cap\bigcap_{n=0}^\infty G_n\neq\varnothing\,.
\end{equation}
Say $X$ and $Y$ are the ambient manifolds of $\mathcal X$ and $\mathcal Y$ respectively.
Notice that, since the space $J^r(X,Y)$ is a locally compact normal space, by the definition of $C^\infty$-topology, we can choose a natural number $k_0\in\mathbb N$ and an open subset $W_0\subset J^{k_0}(X,Y)$ so that $\varnothing\neq M(\overline W_0)\subset U$.
In particular, since $G_0$ is open and dense, we may assume $M(\overline{W}_0)\subset U\cap G_0$.
Thus, in order to see \eqref{eq:smooth-baire}, it suffices to show that there is an element
\begin{equation}
\label{eq:smooth-baire-cl}
g\in M(\overline{W}_0)\cap \bigcap_{n=0}^\infty G_n\,.
\end{equation}

Choose a metric $d_k$ on each $J^k(X,Y)$, and define $\rho_k$ as in \Cref{lem:whitney-metric}.
First, we construct inductively a sequence $\{(k_i,W_i,f_i)\}_{i=0}^\infty$ of triples of a natural number $k_i\in\mathbb{N}$, an open subset $W_i\subset J^{k_i}(X,Y)$, and a smooth function $f_i\in M(W_i)$ such that
\begin{itemize}
  \item $k_0\le k_1\le\dots\le k_i\le\cdots$ ;
  \item $U\supset M(W_0)\supset M(W_1)\supset\dots\supset M(W_i)\supset\cdots$ ;
  \item $M(\overline{W}_i)\subset G_i$ for each $i\in\mathbb{N}$ ;
  \item $\rho_l(f_{i-1},f_i) < 1/2^{i}$ for each $i\in\mathbb{N}$ and each $0\le l\le i$.
\end{itemize}
Suppose we have already defined such a sequence for $i\le n$.
In this case, by \ref{sublem:whit-contin} in \Cref{lem:whitney-metric} and the induction hypothesis, we have the following non-empty open subset
\[
M(W_n)\cap G_{n+1}\cap \bigcap_{l=0}^{n+1}B_l(f_n;\frac1{2^{n+1}})\,.
\]
Take $f_{n+1}$ to be an element of this, then we can also take a sufficient large number $k_{n+1}$ and an open subset $W_{n+1}\subset J^{k_{n+1}}(X,Y)$ which enjoy the required conditions.
By induction, we finally obtain a required sequence $\{(k_i,W_i,f_i)\}_{i=0}^\infty$.

Now, by construction above, the sequence $\{f_i\}$ is a Cauchy sequence in $C^\infty(\mathcal X,\mathcal Y)$ with respect to $\rho_i$ for every $i\in\mathbb N$.
Hence, by \ref{sublem:whit-cauchy} in \Cref{lem:whitney-metric}, there is a smooth function $g:X\to Y$ such that for each $x\in X$ and $k\in\mathbb{N}$, we have
\[
j^kg(x) = \lim_{i\to\infty} j^kf_i(x)
\]
in $J^k(X,Y)$.
In particular, since $f_i(\mathcal X(s))\subset \mathcal Y(s)$ for each $i$ and $\mathcal Y(s)$ is a closed subset of $Y$, we can also verify $g\in C^\infty(\mathcal X,\mathcal Y)$.
Moreover, since we have $\{j^{k_n}f_i(x)\}_{i\ge n}\subset W_n$, we obtain $j^{k_n}g(x)\in \overline{W}_n$.
It follows that $g\in\bigcap_i M(\overline W_i)$.
On the other hand, since we have $M(\overline W_i)\subset G_i$, we also have $\bigcap_i M(\overline W_i)\subset \bigcap_i G_i$.
Thus, we obtain \eqref{eq:smooth-baire-cl}, which completes the proof.
\end{proof}

\subsection{Excellent arrangements}
\label{sec:excellent-arr}

For a study of manifolds, it is very essential to take good coordinates on manifolds.
This is also true when we consider arrangements of submanifolds.
In this section, we discuss coordinates on arrangements.

First, we discuss a local model for ``good'' coordinates.
Indeed, we consider canonical arrangements on Euclidean spaces near the origin.
Because of the existence of corners, we need some technical preliminaries.

\begin{definition}
An arrangement of marked finite sets is a pair $(S,\mathcal I)$ of a finite lattice $S$ and a functor $\mathcal I:S\to\mathbf{MkFin}_{\mathrm{inj}}$ which preserves meets (or pullbacks); i.e. we have
\[
\mathcal I(s\wedge t) = \mathcal I(s)\cap\mathcal I(t) \subset \mathcal I(\max S)\,.
\]
We call $S$ the shape of $\mathcal I$ and $|\mathcal I|:=\mathcal I(\max S)$ the ambient marked set of $\mathcal I$.
\end{definition}

Arrangements of marked finite sets induce those of manifolds.

\begin{definition}
Let $\mathcal I$ be an arrangement of marked finite sets of shape $S$.
Then we define an arrangement $\mathcal E^{\mathcal I}$ of manifolds by
\[
\mathcal E^{\mathcal I}:S\ni s \mapsto \mathbb H^{\mathcal I(s)}\in\mathbf{Emb}\,,
\]
and we call it the standard arrangement associated with $\mathcal I$.
\end{definition}

Roughly speaking, we want to treat with arrangements which are locally $\mathcal E^{\mathcal I}$.
To give a formal definition of this, we prepare a notation:

\begin{definition}
Let $\mathcal X$ be an arrangement of manifolds of shape $S$.
For a point $p\in |\mathcal X|$, we write
\[
s(p) := \bigwedge \{s\in S\mid p\in\mathcal X(s)\}
\]
and call it the scope at $p$ in $\mathcal X$.
If there is no $s\in S$ with $p\in\mathcal X(s)$, then we use the convention that $s(p):=\infty\notin S$.
\end{definition}

Notice that since $\mathcal X:S\to\mathbf{Emb}$ preserves pullbacks, we have $p\in\mathcal X(s)$ if and only if $s\ge s(p)$.

Finally, we can formulate our naive idea of ``excellent'' arrangements.
We denote by $S_{\ge s(p)}$ the full subposet of $S$ spanned by $s\in S$ with $s(p)\le s$.

\begin{definition}
Let $\mathcal X$ be an arrangement of manifolds of shape $S$.
For a point $p\in |\mathcal X|$, an $\mathcal X$-chart around $p$ is a triple $(U,\mathcal I,\varphi)$ of
\begin{itemize}
  \item an open subset $U\subset |\mathcal X|$ of the ambient manifold;
  \item an arrangement $\mathcal I:S_{\ge s(p)}\to\mathbf{MkFin}_{\mathrm{inj}}$ of marked finite sets
  \item an open embedding $\varphi:U\hookrightarrow\mathbb H^{|\mathcal I|}$ with $\varphi(p)=0$
\end{itemize}
which satisfy the following conditions:
\begin{enumerate}[label={\rm(\roman*)}]
  \item For $s\in S$, $U\cap\mathcal X(s)\neq\varnothing$ implies $s\ge s(p)$.
  \item We have $\varphi(U\cap\mathcal X(s))\subset\mathbb H^{\mathcal I(s)}$ for each $s\ge s(p)$.
  \item The following square is a pullback for each $s\ge s(p)$:
\[
\xymatrix{
  U\cap\mathcal X(s) \ar[r]^{\eta_s} \ar[d]_{\varphi_s} \ar@{}[dr]|(.4){\pbcorner} & U \ar[d]^{\varphi} \\
  \mathbb{H}^{\mathcal I(s)} \ar@{^(->}[r] & \mathbb{H}^{|\mathcal I|} }
\]
\end{enumerate}
\end{definition}

\begin{remark}
The condition of pullbacks guarantees that, for each $s(p)\le s\le t\in S$, we have a pullback square:
\[
\xymatrix{
  U\cap\mathcal X(s) \ar[r]^{\eta_{s,t}} \ar[d]_{\varphi_s} \ar@{}[dr]|(.4){\pbcorner} & U\cap\mathcal X(t) \ar[d]^{\varphi_t} \\
  \mathbb{H}^{\mathcal I(s)} \ar@{^(->}[r] & \mathbb{H}^{\mathcal I(t)} }
\]
\end{remark}

\begin{definition}
An arrangement $\mathcal X$ is said to be excellent if each point $p\in |\mathcal X|$ of the ambient manifold admits an $\mathcal X$-chart.
\end{definition}

\begin{remark}
We do not require the closedness for excellent arrangements explicitly.
Actually, if an arrangement $\mathcal X$ is excellent, it is automatically closed.
For, let $p\in |\mathcal X|\setminus\mathcal X(s)$ be a point.
Since $\mathcal X$ is excellent, we can take an $\mathcal X$-cahrt $(U,\mathcal I,\varphi)$.
We have $p\notin\mathcal X(s)$, so $s\not\ge s(p)$, which implies $U$ cannot intersect with $\mathcal X(s)$.
We obtain $p\in U\subset |\mathcal X|\setminus\mathcal X(s)$ and $\mathcal X(s)\hookrightarrow |\mathcal X|$ is a closed embedding.
\end{remark}

\begin{example}
Let $\mathcal I$ be an arrangement of marked sets, and let $\mathcal E^{\mathcal I}$ be the associated standard arrangement.
Then it is obvious that $\mathcal E^{\mathcal I}$ is an excellent arrangement of manifolds.
\end{example}

\begin{example}
Suppose $S$ be a finite totally ordered set, and consider an arrangement $\mathcal X$ of shape $S$ such that, for each $s\in S$, $\partial X(s)=\varnothing$.
Then, \Cref{prop:immersion-chart} asserts that $\mathcal X$ is already excellent (cf. \Cref{ex:linear-arr}).
In particular, every pair $(X,X')$ of manifolds and submanifolds without boundaries is an example of excellent arrangement, which is exactly what is considered in \cite{Ish98}.
More generally, if each embedding $\mathcal X(s)\hookrightarrow\mathcal X(t)$ for $s\le t\in S$ is smooth in the sense of Joyce \cite{Joy09}, then the same result holds without the assumption $\partial X(s)=\varnothing$.
\end{example}

The following results show how we can construct new examples of excellent arrangements from olds.

\begin{proposition}
\label{prop:excel-newold}
\begin{enumerate}[label={\rm(\arabic*)}]
  \item Let $\mathcal X$ and $\mathcal Y$ be two excellent arrangements of shape $S$ and $T$ respectively.
Then the product arrangement $\mathcal X\times\mathcal Y$ is excellent (see \Cref{ex:arr:product}).
  \item Let $\mathcal X$ be an excellent arrangement of shape $S$, and let $\mu:S'\to S$ be a lattice homomorphism.
Then the restricted arrangement $\mathcal X_\mu$ is also excellent (see \Cref{ex:arr:restrict}).
\end{enumerate}
\end{proposition}

\subsection{Vector fields on arrangements}
\label{sec:vecfield-arr}

\begin{definition}
Let $\mathcal X$ be an arrangement of manifolds of shape $S$.
\begin{enumerate}[label={\rm(\arabic*)}]
  \item A vector field on $\mathcal X$ is a vector field $\xi$ on the ambient manifold $|\mathcal X|$ such that $\xi\parallel\mathcal X(s)$ for each $s\in S$.
  \item A vector $\xi$ on $\mathcal X$ is said to be inward-pointing if it restricts to a inward-pointing vector field on $\mathcal X(s)$ for each $s\in S$.
\end{enumerate}
\end{definition}

\begin{remark}
Let $\mathcal X$ be an excellent arrangement of manifolds of shape $S$.
In this case, for each $s,t\in S$ and each $p\in\mathcal X(s\wedge t)$, we have the following pullback square:
\[
\xymatrix{
  T_p\mathcal X(s\wedge t) \ar[r] \ar[d] \ar@{}[dr]|(.4)\pbcorner & T_p\mathcal X(s) \ar[d] \\
  T_p\mathcal X(t) \ar[r] & |\mathcal X| }
\]
It follows that the induced functor
\[
S\to\mathbf{Emb}\,;\quad s \mapsto T\mathcal X(s)\,,
\]
which we denote by $T\mathcal X$, gives rise to an arrangement of manifolds of shape $S$.
From this viewpoint, a vector field $\xi$ on $|\mathcal X|$ is actually that on $\mathcal X$ if and only if $\xi:|\mathcal X|\to T|\mathcal X|=|T\mathcal X|$ is a map over $S$.
\end{remark}

\begin{proposition}
\label{prop:vecfield-arr-flow}
Let $\mathcal X$ be an arrangement of manifolds of shape $S$.
Suppose we have an inward-pointing vector field $\xi$ on $\mathcal X$.
Then there is a unique open submersion $\varphi:\mathcal X\times\mathbb R_+\to\mathcal X$ over $S$ satisfying the following conditions:
\begin{enumerate}[label={\rm(\roman*)}]
  \item For every $p\in |\mathcal X|$ and $t,t'\in\mathbb R_+$, we have
\[
\varphi(p,0)=p\,,\ \text{and}\quad
\varphi(\varphi(p,t),t') = \varphi(p,t+t')\,.
\]
  \item If we write $\partial/\partial t$ the canonical vector field on $\mathbb R_+$, then
\[
\varphi_\ast\frac\partial{\partial t} = \xi\,.
\]
\end{enumerate}
\end{proposition}
\begin{proof}
Since $\xi$ is inward-pointing, the existence of $\varphi:|\mathcal X|\times\mathbb R_+\to|\mathcal X|$ with the two properties is a very classical result, and we omit the proof.
The assertion that $\varphi$ is an $S$-map follows from the assumption that $\xi$ is a vector field on $\mathcal X$ and the uniqueness of the integral curves of vector fields.
\end{proof}

We call $\varphi$ in \Cref{prop:vecfield-arr-flow} the flow of $\xi$.

\begin{lemma}
\label{lem:vecfield-flow-emb}
Let $\mathcal X$ be an excellent arrangement of manifolds, and let $\xi$ be an inward-pointing vector field on $\mathcal X$.
Suppose $N\subset\partial_1|\mathcal X|$ is an open subset such that $\xi\pitchfork N$.
Then the flow $\varphi$ of $\xi$ restricts to an open embedding
\[
\eta:(N\cap\mathcal X)\times\mathbb R_+ \to \mathcal X
\]
of arrangements.
\end{lemma}
\begin{proof}
We first show that $\eta$ is injective.
Note that every integral curve of vector fields is either injective or periodic.
Since we assumed $N\subset\partial_1|\mathcal X|$ and $\xi\pitchfork N$, every integral curve of $\xi$ starting at a point in $N$ is injective and does not intersect those starting at other points in $N$.
This implies that the restriciton map $\eta$ is injective.

Now, to see $\eta$ is an open embedding, it suffices to show that it is a local diffeomorphism.
Since the domain and codomain of $\eta$ are of the same dimension, it suffices to show $\eta$ is a submersion.
Since $\xi\pitchfork N$, for each $p\in N$, we have
\[
T_p|\mathcal X|\simeq T_pN\otimes\mathbb R\xi_p\,.
\]
The properties of the flow $\varphi$ implies that the induced map $\varphi_\ast:T_{(p,t)}(|\mathcal X|\times\mathbb R_+)\to T_{\varphi(p,t)}|\mathcal X|$ maps two vectors $\xi_p$ and $\left.\frac\partial{\partial t}\right|_t$ into the same vector.
On the other hand, $\varphi$ is a submersion, so that the map
\[
\eta_\ast:T_{(p,t)}(N\times\mathbb R_+)
\simeq T_pN\otimes\mathbb R\left.\frac\partial{\partial t}\right|_t
\to T_p|\mathcal X|
\]
is a submersion.
Thus, we obtain the result.
\end{proof}

\section{Treatment of corners}
\label{sec:corner-arr}

One motivation of introducing the notion of arrangements is that we want to treat corners of manifolds better.
For example, it is sometimes a problem to consider the space of smooth maps $F:X\to Y$ between manifolds with corners such that $F$ send each corners of $X$ to designated one of $Y$.
In this section, we discuss it.
We focus in particular on manifolds with faces.
The goal of this section is to see that our relative settings cover the treatment of corners of manifolds with faces and to realize the space of mappings controlled aouund corners as that of relative maps.
Also, in the last section, collarings of manifolds are mentioned.

\subsection{Manifolds with faces}
\label{sec:mfd-face}

At first, we introduce the notion of manifolds with faces, which was originally introduced by J\"anich in \cite{Jan68}.

\begin{definition}[J\"anich, \cite{Jan68}]
Let $X$ be a manifold with corners.
\begin{enumerate}[label={\rm(\arabic*)}]
  \item A connected face of $X$ is the closure of a connected component of $\partial_1X$.
  \item $X$ is called a manifold with faces if every corner $p\in X$ of index $k$ belongs to exactly $k$ connected faces of $X$.
\end{enumerate}
\end{definition}

If $X$ is a manifold with faces, we denote by $\operatorname{bd}X$ the set of connected faces of $X$.
We have a canonical bijection $\pi_0(\partial_1X)\to\operatorname{bd}X$.
We can define manifolds with faces in another way:

\begin{lemma}
\label{lem:mfdface-equiv}
Let $X$ be a manifold with corners.
Then the following two conditions are equivalent:
\begin{enumerate}[label={\rm(\alph*)}]
  \item\label{cond:mfdface-equiv:face} $X$ is a manifold with faces.
  \item\label{cond:mfdface-equiv:nhd} Every point $p\in X$ admits a coordinate open neighborhood $U$ such that the map
\[
\pi_0(\partial_1U) \to \pi_0(\partial_1 X)
\]
is injective.
\end{enumerate}
\end{lemma}
\begin{proof}
For each $p\in X$, we denote by $c(p)$ the number of connected faces of $X$ which contain $p$.
Note that $X$ is a manifold with faces if and only if, for each $p\in X$, we have $c(p)=k$ provided $p$ is a corner of index $k$.
On the other hand, if $p$ is a corner of index $k$, we can take a coordinate open neighborhood $U$ centered at $p$ so that $\pi_0(\partial_1 U)$ consists of exactly $k$ elements; e.g. take a coordinate $\varphi:U\to\mathbb H^{\langle m|k\rangle}$ so that the image $\varphi(U)$ is convex.
Now, consider the map $\pi_0(\partial_1U)\to\pi_0(\partial_1X)$, whose image we denote by $A$.
If $C$ is a connected face of $X$ containing $p$, then it is the closure of an element of $A$.
Conversely, since every connected face of $U$ contains $p$, the closure of each element of $A$ contains $p$.
Thus, we have the inequality
\[
c(p) = \#A \le \#\pi_0(\partial_1U) = k\,.
\]
The middle equality holds if and only if the map is injective.
Hence the implication \ref{cond:mfdface-equiv:face} $\Rightarrow$ \ref{cond:mfdface-equiv:nhd} follows.
The converse follows from the argument above and the observation that if $\varphi:U'\to\mathbb H^{\langle m|k\rangle}$ is a coordinate centered at $p$, and $U\subset U'$ is an open neighborhood of $p$ such that $\varphi(U)\subset\mathbb H^{\langle m|k\rangle}$ is convex, then the map $\pi_0(\partial_1U)\to\pi_0(\partial_1 U')$ is injective.
\end{proof}

The next lemma shows that if $X$ is a manifold with faces, then each face of $X$ is actually a submanifold as expected.
For a finite subset $\sigma\subset\operatorname{bd}X$ of connected faces, we write
\[
\overline\partial_\sigma X := \bigcap_{C\in\sigma} C\,.
\]
In particular, we use the convention that $\overline\partial_\varnothing X = X$.

\begin{lemma}
\label{lem:mfdface-cap}
Let $X$ be a manifold with faces.
Suppose $\sigma=\{C_1,\dots,C_r\}\subset\operatorname{bd}X$ is a set of $r$ distinct connected faces of $X$.
Then, for each $p\in\overline\partial_\sigma X$, there are a coordinate $\varphi:U\to\mathbb H^{\langle m|k\rangle}$ centered at $p$ and an injective map
\[
\beta:\sigma\to\{k+1,\dots,m\}
\]
such that the following square is a pullback:
\[
\xymatrix{
  \overline\partial_\sigma X\cap U \ar[r] \ar[d] \ar@{}[dr]|(.4)\pbcorner & U \ar[d]^\varphi \\
  \{x_{\beta(1)}=x_{\beta(2)}=\dots=x_{\beta(r)}0\} \ar@{^(->}[r] & \mathbb H^{\langle m|k\rangle} }
\]
In particular, the subset $\overline\partial_\sigma X\subset X$ is a submanifold of $X$ of codimension $r$.
\end{lemma}
\begin{proof}
Take a coordinate $\varphi:U\to\mathbb H^{\langle m|k\rangle}$ centered at $p$ such that $\varphi(U)$ is convex, and the map $\pi_0(\partial_1U)\to\pi_0(\partial_1X)$ is injective.
Then we obtain an injective map
\[
C^\varphi:
\{k+1,\dots,m\}
\simeq \pi_0(\partial_1\mathbb H^{\langle m|k\rangle})
\simeq \pi_0(\partial_1U)
\to \pi_0(\partial_1X)
\simeq \operatorname{bd}X\,,
\]
where the first map is given by
\[
i\mapsto \{x_i=0\}\,.
\]
For each $k+1\le i\le m$, we have a pullback square below:
\begin{equation}
\label{sq:prf:mfdface-cap:face}
\vcenter{
  \xymatrix{
    C^\varphi(i)\cap U \ar[r] \ar[d] \ar@{}[dr]|(.4)\pbcorner & U \ar[d]^\varphi \\
    \{x_i=0\} \ar[r] & \mathbb H^{\langle m|k\rangle} } }
\end{equation}
Note that since $p\in\overline\partial_\sigma X=C_1\cap\dots\cap C_r$ and $U$ is an open neighborhood of $p$, the set $\sigma=\{C_1,\dots,C_r\}$ is included in the image of the map $C^\varphi$.
Hence, one can find a map $\beta:\sigma\to\{k+1,\dots,m\}$ so that $C^\varphi\beta=\mathrm{id}_\sigma$.
Then, the required property is veified using the pullback square \ref{sq:prf:mfdface-cap:face}.
\end{proof}

\begin{corollary}
\label{cor:mfdface-uniquecap}
Let $X$ be a manifold with faces, and let $\sigma,\tau\subset\operatorname{bd}X$ be two finite sets of connected faces of $X$ such that $\overline\partial_\sigma X\neq\varnothing$.
Then $\overline\partial_\sigma X\subset\overline\partial_\tau X$ if and only if $\sigma\supset\tau$.
\end{corollary}
\begin{proof}
It is directly follows from the definition that $\sigma\supset\tau$ implies $\overline\partial_\sigma X\subset\overline\partial_\tau X$, so we show the converse.
First, consider the case $\overline\partial_\sigma X=\overline\partial_\tau X$.
This equation implies that, for each $D\in\tau$,
\[
\overline\partial_{\sigma\cup\{D\}} X
= D\cap\overline\partial_\sigma X
= D\cap\overline\partial_\tau X
= \overline\partial_\tau X\,.
\]
Comparing the codimensions in $X$, we deduce from \Cref{lem:mfdface-equiv} that $C\in\tau$.
Hence, we obtain $\sigma\subset\tau$.
The same argument also shows $\sigma\supset\tau$, so we obtain $\sigma=\tau$.

Finally, suppose we have $\overline\partial_\sigma X\subset\overline\partial_\tau X$.
Then, we obtain
\[
\overline\partial_{\sigma\cup\tau}X
= \overline\partial_\sigma X \cap \overline\partial_\tau X
= \overline\partial_\sigma X\,.
\]
By the first part, this implies $\sigma\cup\tau=\sigma$; in other words, we have $\sigma\supset\tau$, which is the required result.
\end{proof}

Now, we see manifolds with faces give rise to excellent arrangements.

\begin{definition}
Let $X$ be a manifold with faces.
We define a poset $\Gamma_X$ whose underlying set is the set of subsets of $\operatorname{bd}X$ and ordered by the opposite of the incusions.
We call $\Gamma_X$ the face lattice of $X$.
For an element $\sigma\in\Gamma_X$, we denote by $c(\sigma)$ the cardinality of $\sigma$.
\end{definition}

We consider the assignment
\begin{equation}
\label{eq:tildeX-def}
\widetilde X:\Gamma_X\to\mathbf{Emb}
\ ;\quad \sigma\mapsto\overline\partial_\sigma X\,,
\end{equation}
which is in fact a pre-arrangement of manifolds thanks to \Cref{cor:mfdface-uniquecap}.
Furthermore, \Cref{cor:mfdface-uniquecap} even implies $\widetilde X$ is an arrangement of manifolds.

More generally, we can combine the arrangements of faces with excellent arrangements.
For an excellent arrangement $\mathcal X$, we write $\Gamma_{\mathcal X}:=\Gamma_{|\mathcal X|}$ for short.

\begin{proposition}
\label{prop:excel-bdry}
Let $\mathcal X$ be an excellent arrangement of shape $S$ such that the ambient manifold $|\mathcal X|$ is a manifold with finite faces.
Define a functor $\widetilde{\mathcal X}:S\times\Gamma_{\mathcal X}\to\mathbf{Emb}$ by
\[
\widetilde{\mathcal X}(s,\sigma) := \mathcal X(s)\cap\overline\partial_\sigma|\mathcal X|\,.
\]
Then $\widetilde{\mathcal X}$ is an excellent arrangement of manifolds.
\end{proposition}

The proof is obvious from the local model for excellent arrangements.

\begin{corollary}
\label{cor:bdry-arr-excel}
If $X$ is a manifold with faces, then the boundary arrangement $\widetilde X$ is excellent.
\end{corollary}
\begin{proof}
It is a special case of \Cref{prop:excel-bdry} for $S=\mathbf{pt}$.
\end{proof}

\begin{corollary}
\label{cor:excel-face}
Let $\mathcal X$ be an excellent arrangement of shape $S$ whose ambient manifold $|\mathcal X|$ is a manifold with finite faces.
Then for $0\le i\le k$, each element $\sigma\in\Gamma_{\mathcal X}$ gives rise to an excellent arrangement
\[
\overline\partial_\sigma\widetilde{\mathcal X}:S\to\mathbf{Emb}
\ ;\quad s\mapsto\mathcal X(s)\cap\overline\partial_\sigma|\mathcal X|\,.
\]
\end{corollary}
\begin{proof}
Notice that the arrangement $\overline\partial_\sigma\widetilde{\mathcal X}$ can be obtained by restricting the excellent arrangement $\widetilde{\mathcal C}$ defined in \ref{prop:excel-bdry} along the lattice homomorphism
\[
\mu:S\to S\times\Gamma_{\mathcal X}
\ ;\quad s\mapsto (s,\sigma)\,.
\]
Thus, the result follows from \Cref{prop:excel-newold}.
\end{proof}

For a general excellent arrangement $\mathcal X$, say, of shape $S$, the new excellent arrangement $\mathcal X$ might have a superfluous part; for instance, the arrangement $\widetilde{\widetilde {\mathcal X}}:S\times\Gamma_{\mathcal X}\times\Gamma_{\mathcal X}\to\mathbf{Emb}$ has obvious doubles.
This sometimes causes inefficient treatments of arrangement, and to avoid them, we consider additional conditions:

\begin{definition}
An arrangement $\mathcal X$ of manifolds of shape $S$ is said to be neat if the following are satisfied:
\begin{enumerate}[label={\rm(\roman*)}]
  \item $\mathcal X$ is excellent.
  \item For each $s\in S$, $\mathcal X(s)$ is a manifold with finite faces.
  \item For each $s\in S$ and each non-negative integer $k$, we have $\partial_k\mathcal X(s) = \mathcal X(s)\cap\partial_k|\mathcal X|$.
\end{enumerate}
\end{definition}

\begin{lemma}
\label{lem:excel-neat}
Let $\mathcal X$ be an excellent arrangement of manifolds of shape $S$ such that each $\mathcal X(s)$ is a manifold with finite faces for $s\in S$.
Then $\mathcal X$ is a neat arrangement if and only if each $p\in|\mathcal X|$ admits an $\mathcal X$-chart $(U,\mathcal I,\varphi)$ such that we have a pullback square
\[
\xymatrix{
  U\cap\mathcal X(s) \ar[r] \ar[d] \ar@{}[dr]|(.4)\pbcorner & U \ar[d]^\varphi \\
  \mathbb H^{\mathcal I(s)} \ar[r] & \mathbb H^{\langle m|k\rangle} }
\]
with $\mathcal I(s)_+ = \langle m|k\rangle_+ = \{k+1,\dots,m\}$ for each $s\in S_{\ge s(p)}$.
\end{lemma}
\begin{proof}
Suppose first that $\mathcal X$ is a neat arrangement, and let $p\in|\mathcal X|$.
Take an arbitrary $\mathcal X$-chart $(U,\mathcal I,\varphi)$ around $p$ with $|\mathcal I|=\langle m|k\rangle$, so that we have a pullback square
\begin{equation}
\label{diag:prf:excel-neat:pb}
\xymatrix{
  U\cap\mathcal X(s) \ar[r] \ar[d] \ar@{}[dr]|(.4)\pbcorner & U \ar[d]^\varphi \\
  \mathbb H^{\mathcal I(s)} \ar[r] & \mathbb H^{\langle m|k\rangle} }
\end{equation}
for each $s\ge s(p)$.
Then, since $\mathcal X$ is neat, for $1\le r\le m-k$, we have
\[
\partial_r(U\cap\mathcal X(s))
= U\cap\partial_r\mathcal X(s)
= U\cap\mathcal X(s)\cap\partial_r|\mathcal X|
= (\partial_r U)\cap\mathcal X(s)\,.
\]
By the pullback square \eqref{diag:prf:excel-neat:pb}, this implies that
\[
\partial_r\mathbb H^{\mathcal I(s)} = \mathbb H^{\mathcal I(s)}\cap\partial_r\mathbb H^{\langle m|k\rangle}\,.
\]
In particular, we have $\partial_{m-k}\mathbb H^{\mathcal I(s)}=\{0\}$, which implies $\#\mathcal I(s)_+=m-k$ and hence $\mathcal I(s)_+=\langle m|k\rangle_+$.

Conversely, suppose we have an $\mathcal X$-chart $(U,\mathcal I,\varphi)$ around $p\in|\mathcal X|$ such that $|\mathcal I|=\langle m|k\rangle$ and $\mathcal I(s)_+=\langle m|k\rangle_+$ for each $s\in S_{\ge s(p)}$.
In this case, one can easily verify the equation
\[
\partial_r\mathbb H^{\mathcal I(s)}
= \mathbb H^{\mathcal I(s)}\cap\partial_r\mathbb H^{\langle m|k\rangle}
\]
for $1\le r\le m-k$.
By the pullback square \eqref{diag:prf:excel-neat:pb}, we obtain $\partial_r (U\cap\mathcal X(s))=\mathcal X(s)\cap \partial_rU$ for each $s\in S_{\ge s(p)}$ and $0\le r\le m-k$.
Thus, $\mathcal X$ is neat if it is covered by such $\mathcal X$-charts, which is the required result.
\end{proof}

\subsection{Edgings of manifolds}
\label{sec:edge}

Next, we introduce a notion to relate corners of two manifolds.
We make use of this to require maps to send each faces to designated ones.

For this purpose, we consider partial maps between the sets of connected faces.
Recall that, for sets $A$ and $B$, a partial map from $A$ to $B$ is a map from a subset $A'\subset A$ to $B$.
If $A\supset A'\xrightarrow f B$ and $B\supset B'\xrightarrow g C$ are two partial maps, then we have a partial map $gf$ defined on $A'\cap f^{-1}(B)$.
This defines a category; we denote by $\mathbf{Fin}^\partial$ the category of finite sets and partial maps.
For a partial map $f:A\to B\in\mathbf{Fin}^\partial$, we write $D(f)$ the subsets of $A$ where $f$ defined.
On the other hand, we define a category $\mathbf{BLat}^{\mathsf f}_{\mathrm{rex}}$ as follows:
The objects are finite Boolean lattices.
For a finite Boolean lattice $\Lambda$, we define the coheight of each element $\lambda\in\Lambda$ to be the length $l$ of the longest chain $\lambda=\lambda_0<\lambda_1<\dots<\lambda_l$ in $\Lambda$.
Then, the morphisms of $\mathbf{BLat}^{\mathsf f}_{\mathrm{rex}}$ are those maps $\varphi:\Lambda\to\Lambda'$ which preserves all infimums and do not raise the coheights.
We have the following basic result:

\begin{proposition}
\label{prop:partmap-blex}
For a finite set $A$, we denote by $\Gamma_A$ the opposite lattice of the powerset of $A$.
For each partial map $f:A\to B\in\mathbf{Fin}^\partial$, we define a map $\widetilde f:\Gamma_A\to \Gamma_B$ by
\[
\widetilde f(A') := f(A'\cap D(\beta))\,.
\]
Then, the assignment $A\mapsto\Gamma_A$ defines a functor $\mathbf{Fin}^\partial\to\mathbf{BLat}^{\mathsf f}_{\mathrm{rex}}$ which is an equivalence of categories.
\end{proposition}
\begin{proof}
We first verify that $\widetilde f:\Gamma_A\to\Gamma_B$ is a morphism of $\mathbf{BLat}^{\mathsf f}_{\mathrm{rex}}$.
Note that, for a subset $A'\subset A$, its coheight in $\Gamma_A$ is nothing but the cardinality of $A'$.
Hence, $\widetilde f$ does not raise the coheight obviously.
In addition, infimums in $\Gamma_A$ can be computed as unions of subsets, so it is also obvious that $\widetilde f$ preserves infimums since the lattice $\Gamma_A$ is distributive.

Now, the functoriality of $\Gamma_{\blankdot}:\mathbf{Fin}^\partial\to\mathbf{BLat}^{\mathsf f}_{\mathrm{rex}}$ is easy.
We show it is an equivalence of categories.
We construct the inverse $\operatorname{coat}:\mathbf{BLat}^{\mathsf f}_{\mathrm{rex}}\to\mathbf{Fin}^\partial$.
For a finite Boolean lattice $\Lambda$, define $\operatorname{coat}(\Lambda)$ to be the subset of $\Lambda$ consisting of elements of coheight $1$.
Then, each morphism $\varphi:\Lambda\to\Lambda'\in\mathbf{BLat}^{\mathsf f}_{\mathrm{rex}}$ induces a partial map $f_\varphi:\operatorname{coat}\Lambda\to\operatorname{coat}\Lambda'\in\mathbf{Fin}^\partial$ such that
\begin{itemize}
  \item $D(f_\varphi):=\{x\in\operatorname{coat}\Lambda\mid \operatorname{coht}(\varphi(x);\Lambda')=1\}$, here $\operatorname{coht}(y;\Lambda')$ is the coheight of $y\in\Lambda'$ in $\Lambda'$;
  \item $D(f_\varphi)(x):=\varphi(x)$.
\end{itemize}
It is verified that the assignment $\varphi\mapsto f_\varphi$ is functorial so that $\operatorname{coat}:\mathbf{BLat}^{\mathsf f}_{\mathrm{rex}}\to\mathbf{Fin}^\partial$.
Finally, there are the following natural isomorphisms:
\[
\begin{gathered}
\Gamma_{\operatorname{coat}\Lambda}\xrightarrow\simeq\Lambda
\ ;\quad S \mapsto \bigwedge S \\
A \xrightarrow\simeq \operatorname{coat}\Gamma_A
\ ;\quad a \mapsto \{a\}
\end{gathered}
\]
Thus, we obtain the result.
\end{proof}

Now, recall that, for a manifold $X$ with finite faces, we have a set $\operatorname{bd}X$ and the associated Boolean lattice $\Gamma_X=\Gamma_{\operatorname{bd}X}$.

\begin{definition}
Let $X$ and $Y$ be manifolds with finite faces.
Then, a partial map from $\operatorname{bd}X$ to $\operatorname{bd}Y$ is called a pre-edging with $Y$ of $X$.
\end{definition}

By virtue of \Cref{prop:partmap-blex}, every pre-edging $\beta$ with $Y$ of $X$ yields an inf-presreving map $\widetilde\beta:\Gamma_X\to\Gamma_Y$.

Note that a pre-edging $\beta$ regards no more geometric information than the number of faces.
Indeed, if $X$ is a manifold with finite faces, then we have a canonical bijection
\[
\operatorname{bd}X
\cong\operatorname{bd}\mathbb R_+^{\pi_0(\partial_1X)}\,.
\]
To use egings $\beta$ to control behaviors of maps around corners, we need to require $\beta$ to care about the intersections of faces at least.
This leads the following notion:

\begin{definition}
Let $X$ and $Y$ be manifolds with faces.
Then, an edging with $Y$ of $X$ is a pre-edging $\beta$ satisfying the condition that if $C_1,\dots,C_r\in D(\beta)\subset\operatorname{bd}X$ are distinct connected faces with $C_1\cap\dots\cap C_r\neq\varnothing\subset X$, then
\begin{enumerate}[label={\rm(\roman*)}]
  \item $\beta(C_1)\cap\dots\cap \beta(C_r)\neq\varnothing\subset Y$.
  \item $\beta(C_1),\dots,\beta(C_r)\in\operatorname{bd}Y$ are all distinct;
\end{enumerate}
\end{definition}

If $X$ is a manifold with faces, the intersection $C_1\cap\dots\cap C_r$ of $r$ distinct connected faces is a set of corners of codimension at least $r$.
Put $\sigma=\{C_1,\dots,C_r\}\in\Gamma_X$, and we have $\overline\partial_\sigma X=C_1\cap\dots\cap C_r$.
In view of \Cref{lem:mfdface-cap}, its codimension equals to the coheight of the element $\sigma=\{C_1,\dots,C_r\}\in\Gamma_X$.
Hence, if $\beta$ is an edging of $X$ with $Y$, the two conditions above guarantees that $\beta$ respects the codimensions of generic corners on faces which $\beta$ is defined on.
More practically, we often use the following characterization:

\begin{lemma}
\label{lem:edgegamma}
Let $X$ and $Y$ be manifolds with finite faces, and let $\beta$ is a pre-edging with $Y$ of $X$.
Then, $\beta$ is an edging if and only if $\widetilde\beta:\Gamma_X\to\Gamma_Y$ satisfies the following properties:
\begin{enumerate}[label={\rm(\roman*)}]
  \item\label{ass:edgegamma:nonemp} For each $\sigma\in\Gamma_X$, $\overline\partial_\sigma X\neq\varnothing$ implies $\overline\partial_{\beta(\sigma)} Y\neq\varnothing$.
  \item\label{ass:edgegamma:join} If $\sigma,\tau\in\Gamma_X$ are two elements with $\overline\partial_{\sigma\wedge\tau} X\neq\varnothing$, then we have
\[
\widetilde\beta(\sigma\vee\tau) = \widetilde\beta(\sigma)\vee\widetilde\beta(\tau)\,.
\]
\end{enumerate}
\end{lemma}
\begin{proof}
First suppose $\widetilde\beta$ enjoys the properties above.
Let $C_1,\dots,C_r\in D(\beta)$ be distinct connected faces with $C_1\cap\dots\cap C_r\neq\varnothing$.
Then, the property \ref{ass:edgegamma:nonemp} implies
\[
\beta(C_1)\cap\dots\cap\beta(C_r)
= \overline\partial_{\beta(\sigma)}Y
\neq\varnothing\,.
\]
On the other hand, since we have $\varnothing\neq\overline\partial_\sigma X\subset C_i\cap C_j$ for each $1\le i<j\le r$, the property \ref{ass:edgegamma:join} implies
\begin{equation}
\label{eq:prf:edgegamma:facevee}
\widetilde\beta(C_i)\vee\widetilde\beta(C_j)
= \widetilde\beta(C_i\vee C_j)
= \widetilde\beta(X)
= Y\in\Gamma_Y\,.
\end{equation}
If both $C_i$ and $C_j$ belong to the subset $D(\beta)$, $\widetilde\beta(C_i)$ and $\widetilde\beta(C_j)$ are of coheight $1$.
Thus, the equation \ref{eq:prf:edgegamma:facevee} implies $\widetilde\beta(C_i)\neq\widetilde\beta(C_j)$.
It follows that $\beta$ is an edging.

Conversely, suppose $\beta$ is an edging.
The property \ref{ass:edgegamma:nonemp} is obvious, so we show \ref{ass:edgegamma:join}.
Suppose we are given $\sigma,\tau\in\Gamma_X$ with $\overline\partial_{\sigma\wedge\tau}X\neq\varnothing$.
Note that we have
\[
\sigma = \bigwedge_{C\in\sigma} C
\ ,\quad \tau = \bigwedge_{C\in\tau} C
\ ,\quad \sigma\vee\tau = \bigwedge_{C\in\sigma\cap\tau} C \in \Gamma_X\,.
\]
Since $\beta$ is an edging, the condition on $\sigma$ and $\tau$ implies $\beta$ is injective on $\sigma\wedge\tau=\sigma\cup\tau\in\Gamma_X$.
Hence, we obtain
\[
\begin{split}
\widetilde\beta(\sigma\vee\tau)
&= \bigwedge_{C\in\sigma\cap\tau} \widetilde\beta(C) \\
&= \{\beta(C)\mid C\in\sigma\cap\tau\cap D(\beta)\} \\
&= \{\beta(C)\mid C\in\sigma\cap D(\beta)\}\cap\{\beta(C)\mid\in\tau\cap D(\beta)\} \\
&= \widetilde\beta(\sigma)\vee\widetilde\beta(\tau)\,,
\end{split}
\]
which implies the property \ref{ass:edgegamma:join}.
\end{proof}

\begin{corollary}
\label{cor:edge-slicesurj}
Let $X$ and $Y$ be manifolds with finite faces, and let $\beta$ be an edging of $X$ with $Y$.
If $\sigma\in\Gamma_X$ is an element with $\overline\partial_\sigma X\neq\varnothing$, then the induced map
\[
\widetilde\beta_{\ge\sigma}:(\Gamma_X)_{\ge\sigma}\to(\Gamma_Y)_{\ge\widetilde\beta(\sigma)}
\ ;\quad \tau \mapsto \widetilde\beta(\tau)
\]
is a surjective lattice homomorphism.
\end{corollary}
\begin{proof}
Since $\overline\partial_\sigma X\neq\varnothing$, \Cref{lem:edgegamma} implies that for each $\tau,\tau'\ge\sigma$, we have $\widetilde\beta(\tau\vee\tau')=\widetilde\beta(\tau)\vee\widetilde\beta(\tau')$ as well as $\widetilde\beta(\tau\wedge\tau')=\widetilde\beta(\tau)\wedge\widetilde\beta(\tau')$.
Thus, the restriction $\widetilde\beta_{\ge\sigma}$ is a lattice homomorphism.
We show it is surjective.
Note that the sublattice $(\Gamma_X)_{\ge\sigma}\subset\Gamma_X$ is generated under infimums by connected faces $C$ which belongs to $\sigma$; in other words, there is an isomorphism $(\Gamma_X)_{\ge\sigma}\simeq\Gamma_\sigma$.
Similarly, we have $(\Gamma_Y)_{\ge\widetilde\beta(\sigma)}\simeq\Gamma_{\widetilde\beta(\sigma)}$.
Moreover, the corresponding map $\Gamma_\sigma\to\Gamma_{\widetilde\beta(\sigma)}$ is associated to the restricted partial map $\beta|_\sigma:\sigma\cap D(\beta)\to\widetilde\beta(\sigma)$.
The latter map is clearly surjective, so the result follows.
\end{proof}

\begin{example}
\label{ex:edge-id}
For any manifold $X$ with finite faces, the identity map on $\operatorname{bd}X$ gives an edging of $X$ with itself.
More generally, for every subset $A\subset\operatorname{bd}X$, the partial map $\operatorname{bd}X\supset A\hookrightarrow\operatorname{bd}X$ defines an edging of $X$.
\end{example}

\begin{example}
\label{ex:edge-comp}
Let $\beta$ be an edging of $X$ with $Y$, and let $\gamma$ be an edging of $Y$ with $Z$.
Then, the composed partial map $\gamma\beta:\operatorname{bd}X\to\operatorname{bd}Z\in\mathbf{Fin}^\partial$ defines an edging of $X$ with $Z$.
\end{example}

\begin{example}
\label{ex:edge-paste}
Suppose we have two edgings $\beta_1$ and $\beta_2$ with $Y_1$ and $Y_2$ of $X$ respectively.
If we have $D(\beta_1)\cap D(\beta_2)=\varnothing\subset \operatorname{bd}X$, then the partial map
\[
\operatorname{bd}X
\supset D(\beta_1)\amalg D(\beta_2)
\xrightarrow{\beta_1\amalg\beta_2} \operatorname{bd}Y_1 \amalg \operatorname{bd}Y_2
\simeq \operatorname{bd}(Y_1\times Y_2)
\]
gives an edging with $Y_1\times Y_2$, which we denote by $\beta_1\amalg\beta_2$.
\end{example}

\begin{example}
\label{ex:edge-face}
Let $X$ be a manifold with finite faces with edging $\beta:\Gamma_X\to\Gamma_Y$ with $Y$.
Notice that for $\sigma\in\Gamma_X$, the lattice $\Gamma_{\overline\partial_\sigma X}$ can be identified with the sublattice of $\Gamma_X$ consisting of elements $\sigma'$ with $\sigma'\le\sigma$.
If $\sigma\cap D(\beta)=\varnothing$, then the composition
\[
\beta_\sigma:\Gamma_{\overline\partial_\sigma X}\hookrightarrow\Gamma_X\xrightarrow\beta \Gamma_Y
\]
is an edging of $\overline\partial_\sigma X$ with $Y$.
We have a canonical identification $D(\beta_\sigma)=D(\beta)_{\le\sigma}\subset\Gamma_X$.
\end{example}

\begin{example}
\label{ex:<n>-mfd}
For a non-negative integer $n$, an $\langle n\rangle$-manifold is a manifold with faces equipped with an edging $\beta$ with $\mathbb R^n_+$ such that $D(\beta)=\operatorname{bd}X$.
This notion was originally introduced by J\"nich in \cite{Jan68}.
\end{example}

For an edging $\beta$ of $X$ with $Y$, we denote by $\Gamma^\beta_X\subset\Gamma_X$ the subset consisting of elements $\sigma$ with $\widetilde\beta(\sigma)=Y\in\Gamma_Y$.
It is easily verified that $\Gamma^\beta_X$ is canonically isomorphic to the sublattice generated by connected faces which do not belong to $D(\beta)$.
We have a canonical isomorphism
\[
\Gamma_X\xrightarrow\simeq\Gamma^\beta_X\times\Gamma_{D(\beta)}
\ ;\quad \sigma \mapsto (\sigma\setminus D(\beta),\sigma\cap D(\beta))\,.
\]
We denote by $\widetilde\beta^\perp:\Gamma_X\to\Gamma^\beta_X$ the associated projection; i.e. $\widetilde\beta^\perp(\sigma)=\sigma\setminus D(\beta)$.
With this notation, we obtain a refinement of \Cref{cor:edge-slicesurj}.

\begin{lemma}
\label{lem:edge-sliceisom}
Let $X$ and $Y$ be manifolds with faces, and let $\beta$ be an edging of $X$ with $Y$.
Suppose $\sigma\in\Gamma_X$ is an element with $\overline\partial_\sigma X\neq\varnothing$.
Then, the map
\[
(\Gamma_X)_{\ge\sigma}
\to (\Gamma^\beta_X)_{\ge\sigma}\times(\Gamma_Y)_{\ge\widetilde\beta(\sigma)}
\ ;\quad \sigma'\mapsto (\widetilde\beta^\perp(\sigma'),\widetilde\beta(\sigma'))
\]
is an isomorphism of lattices.
\end{lemma}
\begin{proof}
The result follows from \Cref{cor:edge-slicesurj} and the general fact that, for a surjective lattice homomorphism $\varphi:\Lambda\to\Lambda'$ between finite Boolean lattices, the map
\[
\Lambda \to \varphi^{-1}\{\max\Lambda'\}\times\Lambda'
\ ;\quad \lambda\mapsto (\lambda\vee\bigwedge\varphi^{-1}\{\max\Lambda'\},\varphi(\lambda))
\]
is an isomorphism of lattices.
\end{proof}

\begin{notation}
Let $X$ be a manifold with faces equipped with an edging $\beta$ with $Y$.
For each $\tau\in\Gamma_Y$, we will write
\[
\begin{gathered}
\overline\partial^\beta_\tau X := \bigcup_{\widetilde\beta(\sigma)=\tau}\overline\partial_\sigma X \subset X \\
\partial^\beta_\tau X := \overline\partial^\beta_\tau X \setminus \bigcup_{\tau'<\tau}\overline\partial^\beta_{\tau'}X\,.
\end{gathered}
\]
\end{notation}

\begin{proposition}
\label{prop:edge-submfd}
Let $X$ be a manifold with faces equipped with an edging $\beta$ with $Y$.
Then the following hold:
\begin{enumerate}[label={\rm(\arabic*)}]
  \item\label{req:edge-submfd:submfd} For each $\tau\in\Gamma_Y$, the subset $\overline\partial^\beta_\tau X\subset X$ is an embedded submanifold.
  \item\label{req:edge-submfd:wedge} If $\tau,\tau'\in\Gamma_Y$, we have
\[
\overline\partial^\beta_{\tau\wedge\tau'}X = \overline\partial^\beta_\tau X\cap\overline\partial^\beta_{\tau'}X \subset X\,.
\]
In particular, $\tau\le\tau'$ implies $\overline\partial^\beta_\tau X\subset\overline\partial^\beta_{\tau'}X$.
\end{enumerate}
\end{proposition}
\begin{proof}
For each $\tau\in\Gamma_Y$, we denote by $\max\widetilde\beta^{-1}\{\tau\}$ the set of elements $\sigma$ which is maximal among those with $\widetilde\beta(\sigma)=\tau$.
If $\sigma,\sigma'\in\max\widetilde\beta^{-1}\{\tau\}$ are two distinct elements, the maximality implies $\beta(\sigma\vee\sigma')\neq\tau$ while $\beta(\sigma)\vee\beta(\sigma')=\tau$.
It follows from \Cref{lem:edgegamma} that
\[
\overline\partial_\sigma X\cap\overline\partial_{\sigma'}X
= \overline\partial_{\sigma\wedge\sigma'}X
= \varnothing\,.
\]
Thus, we obtain
\begin{equation}
\label{eq:prf:edge-submfd:amalg}
\overline\partial^\beta_\tau X
\cong \coprod_{\sigma\in\max\widetilde\beta^{-1}\{\tau\}} \overline\partial_\sigma X\,.
\end{equation}
The part \ref{req:edge-submfd:submfd} is now obvious.
We show \ref{req:edge-submfd:wedge}.
Let $\sigma''\in\widetilde\beta^{-1}\{\tau\wedge\tau'\}$ be an element with $\overline\partial_{\sigma''}X\neq\varnothing$.
Note that the maximality of $\sigma''$ implies that the subset $(\Gamma^\beta_X)_{\ge\sigma''}$ consists of only one element.
Hence, by \Cref{lem:edge-sliceisom}, the map
\[
\widetilde\beta_{\ge\sigma''}:(\Gamma_X)_{\ge\sigma''}\to(\Gamma_Y)_{\ge\tau\wedge\tau'}
\]
is an isomorphism of lattices.
In particular, there is a unique pair $(\sigma,\sigma')\in\widetilde\beta^{-1}\{\tau\}\times\widetilde\beta^{-1}\{\tau'\}$ such that $\sigma\wedge\sigma'=\sigma''$.
This gives rise to a one-to-one correspondence between such $\sigma''$ and pairs $(\sigma,\sigma')\in(\max\widetilde\beta^{-1}\{\tau\})\times(\max\widetilde\beta^{-1}\{\tau'\})$ with $\overline\partial_\sigma X\cap\overline\partial_{\sigma'}X\neq\varnothing$.
Indeed, for such $(\sigma,\sigma')$, \Cref{lem:edge-sliceisom} guarantees $\sigma\wedge\sigma'\in\max\widetilde\beta^{-1}\{\tau\wedge\tau'\}$.
Finally, using the equation \eqref{eq:prf:edge-submfd:amalg}, we obtain
\[
\begin{multlined}
\overline\partial^\beta_\tau X \cap \overline\partial^\beta_{\tau'}X
= \raisedunderop\coprod{\substack{\sigma\in\max\widetilde\beta^{-1}\{\tau\}\cr\sigma'\in\max\widetilde\beta^{-1}\{\tau'\}}} \overline\partial_\sigma X\cap\overline\partial_{\sigma'}X \\[4pt]
 = \coprod_{\sigma''\in\max\widetilde\beta^{-1}\{\tau\wedge\tau'\}} \overline\partial_{\sigma''}X
= \overline\partial^\beta_{\tau\wedge\tau'}X\mathrlap{\,.}
\end{multlined}
\]
\end{proof}

Finally, we see that edgings give rise to excellent arrangements.

\begin{definition}
Let $\mathcal X$ be a neat arrangement of manifolds, and let $Y$ be a manifold with finite faces.
Then, an edging $\beta$ of $\mathcal X$ with $Y$ is that of the ambient manifold $|\mathcal X|$.
\end{definition}

\begin{proposition}
\label{prop:edge-excel}
Let $\mathcal X$ be a neat arrangement of manifolds of shape $S$ equipped with an edging $\beta$ with $Y$.
We denote by $\Gamma^\beta_{\mathcal X}\subset\Gamma_{\mathcal X}$ the subset of elements $\sigma\in\Gamma_{\mathcal X}$ with $\widetilde\beta(\sigma)=Y\in\Gamma_Y$.
Then the assignment
\[
\widetilde{\mathcal X}^\beta:S\times\Gamma^\beta_{\mathcal X}\times\Gamma_Y \to \mathbf{Emb}
\ ;\quad (s;\sigma,\tau) \mapsto \widetilde{\mathcal X}(s,\sigma)\cap\overline\partial^\beta_\tau |\mathcal X|
\]
defines an excellent arrangement whose ambient manifold is $|\mathcal X|$.
\end{proposition}
\begin{proof}
We first show $\widetilde{\mathcal X}^\beta$ is actually an arrangement of manifolds.
As mentioned in the proof of \Cref{prop:edge-submfd}, we have a canonical diffeomorphism
\[
\overline\partial^\beta_\tau|\mathcal X|
\cong \coprod_{\sigma'\in\max\widetilde\beta^{-1}\{\tau\}} \overline\partial_{\sigma'}|\mathcal X|\,.
\]
Hence, we obtain
\begin{equation}
\label{eq:prf:edge-excel}
\widetilde{\mathcal X}^\beta(s;\sigma,\tau)
\cong \coprod_{\sigma'\in\max\widetilde\beta^{-1}\{\tau\}}\widetilde{\mathcal X}(s;\sigma\wedge\sigma')\,.
\end{equation}
which is clearly a submanifold of $|\mathcal X|$.
It is also clear that $\widetilde{\mathcal X}^\beta$ preserves pullbacks in each variable, so $\widetilde{\mathcal X}^\beta$ is an arrangement.

To see $\widetilde{\mathcal X}^\beta$ is excellent, we have to show that each point $p\in|\mathcal X|$ admits an $\widetilde{\mathcal X}^\beta$-chart around $p$.
In this regard, \Cref{lem:edge-sliceisom} asserts that it is nothing but an $\widetilde{\mathcal X}$-chart around $p$, which always exists since $\widetilde{\mathcal X}$ is excellent by \Cref{lem:excel-neat}.
Thus, the result follows.
\end{proof}

\subsection{Maps along edgings}
\label{sec:mapedge}

In this section, we see that we can actually control maps around corners in terms of edgings.
First of all, we clarify what it means.

\begin{definition}
Let $X$ and $Y$ be manifolds with faces, and let $\beta$ be an edging of $X$ with $Y$.
Then, a smooth map $F:X\to Y$ is said to be along $\beta$ if for each $C\in D(\beta)\subset\operatorname{bd}X$, it satisfies $F(C)\subset\beta(C)$.
We denote by $\mathcal F^\beta(X,Y)\subset C^\infty(X,Y)$ the subset of maps along $\beta$; we regard it to be equipped with the subspace topology.
\end{definition}

To translate it into our language, we introduce some notations:

\begin{notation}
In the situation above, we will write
\[
\widetilde X^\beta_{\sparallel}:\Gamma_Y\to\mathbf{Emb}
\ ;\quad\tau\mapsto\mathcal \overline\partial^\beta_\tau X\\
\]
This is clear a restrictions of the excellent arrangement $\widetilde X^\beta$, so that it is again an excellent arrangement by \Cref{prop:excel-newold}.
\end{notation}

\begin{lemma}
\label{lem:edge-alongLan}
Let $X$, $Y$, and $\beta$ be as above.
Then, we have identifications
\[
\mathcal F^\beta(X,Y)
= C^\infty(\widetilde X,\widetilde Y_{\widetilde\beta})
= C^\infty(\widetilde X^\beta_{\sparallel},\widetilde Y)
= C^\infty(\widetilde X^\beta,\widetilde Y_{\Gamma^\beta_X\times\Gamma_Y})
\]
as subsets of $C^\infty(X,Y)$.
\end{lemma}

\begin{corollary}
\label{cor:edgefunc-Baire}
Let $X$, $Y$, and $\beta$ be as in \Cref{lem:edge-alongLan}.
Then, the space $\mathcal F^\beta(X,Y)$ is a Baire space.
\end{corollary}

\begin{remark}
If $\mathcal X$ is an excellent arrangement of manifolds of shape $S$, and if $\beta$ is an edging of $\mathcal X$ with $Y$, then we can also define the subspace $\mathcal F^\beta(\mathcal X,Y)\subset C^\infty(\mathcal X,Y_S)$ of maps along $\beta$, where $Y_S$ is the constant arrangement of shape $S$ at $Y$.
We have a similar formula for $\mathcal F^\beta(\mathcal X,Y)$ to \Cref{lem:edge-alongLan}.
More generally, one can verify $\mathcal F^\beta(\mathcal X,Y)=\mathcal F^\beta(|\mathcal X|,Y)$ under the identification $C^\infty(\mathcal X,Y_S)=C^\infty(|\mathcal X|,Y)$.
Thus, most of the results on the space $\mathcal F^\beta(X,Y)$ still hold for $\mathcal F^\beta(\mathcal X,Y)$ for general excellent arrangements $\mathcal X$.
That is why we only discuss $\mathcal F^\beta(X,Y)$ in this section.
\end{remark}

The subspace $\mathcal F^\beta(X,Y)$ admits some stabilities under categorical operations.

\begin{lemma}
\label{lem:edgefunc-comp}
Let $X$, $Y$, and $Z$ be manifolds with finite faces.
Suppose $\beta$ and $\gamma$ be edgings of $X$ and $Y$ with $Y$ and $Z$ respectively, so that we obtain an edging $\gamma\beta$ of $X$ with $Z$ (\Cref{ex:edge-comp})
If $F:X\to Y$ and $G:Y\to Z$ are along $\beta$ and $\gamma$ respectively, then the composition $GF$ is along $\gamma\beta$.
In other words, the composition gives rise to a map
\[
\mathcal F^\gamma(Y,Z)\times\mathcal F^\beta(X,Y)\to\mathcal F^{\gamma\beta}(X,Z)\,.
\]
Moreover, each map $G\in\mathcal F^\gamma(Y,Z)$ and each proper map $F\in\mathcal F^\beta(X,Y)$ induce continuous maps
\[
\begin{gathered}
\mathcal F^\beta(X,Y)\to\mathcal F^{\gamma\beta}(X,Z)
\ ;\quad F'\mapsto GF' \\
\mathcal F^\gamma(Y,Z)\to\mathcal F^{\gamma\beta}(X,Z)
\ ;\quad G'\mapsto G'F\,.
\end{gathered}
\]
\end{lemma}
\begin{proof}
The result directly follows from \Cref{lem:edge-alongLan} and \Cref{prop:whit-comp}.
\end{proof}

\begin{remark}
The first map in \Cref{lem:edgefunc-comp} is not continuous in general (with respect to the Whitney $C^\infty$-topology).
This is because $C^\infty(X,Y)$ is not topologically functorial in the first variable.
\end{remark}

\begin{lemma}
\label{lem:edgefunc-prod}
Let $X$ be a manifold with finite faces.
Suppose we have two edgings $\beta_1$ and $\beta_2$ of $X$ with $Y_1$ and $Y_2$ respectively such that $D(\beta_1)\cap D(\beta_2)=\varnothing\subset\operatorname{bd} X$, so we have an edging $\beta_1\amalg\beta_2$ of $X$ with $Y_1\times Y_2$ (see \Cref{ex:edge-paste}).
Then, we have a homeomorphism
\[
\mathcal F^{\beta_1\amalg\beta_2}(X,Y_1\times Y_2)
\cong\mathcal F^{\beta_1}(X,Y_1)\times\mathcal F^{\beta_2}(X,Y_2)\,.
\]
\end{lemma}
\begin{proof}
We show the hemoemorphism in \Cref{cor:Cinf-cartprod} restricts to the required one.
Write $\pi_i:Y_1\times Y_2\to Y_i$ the projection for $i=1,2$, then it suffices to show that a map $F:X\to Y_1\times Y_2$ is along $\beta_1\amalg\beta_2$ if and only if the compositions $\pi_1F$ and $\pi_2F$ are along $\beta_1$ and $\beta_2$ respectively.
This follows from the following formula:
\[
\operatorname{bd}(Y_1\times Y_2)
= \{\pi_1^{-1}(D)\mid D\in\operatorname{bd}Y_1\}
\amalg \{\pi_2^{-1}(D')\mid D'\in\operatorname{bd}Y_2\}\,.
\]
\end{proof}

The two results above can be used to construct new maps from olds.
We do not, however, even know whether the space $\mathcal F^\beta(X,Y)$ is empty or not.
It is actually a non-trivial question, and, unfortunately, the space is often empty.
In fact, in order for $\mathcal F^\beta(X,Y)$ to be nonempty, it is necessary that there exists a map $\varphi:\pi_0X\to\pi_0Y$ which commutes the following diagram:
\[
\xymatrix@C=1em{
  \pi_0(\partial_1X) \ar@{=}[r]^-\sim \ar[d] & \operatorname{bd}X \ar[rr]^{\beta}\supset D(\beta) && \operatorname{bd}Y \ar@{=}[r]^-\sim & \pi_0(\partial_1Y) \ar[d] \\
  \pi_0 X \ar@{-->}[rrrr]^\varphi &&&& \pi_0Y }
\]
Using this, one can easily construct examples of $\beta$ so that $\mathcal F^\beta(X,Y)$ is empty.
Note that we want to investigate $\mathcal F^\beta(X,Y)$ to obtain geometric information of $X$, and we are only interested in those $Y$ which is not too complicated.
The good news is that we can show $\mathcal F^\beta(X,Y)$ is nonempty in such cases.

\begin{proposition}
\label{prop:edgefunc-polyhedra}
Let $K$ be a (possibly non-compact) manifold with faces which can be embedded into a Euclidean space $\mathbb R^n$ as a convex polyhedron; i.e. a (not necessarily unbounded) convex subset of $\mathbb R^n$ bounded by finite hyperplanes.
Then, for every manifold $X$ with finite faces, and for every edging $\beta$ of $X$ with $K$, the space $\mathcal F^\beta(X,K)$ is non-empty.
\end{proposition}
\begin{proof}
By the assumption, we may think of $K\subset\mathbb R^n$ as a convex polyhedron.
A good property is that, for every $\alpha\in\Gamma_K$, $\overline\partial_\alpha K\subset\mathbb R^n$ is, unless empty, again a convex polyhedron.

Suppose $\beta$ is an edging of $X$ with $K$.
Set $\Gamma_X^+\subset\Gamma_X$ to be the subposet consisting of elements $\sigma\in\Gamma_X$ with $\overline\partial_\sigma X\neq\varnothing$.
We denote by $\min(\Gamma^+_X)\subset\Gamma^+_X$ the set of minimal elements.
Then, we can choose a map $v:\min(\Gamma^+_X)\to K$ so that $v(\sigma)\in\overline\partial_{\widetilde\beta(\sigma)} K$.
For each $\sigma\in\min(\Gamma^+_X)$, we set
\[
U_\sigma := X\setminus\bigcup_{\sigma'\not\ge\sigma}\overline\partial_{\sigma'} X\,,
\]
then $U_\sigma$ is an open neighborhood of $\overline\partial_\sigma X$ in $X$.
Note that $\{U_\sigma\mid\sigma\in\min(\Gamma^+_X)\}$ gives an open covering of $X$, so that there is a partition $\{\rho_\sigma\}_\sigma$ of unity subordinate to it.
We define a smooth map $F:X\to K\subset\mathbb R^n$ by the formula
\[
F(p) := \sum_{\sigma\in\min(\Gamma^+_X)} \rho_\sigma(p)v(\sigma)\,.
\]
Then $F$ is well-defined since $K\subset\mathbb R^n$ is convex.
Moreover, for $\sigma,\sigma_1,\dots,\sigma_r\in\min(\Gamma^+_X)$, we have
\[
\overline\partial_{\sigma_1\vee\dots\vee\sigma_r} X\cap U_\sigma \neq \varnothing
\]
only when $\sigma\in\{\sigma_1,\dots,\sigma_r\}$.
In other words, in the other case, $\rho_\sigma$ is identically zero on $\overline\partial_{\sigma_1\vee\dots\vee\sigma_r} X$, and we can conclude that $F$ maps $\overline\partial_{\sigma_1\vee\dots\vee\sigma_r} X$ into the convex hull in $\mathbb R^n$ spanned by $v(\sigma_1),\dots,v(\sigma_r)$.
Thus, we obtain
\[
F(\overline\partial_{\sigma_1\vee\dots\vee\sigma_r} X)
\subset \operatorname{conv}\{v(\sigma_1),\dots,v(\sigma_r)\}
\subset \overline\partial_{\widetilde\beta(\sigma_1)\vee\dots\vee\widetilde\beta(\sigma_r)} K\,.
\]
This implies $F$ is along $\beta$ and $F\in\mathcal F^\beta(X,K)$.
\end{proof}

\subsection{Collarings of arrangements}
\label{sec:collar}

In this section, we will prove a relative analogue of a classical result: the Collar Neighborhood Theorem.
Let $\mathcal X$ be a neat arrangement of shape $S$ equipped with an edging $\beta:\operatorname{bd}\mathcal X\to\operatorname{bd}Y$ with $Y$.
By \Cref{prop:edge-excel}, we have an excellent arrangement $\widetilde{\mathcal X}^\beta_{\sparallel}:S\times\Gamma_Y\to\mathbf{Emb}$.
For each $\tau\in\Gamma_Y$, it restricts to an excellent arrangement
\[
\overline\partial^\beta_\tau\mathcal X:S\to\mathbf{Emb}
\ ;\quad s\mapsto\widetilde{\mathcal X}^\beta_{\sparallel}(s,\tau)=\mathcal X(s)\cap\overline\partial^\beta_\tau|\mathcal X|\,.
\]
One can easily veirfy that $\overline\partial^\beta_\tau\mathcal X$ is a neat arrangement.
Then, the relative version of the Collar Neighborhood Theorem will, roughly, assert that the neat arrangement $\mathcal X$ is of the form $\mathbb R^{c(\tau)}_+\times\overline\partial^\beta_\tau\mathcal X$ near $\overline\partial^\beta_\tau|\mathcal X|$ as arrangements.
The goal of this section is to prove this theorem in a more precise form.
Note that some of our proofs are base on those in the paper \cite{Lau00}.

First, we see that some special vector fields give rise to ``collars'' around faces.

\begin{definition}
Let $\mathcal X$ be a neat arrangement of manifolds of shape $S$, and let $C\in\operatorname{bd}\mathcal X$ be a connected face of the ambient manifold.
Then a vector field $\xi$ on $\mathcal X$ is said to be $C$-collaring if $\xi$ is inward-pointing, and we have $\xi\pitchfork C$ and $\xi\parallel C'$ for $C'\in\operatorname{bd}\mathcal X$ with $C'\neq C$.
\end{definition}

If $\xi$ is a $C$-collaring vector field, we obtain an open embedding
\[
\mathbb R_+\times \overline\partial_C\mathcal X\to\mathcal X
\]
using \Cref{prop:vecfield-arr-flow}.
Hence, construction of collars can be reduced to construction of collaring vector fields.
We have the following result:

\begin{lemma}
\label{lem:Ccol-vecfield}
Let $\mathcal X$ be a neat arrangement of manifolds of shape $S$.
Then, for every connected face $C$ of $|\mathcal X|$, there is a $C$-collaring vector fields on $\mathcal X$.
\end{lemma}
\begin{proof}
Take an family $\{(U_\alpha,\mathcal I_\alpha,\varphi_\alpha)\}_{\alpha\in A}$ of $\mathcal X$-charts so that $\{U_\alpha\}_\alpha$ covers $C$, $|\mathcal I_\alpha|$ is of the form $\langle m|k_\alpha\rangle$, and we have the following pullback square:
\begin{equation}
\label{sq:prf:Ccol-vecfield:C}
\xymatrix{
  U_\alpha\cap C \ar[r] \ar[d] \ar@{}[dr]|(.4)\pbcorner & U_\alpha \ar[d]^{\varphi_\alpha} \\
  \{x_m=0\} \ar[r] & \mathbb H^{|\mathcal I_\alpha|} }
\end{equation}
We construct a vector field $\xi_\alpha$ on each $U_\alpha$ as the pullback of the vector field $\frac\partial{\partial x_m}$ on $\mathbb H^{|\mathcal I_\alpha|}$ by $\varphi_\alpha$.
Note that since we have $\mathcal X(s)\pitchfork C$ and $(U_\alpha,\mathcal I_\alpha,\varphi_\alpha)$ is an $\mathcal X$-chart, we obtain that $\frac\partial{\partial x_m}$ is a vector field on $\mathcal E^{\mathcal I_\alpha}$.
Thus, $\xi_\alpha$ is actually a vector field on the arrangement $U_\alpha\cap\mathcal X$.
Choose a partition of unity $\{\chi\}\cup\{\rho_\alpha\}_{\alpha\in A'}$ subordinate to $\{X\setminus C\}\cup\{U_\alpha\}_\alpha$, where $A'\subset A$, and define
\[
\xi := \sum_{\alpha\in A'} \rho_\alpha\cdot \xi_\alpha\,.
\]
It is clearly that $\xi$ is an inward-pointing vector field on $\mathcal X$ such that $\xi\pitchfork C$.
Moreover, for $C'\in\operatorname{bd}\mathcal X$ with $C'\neq C$, we have $\xi_\alpha\parallel (U_\alpha\cap C')$ by \eqref{sq:prf:Ccol-vecfield:C}.
Therefore, we conclude that $\xi$ is a $C$-collaring vector field.
\end{proof}

Actually, we can formulate our collaring theorem on neat arrangements in more sophisticated way.
Let $\mathcal X$ be a neat arrangement of shape $S$ equipped with an edging $\beta$ with $Y$.
Recall that elements of $\Gamma_Y$ are subsets of $\operatorname{bd}Y$, and $\Gamma_Y$ is ordered by the opposite of the inclusions.
For $\tau,\tau'\in\Gamma_Y$, we define
\[
\mathcal C_Y(\tau,\tau')
:=
\begin{cases}
\mathbb R^{\tau'\setminus\tau}_+ \quad&\text{if $\tau\le\tau'\in\Gamma_Y$,} \\
\varnothing \quad&\text{otherwise.}
\end{cases}
\]
Then, we obtain $\mathcal C_Y(\tau,\tau) = \{0\}$ and, for $\tau,\tau',\tau''\in\Gamma_Y$, a canonical map
\[
\mathcal C_Y(\tau',\tau'')\times\mathcal C_Y(\tau,\tau')
\to\mathcal C_Y(\tau,\tau'')
\]
which is a diffeomorphism when $\tau\le\tau'\le\tau''$.
It follows that the functor $\mathcal C_Y:\Gamma_Y^\opposite\times\Gamma_Y\to\mathbf{Mfd}$ gives rise to an enrichment of the category $\Gamma_Y$ over $\mathbf{Mfd}$ (with respect to the cartesian product).
Note that, by \Cref{cor:mfdface-uniquecap}, for $\tau\le\tau'\in\Gamma_Y$, we have
\[
\dim (\mathcal C_{\mathcal X}(\tau,\tau')\times\overline\partial^\beta_\tau\mathcal X(s))
= \dim \overline\partial^\beta_{\tau'}\mathcal X(s)
\]
for every $s\in S$.

\begin{definition}
Let $\mathcal X$ be a neat arrangement of shape $S$, and let $\beta$ be an edging of $\mathcal X$ with a manifold $Y$ with finite faces.
Then a $\beta$-collaring $\gamma$ of $\mathcal X$ is a family of open embeddings
\[
\gamma^\tau_{\tau'}:\mathcal C_Y(\tau,\tau')\times\overline\partial_{\tau}\mathcal X\to\overline\partial_{\tau'}\mathcal X
\]
of arrangements for $\tau,\tau'\in\Gamma_Y$ such that, for each chain $\tau\le\tau'\le\tau''\in\Gamma_Y$, the following square is commutative:
\[
\xymatrix{
  \mathcal C_Y(\tau',\tau'')\times\mathcal C_Y(\tau,\tau')\times\overline\partial^\beta_\tau\mathcal X \ar[r]^-\simeq \ar[d]_{\mathrm{id}\times\gamma^\tau_{\tau'}} & \mathcal C_Y(\tau,\tau'')\times\overline\partial^\beta_\tau\mathcal X \ar[d]^{\gamma^\tau_{\tau''}} \\
  \mathcal C_Y(\tau',\tau'')\times\overline\partial^\beta_{\tau'}\mathcal X \ar[r]^{\gamma^{\tau'}_{\tau''}} & \overline\partial^\beta_{\tau''}\mathcal X }
\]
\end{definition}

\begin{remark}
In view that $\mathcal C_Y$ is an enriched category over $\mathbf{Mfd}$, a $\beta$-collaring is a left actions of $\mathcal C_Y$ to the family $\{\overline\partial^\beta_\tau\mathcal X\}_{\tau\in\Gamma_Y}$ of arrangements.
Hence, we will often omit the indices and write
\[
\gamma=\gamma^\tau_{\tau'}:
\mathcal C_Y(\tau,\tau')\times\overline\partial^\beta_\tau\mathcal X
\to \overline\partial^\beta_{\tau'}\mathcal X
\]
for a $\beta$-collaring $\gamma$.
\end{remark}

If $\gamma$ is a $\beta$-collaring of $\mathcal X$, then for each $\tau\in\Gamma_Y$, we have an open embedding
\[
\mathbb R^{c(\tau)}_+\times\overline\partial^\beta_\sigma\mathcal X
\simeq \mathcal C_Y(\tau,Y)\times\overline\partial^\beta_\tau\mathcal X
\xrightarrow{\gamma} \mathcal X\,,
\]
of arrangements of manifolds.
This roughly says that a $\beta$-collaring gives an open neighborhood diffeomorphic to a product with a half open interval for each connected face respecting the arrangement.

The next is the main theorem in this section:

\begin{theorem}[Collar Neighborhood Theorem]
\label{thm:collaring}
Every neat arrangement $\mathcal X$ admits a $\beta$-collaring for every edging $\beta$ of $\mathcal X$.
\end{theorem}

For the proof of \Cref{thm:collaring}, we essentially follow \cite{Lau00} though it does not contains construction of vector fields.
We use some results on collaring vector fields.

\begin{lemma}
\label{lem:colvec-ext}
Let $\mathcal X$ be a neat arrangement of manifolds.
Suppose we have a subset $\mathcal A\subset\operatorname{bd}\mathcal X$ together with a family $\{\xi_C\}_{C\in\mathcal A}$ such that
\begin{enumerate}[label={\rm(\alph*)}]
  \item\label{ass:colvec-ext:colv} $\xi_C$ is a $C$-collaring vector field on $\mathcal X$;
  \item\label{ass:colvec-ext:lie0} $[\xi_C,\xi_{C'}]\equiv 0$ on a neighborhood of $C\cap C'$ for each $C,C'\in\mathcal A$.
\end{enumerate}
Then, each connected face $D\in\operatorname{bd}\mathcal X\setminus\mathcal A$ admits a $D$-collaring vector field $\zeta$ such that
\[
[\zeta,\xi_C]\equiv0
\]
on a neighborhood of $C$ for each $C\in\mathcal A$.
\end{lemma}
\begin{proof}
Let $\{\varphi^C_t\}_{t\ge0}$ be the one-parameter family of $\xi_C$ for each $C\in\mathcal A$, and we put
\[
\eta^C:\mathbb R_+\times C \to |\mathcal X|
\ \;\quad (t;p) \mapsto \varphi^C_t(p)\,,
\]
which is smooth and, by \Cref{lem:vecfield-flow-emb}, actually defines an open embedding $\mathbb R_+\times\overline\partial_C\mathcal X\to\mathcal X$ of arrangements over $S$.
Set $V_{C,\varepsilon}:=\eta^C([0,\varepsilon)\times C)$ for each positive number $\varepsilon>0$.
The condition \ref{ass:colvec-ext:lie0} implies that, for a sufficiently small $\varepsilon>0$, we have
\[
\operatorname{supp}{[\xi_C,\xi_{C'}]} \cap\overline V_{C,\varepsilon}\cap\overline V_{C',\varepsilon} = \varnothing\,.
\]
In particular, for every $p\in C\cap C'$ and $0\le t,t'<\varepsilon$, we have
\begin{equation}
\label{eq:prf:colvec-ext:vcom}
\varphi^C_t\circ\varphi^{C'}_{t'}(p) = \varphi^{C'}_{t'}\circ\varphi^C_t(p)\,.
\end{equation}
We denote by $\Gamma^{\mathcal A}_{\mathcal X}$ the sublattice of $\Gamma_{\mathcal X}$ generated by $\mathcal A\subset\operatorname{bd}\mathcal X$.
For $\sigma\le\sigma'\in\Gamma^{\mathcal A}_{\mathcal X}$, put
\[
\mathcal C^\varepsilon_{\mathcal X}(\sigma,\sigma') := [0,\varepsilon)^{\sigma'\setminus\sigma}\,.
\]
Say $\sigma'\setminus\sigma = \{C_1,\dots,C_r\}$, we define
\[
\eta^\sigma_{\sigma'}:\mathcal C^\varepsilon_{\mathcal X}(\sigma,\sigma')\times\overline\partial_\sigma|\mathcal X|
\to \overline\partial_{\sigma'}|\mathcal X|
\ ;\quad (t_1,\dots,t_r;p) \mapsto \varphi^{C_1}_{t_1}\circ\dots\circ\varphi^{C_r}_{t_r}(p)\,.
\]
Clearly, this defines an open embedding $\mathcal C^\varepsilon_{\mathcal X}(\sigma,\sigma')\times\overline\partial_\sigma\mathcal X\to\overline\partial_{\sigma'}\mathcal X$ of arrangements.
Notice that, thanks to \eqref{eq:prf:colvec-ext:vcom}, the defining formula of $\eta^\sigma_{\sigma'}$ does not depend on the choice of the order of elements of $\sigma'\setminus\sigma$.
Thus, if $\sigma\le\sigma'\le\sigma''\in\Gamma^{\mathcal A}_{\mathcal X}$, the following diagram of arrangements is commutative:
\begin{equation}
\label{diag:prf:colvec-ext:epscol}
\vcenter{
  \xymatrix{
    \mathcal C^\varepsilon_{\mathcal X}(\sigma',\sigma'')\times\mathcal C^\varepsilon_{\mathcal X}(\sigma,\sigma')\times\overline\partial_\sigma\mathcal X \ar@{=}[r] \ar[d]_{\mathrm{id}\times\eta^\sigma_{\sigma'}} & \mathcal C^\varepsilon_{\mathcal X}(\sigma,\sigma'')\times\overline\partial_\sigma\mathcal X \ar[d]^{\eta^\sigma_{\sigma''}} \\
    \mathcal C^\varepsilon_{\mathcal X}(\sigma',\sigma'')\times\overline\partial_{\sigma'}\mathcal X \ar[r]^{\eta^{\sigma'}_{\sigma''}} & \overline\partial_{\sigma''}\mathcal X }}
\end{equation}
We write $\eta^\sigma:=\eta^\sigma_{|\mathcal X|}$ and $V_{\sigma,\varepsilon}:=\eta^\sigma\left(\mathcal C^\varepsilon_{\mathcal X}(\sigma,|\mathcal X|)\times|\overline\partial_\sigma\mathcal X|\right)$, so \eqref{diag:prf:colvec-ext:epscol} implies
\begin{equation}
\label{eq:prf:colvec-ext:VCepsl}
V_{(C_1\cap\dots\cap C_r),\varepsilon} = V_{C_1,\varepsilon}\cap\dots\cap V_{C_r,\varepsilon}\,.
\end{equation}

Put $m=\dim|\mathcal X|$.
To prove the result, we construct a sequence $\{\zeta^{(d)}\}_{d=0}^m$ of $D$-collaring vector fields such that for each $\sigma\in\Gamma^{\mathcal A}_{\mathcal X}$ with $\dim \overline\partial_\sigma|\mathcal X|< d$ and each $C\in\sigma$, we have
\begin{equation}
\label{eq:prf:colvec-ext:dcomm}
[\zeta^{(d)},\xi_C]\equiv 0
\end{equation}
on a neighborhood of $\overline\partial_\sigma|\mathcal X|\subset|\mathcal X|$.
We shall construct $\{\zeta^{(d)}\}$ inductively on $d$; begining with an arbitrary $D$-collaring vector field $\xi^{(0)}$ taken by using \Cref{lem:Ccol-vecfield}.
Suppose we already have $\zeta^{(d-1)}$.
The construction of $\zeta^{(d)}$ is divided into some parts.
We denote by $\mathcal A^{(k)}$ the subset of $\Gamma^{\mathcal A}_{\mathcal X}$ consisting of $\sigma\in\Gamma^{\mathcal A}_{\mathcal X}$ with $\dim\overline\partial_\sigma|\mathcal X|= k-1$.

\emph{Step1:}
The induction hypothesis implies that, taking a sufficiently small $\varepsilon>0$, we have
\[
\operatorname{supp}{[\zeta^{(d-1)},\xi_C]}\cap\overline V_{\rho,\varepsilon} = \varnothing
\]
for each $\rho\in\mathcal A^{(d-1)}$ and $C\in\rho$.
This implies that that, for $\sigma\in\mathcal A^{(d)}$, we have
\begin{equation}
\label{eq:prf:colvec-ext:Wsigma}
\operatorname{supp}{[\zeta^{(d)},\xi_C]}\cap V_{\sigma,\varepsilon}
\subset V_{\sigma,\varepsilon}\setminus\bigcup_{C'\in\mathcal A\setminus\sigma} \overline V_{C',\varepsilon} =: W_\sigma\,.
\end{equation}
The construction immediately implies $W_\sigma\cap W_{\sigma'}=\varnothing$ whenever $\sigma\neq\sigma'$ for $\sigma,\sigma'\in\Gamma^{\mathcal A}_{\mathcal X}$.
Note that, by the square \eqref{diag:prf:colvec-ext:epscol}, putting $T_\sigma := \overline\partial_\sigma|\mathcal X|\cap W_\sigma$, we obtain a diffeomorphism
\begin{equation}
\label{eq:prf:colvec-ext:eta}
\mathbb R^{c(\sigma)}_+\times T_\sigma
\cong \mathcal C^\varepsilon_{\mathcal X}(\sigma,|\mathcal X|)\times T_\sigma
\xrightarrow{\substack{\eta^\sigma\cr\smash\simeq}} W_\sigma\,.
\end{equation}

\emph{Step2:}
We define a vector field $\zeta^\sigma$ on $W_\sigma$ so that the following square is commutative.
\[
\xymatrix@C=6em{
  \mathcal C^\varepsilon_{\mathcal X}(\sigma,|\mathcal X|)\times T_\sigma \ar[r]^-{0\oplus\zeta^{(d-1)}|_{T_\sigma}} \ar[d]^\simeq_{\eta^\sigma} & T\mathcal C^\varepsilon_{\mathcal X}(\sigma,|\mathcal X|)\oplus TT_\sigma \ar[d]^{\eta^\sigma_\ast}_\simeq \\
  W_\sigma \ar[r]^{\zeta^\sigma} & TW_\sigma }
\]
It follows from the construction of $\eta^\sigma$ that this defines a $D\cap W_\sigma$-collaring vector field on the arrangement $\mathcal X\cap W_\sigma$.
Moreover, with respect to the coordinate $((t_C)_C;p)\in \mathcal C^\varepsilon_{\mathcal X}(\sigma,|\mathcal X|)\times T_\sigma$, we have
\begin{equation}
\label{eq:prf:colvec-ext:zetasigma}
[\zeta^\sigma,\xi_C] = \bigr[\zeta|_{T_\sigma},\frac\partial{\partial t_C}\bigl] = 0
\end{equation}
on $W_\sigma$.

\emph{Step3:}
Define a smooth function $\lambda_\sigma:|\mathcal X|\to [0,1]$ as follows:
Since $|\mathcal X|$ is compact, using the diffeomorphism \eqref{eq:prf:colvec-ext:eta}, we can take a compact subset $K_\sigma\subset T_\sigma$ so that
\[
\bigcup_{C\in\mathcal A\setminus\sigma} \operatorname{supp}{[\zeta^{(d-1)},\xi_C]} \cap \overline V_{\sigma,\varepsilon/2}
\subset \eta^\sigma(\mathcal C^\varepsilon_{\mathcal X}(\sigma,|\mathcal X|)\times K_\sigma)\,.
\]
Take smooth funcitons $\lambda'_\sigma:T_\sigma\to[0,1]$ and $\lambda''_\varepsilon:[0,\varepsilon)\to[0,1]$ with compact supports such that
\[
\lambda'_\sigma \equiv 1\ \text{on}\ K_\sigma
\ ,\quad \lambda''_\varepsilon\equiv1\ \text{on}\ [0,\frac\varepsilon2]\,.
\]
Then, we set
\[
\lambda_\sigma(p) :=
\begin{cases}
\lambda'_\sigma(p')\prod_{C\in\sigma}\lambda''_\varepsilon(t_C) \quad& \text{if}\ p=\eta^\sigma((t_C)_C;p') \\
0 \quad& \text{otherwise}\,.
\end{cases}
\]
By construction, we have $\lambda_\sigma\equiv 1$ on $\operatorname{supp}{[\zeta^{(d-1)},\xi_C]}\cap\overline V_{\sigma,\varepsilon/2}$ for each $C\in\sigma$ and
\begin{equation}
\xi_C\lambda_\sigma(\eta^\sigma((t_{C'})_{C'};p)
= \frac{d\lambda''_\varepsilon}{dt}(t_C)\lambda'_\sigma(p')\prod_{C'\neq C}\lambda''_\varepsilon(t_{C'})
\end{equation}
for $C\in\sigma$, which is identically zero on $T_\sigma\times[0,\varepsilon/2]$.
In other words, $\xi_C\lambda_\sigma\equiv 0$ on $V_{\sigma,\varepsilon/2}$.

\emph{Step4:}
We finally define $\zeta^{(d)}$ by the following formula:
\[
\zeta^{(d)}
:= \left(1-\sum_{\sigma\in\mathcal A^{(d)}}\lambda_\sigma\right)\zeta^{(d-1)} + \sum_{\sigma\in\mathcal A^{(d)}}\lambda_\sigma\zeta^\sigma\,.
\]
Since supports of $\lambda_\sigma$ are all disjoint, $\zeta^{(d)}$ is in fact $D$-collaring vector field on the arrangement $\mathcal X$.
Moreover, on the open subset $V_{\sigma,\varepsilon/2}$, $\zeta^{(d)}$ agrees with the vector field
\[
(1-\lambda_\sigma)\zeta^{(d-1)} + \lambda_\sigma\zeta^\sigma\,.
\]
On the other hand, we have
\begin{equation}
\label{eq:prf:colvec-ext:zetacom}
\begin{split}
&[(1-\lambda_\sigma)\zeta^{(d-1)}+\lambda_\sigma\zeta^\sigma,\xi_C] \\
&= (1-\lambda_\sigma)[\zeta^{(d-1)},\xi_C] + \lambda_\sigma[\zeta^\sigma,\xi_C] + (\xi_C\lambda_\sigma)(\zeta^{(d-1)} - \zeta^\sigma)\,.
\end{split}
\end{equation}
By \eqref{eq:prf:colvec-ext:zetasigma} and the properties of $\lambda_\sigma$, all the three terms in \eqref{eq:prf:colvec-ext:zetacom} vanish on $V_{\sigma,\varepsilon/2}$.
Thus, we conclude that $\zeta^{(d)}$ has a required vector field.
\end{proof}

\begin{proposition}
\label{prop:colvec}
Let $\mathcal X$ be a neat arrangement of manifolds.
Suppose we have a subset $\mathcal A\subset\operatorname{bd}\mathcal X$ together with a family $\{\xi'_C\}_{C\in\mathcal A}$ such that
\begin{enumerate}[label={\rm(\alph*)}]
  \item\label{ass:colvec:colv} $\xi'_C$ is a $C$-collaring vector field on $\mathcal X$;
  \item\label{ass:colvec:lie0} $[\xi'_C,\xi'_{C'}]\equiv 0$ on a neighborhood of $C\cap C'$ for each $C,C'\in\mathcal A$.
\end{enumerate}
Then, there is a family $\{\xi_C\}_{C\in\operatorname{bd}\mathcal X}$ such that
\begin{enumerate}[label={\rm(\roman*)}]
  \item\label{req:colvec:colv} each $\xi_C$ is a $C$-collaring vector field on $\mathcal X$ which agrees with $\xi'_C$ near $C$ if $C\in\mathcal A$;
  \item\label{req:colvec:disj} if $C\cap C'=\varnothing$, $\operatorname{supp}\xi_C\cap\operatorname{supp}\xi_{C'}=\varnothing$;
  \item\label{req:colvec:lie0} for every $C,C'\in\operatorname{bd}\mathcal X$, $[\xi_C,\xi_{C'}]\equiv 0$ on whole $|\mathcal X|$.
\end{enumerate}
\end{proposition}
\begin{proof}
Suppose $\{\xi'_C\}_{C\in\mathcal A}$ is a family as in the assumption.
Using \Cref{lem:colvec-ext}, we can extend the family to $\{\xi'_C\}_{C\in\operatorname{bd}\mathcal X}$ indexed by all the connected faces of $|\mathcal X|$ keeping the conditions \ref{ass:colvec:colv} and \ref{ass:colvec:lie0}.
Take the one-parameter family $\{\varphi^C_t\}_t$ of the vector field $\xi'_C$ for each $C\in\operatorname{bd}\mathcal X$.
Using \Cref{lem:vecfield-flow-emb}, we obtain an open embedding
\[
\eta^C:\mathbb R_+\times\overline\partial_C\mathcal X\to \mathcal X
\ ;\quad (t;p) \mapsto \varphi^C_t(p)
\]
of arrangements.
For a positive number $\varepsilon$, we put $V_{C,\varepsilon}:=\eta^C([0,\varepsilon)\times C)$.
Then, by condition \ref{ass:colvec:lie0}, for a sufficiently small number $\varepsilon>0$, we have
\begin{equation}
\label{eq:prf:colvec:suppVV}
\operatorname{supp}{[\xi'_C,\xi'_{C'}]}\cap\overline V_{C,\varepsilon}\cap\overline V_{C',\varepsilon} = \varnothing\,.
\end{equation}
Take a smooth function $\rho:[0,\varepsilon)\to[0,1]$ which values identically $0$ and $1$ on $[0,\varepsilon/3]$ and $[2\varepsilon/3,\varepsilon)$ respectively.
We define a function $\rho^C$ by
\[
\rho^C(p)
:= \begin{cases}
\rho(t) \quad&\text{if}\ p=\eta^C(t;p_0)\in V_{C,\varepsilon} \\
0 &\text{otherwise,}
\end{cases}
\]
which is claerly smooth.
For $C,C'\in\operatorname{bd}\mathcal X$, by virtue of \eqref{eq:prf:colvec:suppVV}, we have
\[
\begin{split}
[\xi'_{C'}(\rho^C)](\eta_C(t;p))
&= \left.\frac{d}{ds}\right|_{s=0} \rho^C\varphi^{C'}_s\varphi^C_t(p) \\
&= \left.\frac{d}{ds}\right|_{s=0} \rho^C\varphi^C_t\varphi^{C'}_s(p) \\
&= \left.\frac{d}{ds}\right|_{s=0} \rho(t) \\
&= 0
\end{split}
\]
whenever $\eta_C(p,t)\in V_{C,\varepsilon}\cap V_{C',\varepsilon}$.
In other words, we obtain
\begin{equation}
\label{eq:prf:colvec:xirho}
\xi'_{C'}(\rho^C)\equiv 0\ \text{on}\ V_{C,\varepsilon}\cap V_{C',\varepsilon}\,.
\end{equation}
Finally, we define $\xi_C$ by
\[
\xi_C := \rho^C\cdot\xi'_C
\]
for each $C\in\operatorname{bd}\mathcal X$.
Clearly $\xi_C$ is a $C$-collaring vector field which agrees with $\xi'_C$ on $V_{C,\varepsilon/3}$.
Moreover, since the support of $\xi_C$ is contained in $V_{\sigma,\varepsilon}$, the equations \eqref{eq:prf:colvec:suppVV} and \eqref{eq:prf:colvec:xirho} implies the properties \ref{req:colvec:disj} and \ref{req:colvec:lie0} for sufficiently small $\varepsilon>0$.
Therefore, $\{\xi_C\}_C$ is in fact a required family.
\end{proof}

\begin{proof}[Proof of \Cref{thm:collaring}]
Let $\mathcal X$ be a neat arrangement of shape $S$, and let $\beta$ is an edging of $\mathcal X$ with $Y$.
Apply \Cref{prop:colvec} to $\mathcal A=\varnothing\subset\operatorname{bd}\mathcal X$, we obtain a family $\{\xi_C\}_C$ satisfying the conditions \ref{req:colvec:colv}, \ref{req:colvec:disj}, and \ref{req:colvec:lie0} in the proposition.
For a connected face $D\in\Gamma_Y$, we define
\[
\xi^\beta_D:=\sum_{\substack{C\in D(\beta)\cr\beta(C)=D}} \xi_C\,.
\]
Since connected faces $C$ with $C\in D(\beta)$ and $\beta(C)=D$ are disjoint, the vector field $\xi^\beta_D$ satisfies the following properties:
\begin{enumerate}[label={\rm(\roman*')}]
  \item\label{cond:collaring:colv} $\xi^\beta_D$ is an inward-pointing vector field on the arrangement $\mathcal X$ so that $\xi^\beta_D\pitchfork\overline\partial^\beta_D|\mathcal X|$ and $\xi^\beta_D\parallel C$ for $C\in\operatorname{bd}\mathcal X\setminus\beta^{-1}\{D\}$;
  \item\label{cond:collaring:lie0} $[\xi^\beta_D,\xi^\beta_{D'}]\equiv 0$ on whole $|\mathcal X|$ for every $D,D'\in\operatorname{bd}Y$.
\end{enumerate}
Let $\{\varphi^D_t\}_t$ be the one-parameter family associated with $\xi^\beta_D$.
Then, for $\tau\le\tau'\in\Gamma_Y$, say $\tau'\setminus\tau=\{D_1,\dots,D_r\}$, we define
\[
\gamma^\tau_{\tau'}:\mathcal C_Y(\tau,\tau')\times\overline\partial^\beta_\tau|\mathcal X|
\to \overline\partial^\beta_{\tau'}|\mathcal X|
\ ;\quad (t_1,\dots,t_r;p) \mapsto \varphi^{D_1}_{t_1}\dots\varphi^{D_r}_{t_r}(p)\,.
\]
By the property \ref{cond:collaring:colv}, $\gamma^\tau_{\tau'}$ actually defines an open embedding
\[
\gamma^\tau_{\tau'}:\mathcal C_Y(\tau,\tau')\times\overline\partial^\beta_\tau\mathcal X
\to\overline\partial^\beta_{\tau'}\mathcal X
\]
of arrangements.
Moreover, the property \ref{cond:collaring:lie0} implies that the map $\gamma^\tau_{\tau'}$ is independent of the choice of orderings on the set $\tau'\setminus\tau$.
This guarantees that, for $\tau\le\tau'\le\tau''\in\Gamma_Y$, the following square is commutative:
\[
\xymatrix{
  \mathcal C_Y(\tau',\tau'')\times\mathcal C_Y(\tau,\tau')\times\overline\partial^\beta_\tau|\mathcal X| \ar[r]^-\cong \ar[d]_{\mathrm{id}\times\gamma^\tau_{\tau'}} & \mathcal C_Y(\tau,\tau'')\times\overline\partial^\beta_\tau|\mathcal X| \ar[d]^{\gamma^\tau_{\tau''}} \\
  \mathcal C_Y(\tau',\tau'')\times\overline\partial^\beta_{\tau'}|\mathcal X| \ar[r]^{\gamma^{\tau'}_{\tau''}} & \overline\partial^\beta_{\tau''}|\mathcal X| }
\]
Thus, we obtain a $\beta$-collaring $\gamma=\{\gamma^\tau_{\tau'}\}_{\tau\le\tau'}$ of $\mathcal X$.
\end{proof}

To end this section, we see that collarings admit uniqueness in some sense.
The results below are essentially due to Munkres (Section 6 in \cite{Mun66}) though we give a slightly different proof, especially in the last part.

\begin{lemma}
\label{lem:collar-diffeo}
Let $\mathcal X$ be a neat arrangement of manifolds of shape $S$.
Suppose we have an open neighborhood $W\subset\mathbb R_+\times|\mathcal X|$ of $\{0\}\times|\mathcal X|$ and an open embedding $\eta:W\to\mathbb R_+\times|\mathcal X|$ which agrees with the canonical embedding on $\{0\}\times|\mathcal X|\subset W$.
Then, for every open neighborhood $V\subset W$ of $\{0\}\times|\mathcal X|$, there is a diffeomorphism $\Phi:(\mathbb R_+\times\mathcal X)\cap W\xrightarrow\sim (\mathbb R_+\times\mathcal X)\cap W$ of arrangements satisfying the following properties:
\begin{enumerate}[label={\rm(\roman*)}]
  \item The support $\operatorname{supp}\Phi$ is contained in the open subset $V\cap\eta(V)$; here $\operatorname{supp}\Phi$ is the closure of the set $\{p\in W\mid p\neq\Phi(p)\}$ in $W$.
  \item The restriction of $\Phi$ to $\{0\}\times|\mathcal X|$ is the identity.
  \item The open embedding $\eta\Phi^{-1}:W\to\mathbb R_+\times|\mathcal X|$ agrees with the canonical embedding $W\hookrightarrow\mathbb R_+\times|\mathcal X|$ on a neighrborhood of $\{0\}\times|\mathcal X|$.
\end{enumerate}
\end{lemma}
\begin{proof}
For simplicity, let us write $\mathcal W:=(\mathbb R_+\times\mathcal X)\cap W$ the arrangement.
We also write
\[
\begin{gathered}
\pi_{\mathbb R_+}:\mathbb R_+\times|\mathcal X|\twoheadrightarrow\mathbb R_+ \\
\pi_{\mathcal X}:\mathbb R_+\times|\mathcal X|\twoheadrightarrow|\mathcal X|
\end{gathered}
\]
the canonical projections.
Put $\tau:=\pi_{\mathbb R_+}\eta:W\to\mathbb R_+$ and $F:=\pi_{\mathcal X}\eta:W\to|\mathcal X|$.
Note that $F$ is actually a map $F:\mathcal W\to\mathcal X$ of arrangements.

\emph{Step1:}
We first construct a diffeomorphism $\Psi:\mathcal W\to \mathcal W$ of arrangements such that
\begin{enumerate}[label={\rm(\roman*')}]
  \item $\operatorname{supp}\Psi$ is compact and contained in the open set $\subset V\cap\eta(V)$;
  \item The restriction of $\Psi$ to $\{0\}\times|\mathcal X|$ is the identity.
  \item $\eta\Psi^{-1}(t,p)=t$ for any sufficiently small $t\ge0$ and for every $p\in|\mathcal X|$.
\end{enumerate}
Since we assumed $|\mathcal X|$ is compact, and since $\eta$ is an open embedding, we can take a positive number $\delta>0$ so that $[0,\delta]\times|\mathcal X| \subset V\cap \eta(V)$ and
\begin{equation}
\label{eq:prf:collar-diffeo:taumono}
\frac{\partial[\pi_{\mathbb R_+}\eta}{\partial t}(p,t):W\to\mathbb R
\end{equation}
is positive on $[0,\delta]\times|\mathcal X|$.
Choose a sufficiently small number $\varepsilon>0$ so that
\[
\varepsilon\left|\frac{\partial\tau}{\partial t}(t,p)\right| < 1
\]
for every $(t,p)\in[0,\delta]\times|\mathcal X|$.
By virtue of the Mean Value Theorem, this implies that
\begin{equation}
\label{eq:prf:collar-diffeo:tauepsilon}
0\le \tau(\varepsilon t,p) < t
\end{equation}
for every $(t,p)\in[0,\delta]\times|\mathcal X|$.
Choose a smooth monotonic function $\widetilde\varepsilon:\mathbb R_+\to[\varepsilon,1]$ which values identically $\varepsilon$ and $1$ on $[0,2\delta/3]$ and on $[\delta,+\infty)$ respecitvely.
We define a map $\theta:\mathbb R_+\to\mathbb R_+$ by
\[
\theta(t) := \widetilde\varepsilon(t)t\,.
\]
Then $\theta$ is a diffeomorphism with $\theta(0)=0$ which is the identity on $[\delta,+\infty)$, and we have $(\theta\times\mathrm{id})(W) = W$.
In addition, \eqref{eq:prf:collar-diffeo:tauepsilon} implies that
\begin{equation}
\label{eq:prf:collar-diffeo:thetat}
\tau(\theta(t),p) = \tau(\varepsilon t,p) < t
\end{equation}
provided $0\le t\le 2\delta/3$.
Also, choose a smooth monotonic function $\lambda:\mathbb R_+\to[0,1]$ which values identically $0$ and $1$ on $[0,\delta/3]$ and on $[2\delta/3,+\infty)$ respectively, and define
\[
H:W\to W
\ ;\quad (t,p) \mapsto ((1-\lambda(t))\tau(p,\theta(t)) + \lambda(t)t,p)\,.
\]
We have $H(t,p)=(t,p)$ for $(t,p)$ with $t\ge 2\delta/3$ and
\[
\frac{\partial[\pi_{\mathbb R_+}H]}{\partial t}(t,p)
= (1-\lambda(t))\frac{\partial\tau}{\partial t}(p,\theta(t))\frac{d\theta}{dt}(t) + \lambda(t) + \frac{d\lambda}{dt}(t)(t-\tau(p,\theta(t)))\,,
\]
which is positive if $t\le 2\delta/3$ by \eqref{eq:prf:collar-diffeo:thetat}.
It follows that $H$ is a diffeomorphism.
Now, for $t\le\delta/3$ and $p\in|\mathcal X|$, we obtain
\begin{equation}
\label{eq:prf:collar-diffeo:etaH}
\pi_{\mathbb R_+}H(t,p)
= \tau(\theta(t),p)
= \pi_{\mathbb R_+}\eta(\theta\times\mathrm{id})(t,p)\,.
\end{equation}
We set $\Psi:=(\theta\times\mathrm{id})^{-1}H:W\to W$.
Then, we obtain
\[
\begin{gathered}
\Psi|_{\{0\}\times|\mathcal X|} = \mathrm{id}
\ ,\quad \Psi|_{[\delta,+\infty)\times|\mathcal X|} = \mathrm{id} \\
\end{gathered}
\]
from the construction while \eqref{eq:prf:collar-diffeo:etaH} implies
\[
\eta\Psi^{-1}(t,p) = t
\]
for each $(t,p)\in H([0,\delta/3]\times|\mathcal X|)$.
In addition, an easy calculus shows that $\pi_{\mathcal X}\Psi = \pi_{\mathcal X}:W\to|\mathcal X|$, so $\Psi$ actually defines a diffeomorphism $\mathcal W\to\mathcal W$ of arrangements.
Thus, we conclude that $\Psi$ is a required diffeomorphism.

\emph{Step2:}
In view of the previous step, to show the result, we may assume that there is a positive number $\delta>0$ such that $[0,\delta]\times|\mathcal X|\subset V\cap\eta(V)$ and
\[
\tau(t,p) = \pi_{\mathbb R_+}\eta(t,p) = t
\]
for every $(t,p)\in[0,\delta]\times|\mathcal X|$.
Take a smooth function $\rho:\mathbb R_+\to[0,1]$ which values identically $1$ on $[0,\delta/2]$ and whose support is contained in $[0,\delta)$.
Writing $F:=\pi_{\mathcal X}\eta:W\to|\mathcal X|$, we define $\Phi:W\to W$ by
\[
\Phi(t,p) := (t,F(\rho(t)t,p))\,.
\]
Since $F$ is actually a map $\mathcal W\to\mathcal X$ of arrangements, $\Phi$ also defines a map $\Phi:\mathcal W\to\mathcal W$ of arrangements.
In addition, one can see that $\Phi$ is bijecitve and the Jacobian determinant of it at $(t,p)\in W$ equals to that of $\eta$ at $(\rho(t)t,p)$ which vanishes nowhere.
These imply that $\Phi$ is a diffeomorphism $\mathcal W\to\mathcal W$ of arrangements.
We show $\Phi$ is in fact a required one.
If $t=0$ or $t\ge\delta$, we have $\Phi(t,p) = (t,F(0,p)) = (t,p)$, which implies $\Phi|_{\{0\}\times|\mathcal X|}=\mathrm{id}$ and $\operatorname{supp}\Phi\subset [0,\delta]\times|\mathcal X|\subset V\cap\eta(V)$.
On the other hand, if $0\le t\le\delta/2$, we have
\[
\Phi(t,p)
= (t,F(t,p))
= \eta(t,p)\,.
\]
Thus, in this case, we obtain that $\eta\Phi^{-1}$ is the identity on the open subset $\Phi([0,\delta/2)\times|\mathcal X|)$ which contains $\{0\}\times|\mathcal X|$.
Therefore, $\Phi$ satisfies all the required properties, which completes the proof.
\end{proof}

\begin{remark}
Recall that we required neat arrangements to have compact ambient manifolds.
That is why the proof of \Cref{lem:collar-diffeo} is much simpler than Lemma 6.1 in \cite{Mun66}.
Note that we can also relax the compactness in \Cref{lem:collar-diffeo}.
This does not, however, really help our purpose while it requires more works.
\end{remark}

\begin{proposition}
\label{prop:col-unique}
Let $\mathcal X$ and $\mathcal X'$ be two neat arrangements of manifolds of the same shape $S$.
Suppose we have open embeddings
\[
\gamma:\mathbb R_+\times\overline\partial_C\mathcal X \to \mathcal X
\ ,\quad \gamma':\mathbb R_+\times\overline\partial_{C'}\mathcal X'\to\mathcal X'
\]
for connected faces $C\in\operatorname{bd}\mathcal X$ and $C'\in\operatorname{bd}\mathcal X'$.
Then, for every diffeomorphism $H:\mathcal X\to\mathcal X'$ and an open neighborhood $V\subset\mathbb R_+\times C$ of $\{0\}\times C$, there is a diffeomorphism $\Phi:\mathbb R_+\times\overline\partial_C\mathcal X\xrightarrow\sim\mathbb R_+\times\overline\partial_C\mathcal X$ of arrangements such that
\begin{enumerate}[label={\rm(\alph*)}]
  \item the support of $\Phi$ is contained in $V$;
  \item two maps $H\gamma\Phi,\gamma'(\mathrm{id}\times H|_C):\mathbb R_+\times\overline\partial_C\mathcal X\to\mathcal X'$ coincide with one another on a neighborhood of $\{0\}\times C$.
\end{enumerate}
\end{proposition}
\begin{proof}
Consider the composition map
\[
\widetilde\eta:\mathbb R_+\times\overline\partial_C\mathcal X
\xrightarrow{\mathrm{id}\times (H|_C)} \mathbb R_+\times\overline\partial_{C'}\mathcal X'
\xrightarrow{\gamma'} \mathcal X'
\xrightarrow{H^{-1}} \mathcal X\,,
\]
which is clearly an open embedding and restricts to the canonical inclusion $\{0\}\times\overline\partial_C\mathcal X\hookrightarrow\mathcal X$.
Since $\gamma:\mathbb R_+\times\overline\partial_C\mathcal X\to\mathcal X$ is an open embedding onto an neighborhood of $C=|\overline\partial_C\mathcal X|$, we can find a positive number $\delta>0$ such that $[0,\delta]\times C$ and $\widetilde\eta([0,\delta]\times C)$ are contained in $V$ and in the image of $\gamma$ respectively.
Then, there is a unique open embedding $\eta:[0,\delta)\times\overline\partial_C\mathcal X\to\mathbb R_+\times\overline\partial_C\mathcal X$ so that $\gamma\eta=\widetilde\eta$.
Applying \Cref{lem:collar-diffeo}, we obtain an auto-diffeomorphism $\Phi'$ on the arrangement $[0,\delta)\times\overline\partial_C\mathcal X$ such that
\begin{itemize}
  \item $\operatorname{supp}\Phi'\subset[0,\delta)\times|\overline\partial_C\mathcal X|$ is compact;
  \item the composition $\eta\Phi'^{-1}:[0,\delta)\times\overline\partial_C\mathcal X\to\mathbb R_+\times\overline\partial_C\mathcal X$ agrees with the canonical inclusion on a neighborhood of $\{0\}\times C$.
\end{itemize}
One can verify that $\Phi'$ extends to an auto-diffeomorphism $\Phi$ on $\mathbb R_+\times\overline\partial_C\mathcal X$ which is the identity outside $[0,\delta)\times\overline\partial_C\mathcal X$, so that $\Phi$ has a compact support contained in $[0,\delta)\times C$, which is a subset of $V$.
Moreover, the second condition implies that, on a neighborhood of $\{0\}\times C$, the map $\widetilde\eta\Phi^{-1}=\gamma\eta\Phi^{-1}$ agrees with $\gamma$.
Thus, we obtain $\widetilde\eta=\gamma\Phi$, which is what we want.
\end{proof}

\section{Relative transversality theorem}
\label{sec:rel-transversality}

\subsection{Parametric Transversality Theorem}
\label{sec:transv}

To begin with, we review the classical results on transversality.
For a fixed closed submanifold $W\subset Y$, we consider the subset
\[
\mathcal T_W := \{F\in C^\infty(X,Y)\mid F\pitchfork W\}\,\subset C^\infty(X,Y)\,.
\]
Since some important properties of smooth maps can be written as transversalities, it is important to investigate the set $\mathcal T_W$.
Fortunately, we have Whitney $C^\infty$-topology, so we can discuss topological properties of the set.
For example, we can see the transversality is essentially an open condition:

\begin{proposition}[cf. Proposition II.4.5 in \cite{GG73}]
\label{prop:transv-open}
Let $X$ and $Y$ be smooth manifolds with corners, and let $W\subset Y$ be a submanifold.
For every compact subset $A\subset Y$, the set
\[
\mathcal T_{W,A} := \{F\in C^\infty(X,Y)\mid F\pitchfork W\ \text{on}\  W\cap A\}
\]
is open in $C^\infty(X,Y)$ with respect to the Whitney $C^1$ (thus, also $C^\infty$) topology.
\end{proposition}

In the later sections, we will discuss perturbation of smooth maps respecting some transversal conditions.
In the purpose, we want ``sufficiently many'' maps satisfying transversal conditions near arbitrary maps.
The following result provides a good criterion for this:

\begin{theorem}[Parametric Transversality Theorem, cf. Lemma 6.4 in \cite{Mic80}]
\label{theo:param-transv}
Let $B$, $X$ and $Y$ be manifolds with corners, and let $W\subset Y$ be a submanifold.
Suppose we have a smooth map $\Phi:B\times X\to Y$ with $\Phi\pitchfork W$.
For each $b\in B$, we write $\Phi_b=\Phi(b,\blankdot):X\to Y$.
Then the set
\[
\{b\in B\mid \Phi_b\pitchfork W\}
\]
is a dense subset of $B$.
\end{theorem}

This is a key result to prove our main result.
Precisely, we use it in the following form:

\begin{corollary}
\label{cor:param-perturb}
Let $X$ and $Y$ be manifolds with corners, and let $B$ be a topological space.
Suppose we have an open subset $U\subset B$ which is a manifold and a continuous map $\Phi:B\times X\to Y$ whose restriction to $U$ is a smooth submersion $U\times X\to Y$.
Then, for every submanifold $W\subset Y$ and every point $b\in \overline U$ in the closure of $U$ in $B$, there is a sequence $\{b_1,b_2,\dots\}\subset U$ which converges to $b$ in $B$ and $\Phi_{b_i}\pitchfork W$ for each $i$.
\end{corollary}
\begin{proof}
Since the smooth map $\Phi_{U\times X}:U\times X\to Y$ is a submersion, we have $\Phi\pitchfork W$ for every submanifold $W\subset Y$.
Applying \Cref{theo:param-transv}, we obatin a dense subset of $U$ where we have $\Phi_b\pitchfork W$.
This dense subset is also dense in the closure $\overline U$ in $B$, so we obtain the result.
\end{proof}

\subsection{Relative jets on arrangements}
\label{sec:multjet-config}

In this section, we want to discuss jet bundles and transversality theorems in more general relative settings than usual (or those in \cite{Ish98}).
We define multijet bundles over excellent arrangements.

\begin{definition}
Let $\mathcal X$ and $\mathcal Y$ be excellent arrangements of shape $S$.
We denote by $0$ the minimum element of $S$.
Then for each non-negative integer $r\ge 0$, we define the $r$-th relative jet space $J^r(\mathcal X,\mathcal Y)$ to be the image of the composition
\[
\begin{split}
C^\infty(\mathcal X,\mathcal Y)\times \mathcal X(0)
&\hookrightarrow C^\infty(|\mathcal X|,|\mathcal Y|)\times |\mathcal X| \\
&\xrightarrow{j^r\times |\mathcal X|} C^\infty(|\mathcal X|,J^r(|\mathcal X|,|\mathcal Y|))\times|\mathcal X| \\
&\xrightarrow{\mathrm{eval}} J^r(|\mathcal X|,|\mathcal Y|)\,.
\end{split}
\]
\end{definition}

We have a canonical map $\pi:J^r(\mathcal X,\mathcal Y)\to\mathcal X(0)\times\mathcal Y(0)$.
It is, however, not obvious from the definition that $J^r(\mathcal X,\mathcal Y)$ is a manifold (with corners).
To see this, we give more algebraic definition for $J^r(\mathcal X,\mathcal Y)$.

\begin{definition}
Let $X'\hookrightarrow X$ be an embedding of submanifold, and let $p\in X'$.
Then we have an $\mathbb R$-algebra homomorphism $C^\infty_p(X)\to C^\infty_p(X')$.
We define an ideal $\mathfrak v_p(X')\subset C^\infty_p(X)$ by
\[
\mathfrak v_p(X') := \operatorname{ker}\left( C^\infty_p(X)\to C^\infty_p(X')\right)\,.
\]
\end{definition}

In other words, $\mathfrak v_p(X')\subset C^\infty_p(X)$ consists of functions defined near $p$ which are identically zero on the submanifold $X'$.
The ideal $\mathfrak v_p(X')$ locally determines the submanifold $X'$ around $p$.
The following lemma is a key observation:

\begin{lemma}
\label{lem:def-ideal-Cp}
Let $X'\subset X$ and $Y'\subset Y$ be submanifolds of manifolds with corners.
Let $F\in C^\infty_p(X,Y)$ be a smooth function locally defined around $p\in X$ and $q:=F(p)\in Y$.
Then $F$ maps a neighborhood of $p$ on $X'$ into $Y'$ if and only if we have
\[
F^\ast\mathfrak v_q(Y') \subset \mathfrak v_p(X')\,.
\]
\end{lemma}
\begin{proof}
First suppose that $F$ maps a neighborhood of $p$ on $X'$ into $Y'$.
We show $F^\ast\mathfrak v_q(Y')\subset\mathfrak v_p(X')$.
This is equivalent to that for each function $\lambda$ defined near $q$ on $Y$ which is identically zero on $Y'$, the composition $\lambda F$ is also identically zero on $X'$.
This is obvious from the assumption of $F$.

Conversely, suppose we have $F^\ast\mathfrak v_q(Y')\subset\mathfrak v_p(X')$.
By \Cref{cor:submfd-defzero}, there is a neighborhood $V$ of $q$ on $Y$ together with a smooth function $\lambda:V\to[0,1]$ such that $Y'\cap V = \{\lambda=0\}$.
Take a sufficiently small neighborhood $U$ of $p$ on $X$ so that $f$ is defined on $U$.
Since we can take $V$ as small as needed, we may assume $F^{-1}(V)\subset U$.
Then, by the assumption $F^\ast\mathfrak v_q(Y')\subset\mathfrak v_p(X')$, the function $\lambda F:U\to\mathbb [0,1]$ is identically zero on a neighborhood, say $U_0$, of $p$ on $X'$.
This implies $\lambda$ is identically zero on $F(U_0)$, and we obtain $F(U_0)\subset \{\lambda=0\}\subset Y'$.
\end{proof}

Let $\mathcal X$ and $\mathcal Y$ be excellent arrangements of shape $S$ with the minimum $0$.
For each $(p,q)\in\mathcal X(0)\times\mathcal Y(0)$, we define an $\mathbb R$-vector space $J^r(\mathcal X,\mathcal Y)_{p,q}$ to be the set of $\mathbb R$-algebra homomorphisms
\[
f:C^\infty_q(|\mathcal Y|)/\mathfrak m_q^{r+1}\to C^\infty_p(|\mathcal X|)/\mathfrak m_p^{r+1}
\]
such that $f$ restricts to a homomorphism
\[
\left(\mathfrak v_q(\mathcal Y(s))+\mathfrak m_q^{r+1}\right)\big/\mathfrak m_q^{r+1}
\to \left(\mathfrak v_p(\mathcal X(s))+\mathfrak m_p^{r+1}\right)\big/\mathfrak m_p^{r+1}
\]
for each $s\in S$.
Hence, $J^r(\mathcal X,\mathcal Y)_{(p,q)}$ is a subset of $J^r(|\mathcal X|,|\mathcal Y|)_{(p,q)}$.
We have a coordinate-dependent description.

\begin{lemma}
\label{lem:rel-jet-coord}
Let $\mathcal X$ be an excellent arrangement of manifolds of shape $S$, say $0\in S$ is the minimum.
Suppose we have an $\mathcal X$-chart $(U,\varphi,\mathcal I)$ around $p\in\mathcal X(0)$ with $|\mathcal I|=\langle m|k\rangle$.
Then, the induced homomorphism
\[
\varphi_!: C^\infty_p(|\mathcal X|) \to \mathbb R\double[x_1,\dots,x_m\double]
\]
maps each ideal $\mathfrak v_p(\mathcal X(s))$ onto the ideal
\[
(x_i\mid i\notin\mathcal I(s)) \subset \mathbb R\double[x_1,\dots,x_m\double]\,.
\]
\end{lemma}
\begin{proof}
Since $(U,\varphi,\mathcal I)$ is an $\mathcal X$-chart with $|\mathcal I|=\langle m|k\rangle$, we have a pullback square
\begin{equation}
\label{diag:prf:rel-jet-coord:chartpb}
\xymatrix{
  \mathcal X(s)\cap U \ar[d] \ar[r] \ar@{}[dr]|(.4)\pbcorner & U \ar[d]^{\varphi} \\
  \mathbb H^{\mathcal I(s)} \ar[r] & \mathbb H^{\langle m|k\rangle} }
\end{equation}
for each $s\in S$.
For a smooth function $f\in C^\infty(U)$, using the standard coordinate $(x_1,\dots,x_m)\in\mathbb H^{\langle m|k\rangle}$, we have
\[
\varphi_!(f) = \sum_\alpha \frac{\partial^{|\alpha|} f\varphi^{-1}}{\partial^\alpha x}(0)\frac{x^\alpha}{\alpha!}\,.
\]
By the pullback square \eqref{diag:prf:rel-jet-coord:chartpb}, if $f$ belongs to the ideal $\mathfrak v_p(\mathcal X(s))$, then the map $f\varphi^{-1}:\varphi(U)\to\mathbb R$ vanishes on the subspace $\varphi(U)\cap\mathbb H^{\mathcal I(s)}$.
Hence, we obtain $f\varphi^{-1}(0)=0$ and
\[
\frac{\partial f\varphi^{-1}}{\partial x_i}(0)=0
\]
for every $i\in\mathcal I(s)$.
It follows that the homomorphism $\varphi_!:C^\infty_p(|\mathcal X|)\to\mathbb R\double[x_1,\dots,x_m\double]$ restricts to
\[
\mathfrak v_p(\mathcal X(s))\to (x_i\mid i\notin\mathcal I(s))\,.
\]
This map is onto since, for each $i\notin\mathcal I(s)$, the map $\mathbb H^{\langle m|k\rangle}\ni x\mapsto x_i\in\mathbb R$ gives an element of $\mathfrak v_p(\mathcal X(s))$ through $\varphi$.
\end{proof}

\begin{corollary}
\label{cor:reljet-fib}
Let $\mathcal X$ and $\mathcal Y$ be excellent arrangements of the same shape $S$, say $0\in S$ is the minimum.
Suppose we have an $\mathcal X$-chart $(U,\varphi,\mathcal I)$ and a $\mathcal Y$-chart $(V,\psi,\mathcal J)$ around $p\in\mathcal X(0)$ and $q\in\mathcal Y(0)$ respectively.
Then they induce a canonical isomorphism
\[
J^r(\mathcal X,\mathcal Y)_{(p,q)}
\simeq P^r(\mathcal I,\mathcal J)_0\,,
\]
where $P^r(\mathcal I,\mathcal J)_0$ is the set of polynomial maps $G:\mathbb R^{|\mathcal I|}\to\mathbb R^{|\mathcal J|}$ of degree at most $r$ such that $G(0)=0$ and $G(\mathbb R^{\mathcal I(s)})\subset\mathbb R^{\mathcal J(s)}$ for each $s\in S$.
\end{corollary}

\begin{remark}
Under the identification $P^1(|\mathcal I|,|\mathcal J|)_0\cong\mathrm{Hom}_{\mathbb R}(\mathbb R^{|\mathcal I|},\mathbb R^{|\mathcal J|})$, we have
\[
P^1(\mathcal I,\mathcal J)_0
\cong \left\{f\in\mathrm{Hom}_{\mathbb R}(\mathbb R^{|\mathcal I|},\mathbb R^{|\mathcal J|})\mid\forall s\in S:f(\mathbb R^{\mathcal I(s)})\subset\mathbb R^{\mathcal J(s)}\right\}\,.
\]
\end{remark}

As a consequence of \Cref{cor:reljet-fib}, we obtain the required result:

\begin{proposition}
\label{prop:rel-jetbdl}
Let $\mathcal X$ and $\mathcal Y$ be excellent arrangements of shape $S$.
Then, for each $r\ge 0$, there is a smooth fiber bundle
\[
J^r(\mathcal X,\mathcal Y)\to\mathcal X(0)\times\mathcal Y(0)
\]
which is a subbundle of (the restriction of) the usual $r$-th jet bundle
\[
J^r(|\mathcal X|,|\mathcal Y|)\times_{|\mathcal X|\times|\mathcal Y|}(\mathcal X(0)\times\mathcal Y(0))
\to \mathcal X(0)\times\mathcal Y(0)
\]
with fiber $P^r(\mathcal I,\mathcal J)_0\subset P^r(|\mathcal I|,|\mathcal J|)_0$.
\end{proposition}

The dimension of the manifold $J^r(\mathcal X,\mathcal Y)$ is complicated in general.
We only mention the case $r=1$.

\begin{lemma}
\label{lem:rel1jet-dim}
Let $\mathcal X$ and $\mathcal Y$ be excellent arrangements of shape $S$.
Then, for $p\in\mathcal X(0)$ and $q\in\mathcal Y(0)$, we have a canonical identification
\[
J^1(\mathcal X,\mathcal Y)_{(p,q)}
\cong\left\{\theta\in\mathrm{Hom}_{\mathbb R}(T_p|\mathcal X|,T_q|\mathcal Y|)\mid\forall s\in S:\theta(T_p\mathcal X(s))\subset T_q\mathcal Y(s)\right\} \\
\]
Consequently, we have
\[
\dim J^1(\mathcal X,\mathcal Y)_{(p,q)}
= \sum_{s\in S} \dim\left(T_p\mathcal X(s)/\sum_{s'<s}T_p\mathcal X(s')\right)\cdot\dim\mathcal Y(s)\,.
\]
\end{lemma}
\begin{proof}
The first assertion follows from the remark just before \Cref{prop:rel-jetbdl}.
To show the latter, take an $\mathcal X$-chart $(U,\varphi,\mathcal I)$ around $p\in\mathcal X(0)$.
Say $\xi_i:=\varphi^\ast(\partial/\partial_i)$ for each $i\in|\mathcal I|$, and write
\[
s_{\mathcal I}(i):=\min\{s\in S:i\in\mathcal I(s)\}\,.
\]
Then, for $\theta\in\mathrm{Hom}_{\mathbb R}(T_p|\mathcal X|,T_q|\mathcal Y|)$, $\theta\in J^1(\mathcal X,\mathcal Y)_{(p,q)}$ if and only if
\[
\theta(\xi_i)\in T_q\mathcal Y(s_{\mathcal I}(i))\,.
\]
Hence, we obtain isomorphisms
\[
J^1(\mathcal X,\mathcal Y)_{(p,q)}
\cong \bigoplus_{i\in|\mathcal I|} T_q\mathcal Y(s_{\mathcal I}(i))
\cong \bigoplus_{s\in S} (T_q\mathcal Y(s))^{\oplus \#\{i:s_{\mathcal I}(i)=s\}}\,.
\]
On the other hand, it is easily verified that
\[
\dim\left(T_p\mathcal X(s)/\sum_{s'<s}T_p\mathcal X(s')\right)
= \{i:s_{\mathcal I}(i)=s\}\,.
\]
so the result follows.
\end{proof}

The bundle has a functorial property on the second variable:

\begin{lemma}
\label{lem:mjet-funcemb}
Let $\mathcal X$, $\mathcal Y$, and $\mathcal Z$ be excellent arrangements of the same shape $S$.
Then every $S$-map $F\in C^\infty(\mathcal Y,\mathcal Z)$ induces a smooth map
\[
J^r(\mathcal X,\mathcal Y) \to J^r(\mathcal X,\mathcal Z)
\]
for each $r\in\mathbb Z_{\ge0}$, which is an embedding provided so is $F$.
\end{lemma}
\begin{proof}
Say $0\in S$ is the minimum element.
For simplicity, we write
\[
\widetilde J^r(\mathcal X,\mathcal Y) := J^r(|\mathcal X|,|\mathcal Y|)\times_{(|\mathcal X|\times|\mathcal Y|)}(\mathcal X(0)\times\mathcal Y(0))\,.
\]
Then, since $J^r(\mathcal X,\mathcal Y)$ is a subbundle of $\widetilde J^r(\mathcal X,\mathcal Y)$, by \Cref{lem:jet-functor}, we have a smooth map
\begin{equation}
\label{eq:prf:mjet-funcemb:indF}
J^r(\mathcal X,\mathcal Y)
\hookrightarrow \widetilde J^r(\mathcal X,\mathcal Y)
\xrightarrow{F_\ast} \widetilde J^r(\mathcal X,\mathcal Z)\,,
\end{equation}
which is an embedding provided so is $F$ by \Cref{lem:jet-functor}.
Thus, to obtain the result, it suffices to show that the map \eqref{eq:prf:mjet-funcemb:indF} factors through the submanifold $J^r(\mathcal X,\mathcal Z)\subset \widetilde J^r(\mathcal X,\mathcal Z)$.
We can check this fiberwisely; we show the map
\[
F_\ast:J^r(|\mathcal X|,|\mathcal Y|)_{(p,q)}\to J^r(|\mathcal X|,|\mathcal Z|)_{(p,F(q))}
\]
induced by $F$ restricts to a map $J^r(\mathcal X,\mathcal Y)_{(p,q)}\to J^r(\mathcal X,\mathcal Z)_{(p,F(q))}$ for each $(p,q)\in\mathcal X(0)\times\mathcal Y(0)$.
Note that since $F$ is an $S$-map, for each $q\in\mathcal Y(0)$, the induced homomorphism
\[
J^r(|\mathcal Z|)\to J^r(|\mathcal Y|)
\]
maps $\mathfrak v_{F(q)}(\mathcal Z(s))$ into $\mathfrak v_q(\mathcal Y(s))$ for each $s\in S$ by \Cref{lem:def-ideal-Cp}.
Thus, if $g:J^r(|\mathcal Y|)_q\to J^r(|\mathcal X|)_p$ is a $\mathbb R$-algebra homomorphism with $g(\mathfrak v_q(\mathcal Y(s)))\subset\mathfrak v_p(\mathcal X(s))$, which is hence an element of $J^r(\mathcal X,\mathcal Y)_{(p,q)}$, then the homomorphism
\[
F_\ast(g):J^r(|\mathcal Z|)_{F(q)} \to J^r(|\mathcal X|)_p
\]
restricts to the composition
\[
\mathfrak v_{F(q)}(\mathcal Z(s))
\xrightarrow{F^\ast} \mathfrak v_q(\mathcal Y(s))
\xrightarrow{g} \mathfrak v_p(\mathcal X(s))
\]
for each $s\in S$.
This is just what we want to show.
\end{proof}

Note that, by definition, we have a well-defined map
\begin{equation}
\label{eq:def-rel-jr}
\begin{array}{rccc}
j^r:& C^\infty(\mathcal X,\mathcal Y) & \to & C^\infty(\mathcal X(0),J^r(\mathcal X,\mathcal Y)) \\
& f & \mapsto & j^r(f)|_{\mathcal X(0)}
\end{array}
\end{equation}

\begin{lemma}
\label{lem:rel-jr-conti}
Let $\mathcal X$ and $\mathcal Y$ be excellent arrangements.
Then the map $j^r$ defined in \eqref{eq:def-rel-jr} is continuous (with respect to the Whitney $C^\infty$-topology).
\end{lemma}
\begin{proof}
Since we have an embedding
\[
J^r(\mathcal X,\mathcal Y) \to J^r(|\mathcal X|,|\mathcal Y|)\,,
\]
by \ref{subprop:whit-comp:post} in \Cref{prop:whit-comp}, it suffices to show that the composition
\[
C^\infty(\mathcal X,\mathcal Y)
\to C^\infty(\mathcal X(0),J^r(\mathcal X,\mathcal Y))
\to C^\infty(\mathcal X(0),J^r(|\mathcal X|,|\mathcal Y|))
\]
is continuous.
Note that it can be obtained in another way as
\[
\begin{split}
C^\infty(\mathcal X,\mathcal Y)
&\hookrightarrow C^\infty(|\mathcal X|,|\mathcal Y|) \\
&\xrightarrow{j^r} C^\infty(|\mathcal X|,J^r(|\mathcal X|,|\mathcal Y|)) \\
&\xrightarrow{\iota^\ast} C^\infty(\mathcal X(0),J^r(|\mathcal X|,|\mathcal Y|))\,.
\end{split}
\]
The first arrow is the embedding, the second arrow is continuous by Proposition II.3.4 in \cite{GG73}.
Moreover, since $\mathcal X$ is an excellent arrangement so that $\mathcal X(0)\to|\mathcal X|$ is a closed embedding, the last arrow is also continuous by \ref{subprop:whit-comp:pre} in \Cref{prop:whit-comp}.
Therefore, the result follows.
\end{proof}

Now, we define ``multi-relative'' jet bundles.
Before doint this, we prepare some notations.
For a finite lattice $S$, we set
\[
S^{[1]} := \left\{ (s,t)\in S\times S\ \middle|\ s\le t\right\}
\]
and give an ordering by $(s,t)\le(s',t')\iff s\le s'$ and $t\le t'$.
One can easily check that $S^{[1]}$ is again a finite lattice.
For $\kappa=(s,t)\in S^{[1]}$, we will write $\kappa_0:=s$ and $\kappa_1:=t$.
Note that we can identify an element $\kappa\in S^{[1]}$ with an interval $[\kappa_0,\kappa_1]$ of $S$, that is, a sublattice of $S$ consisting elements $s\in S$ with $\kappa_0\le s\le \kappa_1$.

For an excellent arrangement $\mathcal X$ of manifolds of shape $S$, and for each $\kappa\in S^{[1]}$, we denote by $\mathcal X_\kappa$ the restriction of $\mathcal X$ to the interval $[\kappa_0,\kappa_1]$ (see \Cref{ex:arr:restrict}).
For example, the ambient manifold of $\mathcal X_\kappa$ equals $\mathcal X(\kappa_1)$.
Moreover, for each $s\in S$, we define
\[
\mathcal X\double(s\double) := \mathcal X(s)\setminus{\textstyle\bigcup_{t<s}}\mathcal X(t)\,.
\]
In other words, $\mathcal X\double(s\double)$ is the subset of $\mathcal X(s)$ consisting of points of scope $s$.
It is clear that $\mathcal X\double(s\double)$ is open in $\mathcal X(s)$ so that a submanifold.
Notice that we have the relative jet bundle
\[
J^r(\mathcal X_\kappa,\mathcal Y_\kappa) \to \mathcal X(\kappa_0)\times\mathcal Y(\kappa_0)\,.
\]

Finally,for a map $\mathbf n:S^{[1]}\to\mathbb Z_{\ge 0}$, we set
\[
\begin{gathered}
\mathcal X^{\mathbf n} := \raisedunderop{\prod}{\kappa\in\smash{S^{[1]}}} \mathcal X(\kappa_0)^{\times\mathbf n(\kappa)}\,, \\
\mathcal X^{(\mathbf n)} := |\mathcal X|^{(|\mathbf n|)}\cap\raisedunderop{\prod}{\kappa\in\smash{S^{[1]}}} \mathcal X\double(\kappa_0\double)^{\times\mathbf n(\kappa)}\,,
\end{gathered}
\]
where $|\mathbf n|:=\{(\kappa,i)\in S^{[1]}\times\mathbb N\mid 1\le i\le\mathbf n(\kappa)\}$ and
\[
X^{(A)}
:= \left\{(x_a)_{a\in A}\in X^{A}\ \middle|\ x_a\neq x_b\ \text{for}\ a\neq b\in A\right\}
\]
for any set $A$.

\begin{definition}
Let $\mathcal X$ and $\mathcal Y$ be excellent arrangements of manifolds of shape $S$.
Let $\mathbf n:S^{[1]}\to\mathbb Z_{\ge 0}$ be an arbitrary map and $r\ge 0$ be a non-negative integer.
Then the $r$-th $\mathbf n$-multijet bundle $J^r_{\mathbf n}(\mathcal X,\mathcal Y)$ is defined by the following pullback square:
\begin{equation}
\label{diag:def:multijet}
\renewcommand{\objectstyle}{\displaystyle}
\xymatrix{
  J^r_{\mathbf{n}}(\mathcal X,\mathcal Y) \ar[r] \ar[d] \ar@{}[dr]|(.4){\pbcorner} & \raisedunderop\prod{\kappa\in\smash{S^{[1]}}} J^r(\mathcal X_\kappa,\mathcal Y_\kappa)^{\times\mathbf n(\kappa)} \ar@{->>}[d] \\
  \mathcal X^{(\mathbf n)}\times \mathcal Y^{\mathbf n} \ar@{^(->}[r] & \mathcal X^{\mathbf n}\times \mathcal Y^{\mathbf n} }
\end{equation}
\end{definition}

Since, by \Cref{prop:rel-jetbdl}, the right vertical arrow of the square \eqref{diag:def:multijet} is a smooth fiber bundle, as its pullback, the left vertical arrow also exhibits $J^r_{\mathbf n}(\mathcal X,\mathcal Y)$ as a smooth fiber bundle over the manifold $\mathcal X^{(\mathbf n)}\times\mathcal Y^{\mathbf n}$ with corners.
Moreover, similarly to the usual multijet bundle, we have a canonical map
\[
j^r_{\mathbf n}:C^\infty(\mathcal X,\mathcal Y)\to C^\infty(\mathcal X^{(\mathbf n)},J^r_{\mathbf n}(\mathcal X,\mathcal Y))
\]
as the map applying $j^r$ defined in \eqref{eq:def-rel-jr} in each component; i.e.
\begin{equation}
\label{eq:multi-jet}
\begin{multlined}
j^r_{\mathbf n}(f)\left( (x^\nu)_{\nu\in |\mathbf n|}\right) \\
= \left(j^r(f|_{\mathcal X\double(\kappa_0(\nu)\double)})(x^\nu)\right)_{\nu\in|\mathbf n|}
\,\in J^r_{\mathbf n}(\mathcal X,\mathcal Y)\,.
\end{multlined}
\end{equation}
For a map $F\in C^\infty(\mathcal X,\mathcal Y)$, we call the smooth map $j^r_{\mathbf n}F:\mathcal X^{(\mathbf n)}\to J^r_{\mathbf n}(\mathcal X,\mathcal Y)$ the $r$-th $\mathbf n$-multijet of $F$.

\begin{notation}
For excellent arrangements $\mathcal X$ and $\mathcal Y$ of shape $S$, it is often convenient to consider an $S^{[1]}$-indexed family
\[
\mathcal J^r(\mathcal X,\mathcal Y)
:= \left\{ J^r(\mathcal X_\kappa,\mathcal Y_\kappa)\right\}_{\kappa\in\smash{S^{[1]}}}\,.
\]
of manifolds.
Although $\mathcal J^r(\mathcal X,\mathcal Y)$ is not an arrangement of manifolds in the precise sense, we will also write
\[
\mathcal J^r(\mathcal X,\mathcal Y)^{\mathbf n}
:= \raisedunderop\prod{\kappa\in\smash{S^{[1]}}} J^r(\mathcal X_\kappa,\mathcal Y_\kappa)^{\mathbf n(\kappa)}\,.
\]
\end{notation}

\begin{lemma}
\label{lem:multijet-conti}
Let $\mathcal X$ and $\mathcal Y$ be closed $S$-arrangements on manifolds $X$ and $Y$ with corners respectively.
Then for every $r\ge 0$ and $\mathbf n:S\to\mathbb Z_{\ge 0}$, the map
\[
j'^r_{\mathbf n}:C^\infty(\mathcal X,\mathcal Y)
\to C^\infty(\mathcal X^{\mathbf n},\mathcal J^r(\mathcal X,\mathcal Y)^{\mathbf n})
\]
by the same formula as \eqref{eq:multi-jet} is continuous (with respect to the Whitney $C^\infty$-topology).
\end{lemma}
\begin{proof}
We have a facotorization of the map as
\[
\begin{split}
C^\infty(\mathcal X,\mathcal Y)
&\to \raisedunderop\prod{\kappa\in\smash{S^{[1]}}} C^\infty(\mathcal X(\kappa_0),J^r(\mathcal X_\kappa,\mathcal Y_\kappa))^{\times\mathbf n(s)} \\
&\to C^\infty(\mathcal{X}^{\mathbf{n}},\mathcal{J}^r(\mathcal{X},\mathcal{Y})^{\mathbf{n}})
\end{split}
\]
The first map is continuous by \Cref{cor:submfd-jet}, and the second is also continuous by \Cref{prop:Cinf-prod-contin}.
Hence we obtain the result.
\end{proof}

\subsection{Polynomial perturbation}
\label{sec:poly-puturb}

\begin{definition}
\begin{enumerate}[label={\rm(\arabic*)}]
  \item For finite sets $I$ and $J$, we denote by $P^r(I,J)$ the set of polynomial mappings $\mathbb R^I\to\mathbb R^J$ of degree at most $r$.
We also use the same notation when $I$ and $J$ are marked sets.
  \item For arrangements $\mathcal I$ and $\mathcal J$ of finite sets of shape $S$, we denote by $P^r(\mathcal I,\mathcal J)$ the subset of $P^r(|\mathcal I|,|\mathcal J|)$ consisting of polynomial mappings $f$ such that $f(\mathbb R^{\mathcal I(s)})\subset\mathbb R^{\mathcal J(s)}$ for each $s\in S$.
  \item For marked finite sets $I$ and $J$, and for a subset $A\subset\mathbb H^I$, we denote by $P^r_A(I,J)$ the subset of $P^r(I,J)$ consisting of $f\in P^r(I,J)$ such that $f(A)\subset\mathbb H^J$.
  \item For arrangements $\mathcal I$ and $\mathcal J$ of marked finite sets of shape $S$, and for a subset $A\subset\mathbb H^{|\mathcal I|}$, we denote by $P^r_A(\mathcal I,\mathcal J)$ the subset of $P^r(\mathcal I,\mathcal J)$ consisting of $f$ with $f(A\cap\mathbb H^{\mathcal I(s)})\subset\mathbb H^{\mathcal J(s)}$ for each $s\in S$.
\end{enumerate}
\end{definition}

Recall that if $\#I = m$ and $\#J=n$, then there is an isomorphism
\[
P^r(I,J) \simeq \left(\mathbb R[X_1,\dots,X_m]/(X_1,\dots,X_m)^{r+1}\right)^n\,,
\]
and they are diffeomorphic to the Euclidean space of dimension
\[
n\cdot\begin{pmatrix} r+m \\ m \end{pmatrix}\,.
\]
Moreover, $P^r(\mathcal I,\mathcal J)$ is an $\mathbb R$-linear subspace of $P^r(I,J)$.
Indeed, an element $f=(f_j)_{j\in J}\in P^r(I,J)$, say $f_j(x)=\sum_{|\alpha|\le r} a_{j\alpha} x^\alpha$, belongs to $P^r(\mathcal I,\mathcal J)$ if and only if we have
\[
\forall s\in S\quad\forall j\notin\mathcal J(s)\quad\forall\alpha\in\mathbb N^{\mathcal I(s)}\,:\, a_{j\alpha}=0\,.
\]

We give $P^r(I,J)$ and $P^r(\mathcal I,\mathcal J)$ the standard topologies of finite dimensional $\mathbb R$-linear spaces.
We also give $P^r_A(I,J)$ and $P^r_A(\mathcal I,\mathcal J)$ the relative topologies.

\begin{lemma}
\label{lem:mkH-preserve}
Let $\mathcal I$ and $\mathcal J$ be arrangements of marked finite sets.
Let $A\subset \mathbb H^{|\mathcal I|}$ be a bounded subset (with respect to the standard metric).
Then there is an open subset $T\subset P^r(\mathcal I,\mathcal J)$ such that
\begin{enumerate}[label={\rm(\alph*)}]
  \item $T\subset P^r_A(\mathcal I,\mathcal J)$;
  \item the closure $\overline T$ in $P^r(\mathcal I,\mathcal J)$ contains the zero polynomial $0\in P^r(\mathcal I,\mathcal J)$.
\end{enumerate}
\end{lemma}
\begin{proof}
Since $A$ is bounded, we can take $\delta > 0$ so that for each $x\in A$, we have $\|x\| < \delta$.
In this case, for every polynomial $g(x)=\sum_\alpha b_\alpha x^\alpha\in\mathbb R[y_i:i\in I]$, we have
\[
g(x)\ge a_0 - \sum_\alpha |b_\alpha|\cdot\delta^{|\alpha|}
\]
for every $x\in A$.
For $f=(f_j)_{j\in J}\in P^r(\mathcal{I},\mathcal{J})$ with $f_j(x)=\sum_{|\alpha|\le r} a_{j\alpha}x^\alpha$, consider the following condition:
\[
\forall s\in S\quad\forall j\notin\mathcal J_+(s)\,:\, a_{j0} - \sum_{|\alpha|\le r} |a_{j\alpha}|\cdot\delta^{|\alpha|} > 0\,.
\]
Clearly, such polynomial functions $f\in P^r(\mathcal I,\mathcal J)$ form an open subset $T\subset P^r(\mathcal I,\mathcal J)$.
It is also easily verified that $T$ is a required open subset.
\end{proof}

\begin{lemma}
\label{lem:poly-perturb-subm}
Let $\mathcal I$ and $\mathcal J$ be arrangements of marked finite sets of shape $S$.
Let $U\subset\mathbb H^{|\mathcal I|}$ be a bounded open subset.
Say $0\in S$ is the minimum, and put $U_0:=U\cap\mathbb H^{\mathcal I(0)}$.
Suppose moreover that we have an open subset $T\subset P^r(\mathcal I,\mathcal J)$ such that $T\subset P^r_U(\mathcal I,\mathcal J)$.
Then for every $s\in S$ and every smooth map $f\in C^\infty(\mathbb H^{\mathcal I}\cap U,\mathbb H^{\mathcal J})$ between the canonical arrangements, the map
\[
G:T\times U_0\to J^r(\mathbb H^{\mathcal I}_{\le s}\cap U,\mathbb H^{\mathcal J}_{\le s})\,;\ (h,x)\mapsto j^r(f+h)(x)
\]
is a submersion.
\end{lemma}
\begin{proof}
Since $G$ is a bundle map over $U_0$, it suffices to show that $G$ induces submersions on fibers.
Notice that, for $x_0\in U_0$, the fiber of the bundle $J^r(\mathbb H^{\mathcal I}_{\le s}\cap U,\mathbb H^{\mathcal J}_{\le s})$ over $x_0$ is computed as
\[
J^r(\mathbb H^{\mathcal I}_{\le s}\cap U,\mathbb H^{\mathcal J}_{\le s})_{x_0}
\simeq \mathbb H^{\mathcal J(0)}\times P^r(\mathcal I_{\le s},\mathcal J_{\le s})_0\,,
\]
where $P^r(\mathcal I_{\le s},\mathcal J_{\le s})_0\subset P^r(\mathcal I_{\le s},\mathcal J_{\le s})$ is the subset of polynomial mappings preserving the origin.
In particular, there is a canonical embedding
\[
J^r(\mathbb H^{\mathcal I}_{\le s}\cap U,\mathbb H^{\mathcal J}_{\le s})_{x_0}
\to P^r(\mathcal I_{\le s},\mathcal J_{\le s})\,.
\]
Therefore, to see the induced map $G_{x_0}:T\to J^r(\mathbb H^{\mathcal I}_{\le s}\cap U,\mathbb H^{\mathcal J}_{\le s})_{x_0}$ is a submersion, it suffices to see the composition
\[
G_{x_0}:T
\to J^r(\mathbb H^{\mathcal I}_{\le s}\cap U,\mathbb H^{\mathcal J}_{\le s})_{x_0}
\to P^r(\mathcal I_{\le s},\mathcal J_{\le s})
\]
is a submersion.
This map factors through a restriction of the submersion
\[
P^r(\mathcal I,\mathcal J)\ni h
\mapsto \sum_{|\alpha|\le r}\frac1{\alpha!}\frac{\partial^{|\alpha|} f}{\partial^\alpha x}(x_0)x^\alpha + h
\in P^r(\mathcal I,\mathcal J)\,.
\]
to the open subset $T\subset P^r(\mathcal I,\mathcal J)$ followed by the projection
\[
P^r(\mathcal I,\mathcal J) \twoheadrightarrow P^r(\mathcal I_{\le s},\mathcal J_{\le s})\,.
\]
Thus the result follows.
\end{proof}

We want to perturb functions in $C^\infty(\mathcal X,\mathcal Y)$ by polynomials in $P^r(\mathcal I,\mathcal J)$.
A problem is that if $U\subset\mathbb H^{|\mathcal I|}$ is an open subset, the map
\[
P^r(\mathcal I,\mathcal J) \to C^\infty(\mathbb H^{\mathcal I}\cap U,\mathbb R^{\mathcal J})
\]
is not continuous.
In fact, we can only perturb functions ``on compact supports'':

\begin{lemma}
\label{lem:conti-perturb}
Let $\mathcal I$ and $\mathcal J$ be arrangements of marked finite sets of shape $S$.
Suppose we are given an open subset $U\subset \mathbb H^{|\mathcal I|}$ and a smooth function $\rho:U\to\mathbb R$ with compact support.
Then the map
\[
\begin{array}{rccc}
\rho_!: & P^r(\mathcal I,\mathcal J) &\to& C^\infty(\mathbb H^{\mathcal I}\cap U,\mathbb R^{\mathcal J}) \\
& b(x) &\mapsto& \left[ x \mapsto \rho(x)b(x)\right]
\end{array}
\]
is continuous.
\end{lemma}
\begin{proof}
We have to show that for each polynomial $b(x)\in P^r(\mathcal I,\mathcal J)$ and each open subset $\Phi\subset J^q(U,\mathbb R^{|\mathcal J|})$ with $\rho_!(b(x))\in M(\Phi)$, there is an open neighborhood $B$ of $b(x)$ on $P^r(\mathcal I,\mathcal J)$ such that $\rho_!(B)\subset M(\Phi)$.
Notice that we have a canonical identification
\[
J^q(U,\mathbb R^{|\mathcal J|}) \simeq U\times P^q(|\mathcal I|,|\mathcal J|)
\]
under which the map
\[
j^q:C^\infty(U\cap\mathbb H^{\mathcal I},\mathbb R^{\mathcal J})
\to C^\infty(U,J^q(U,\mathbb R^{|\mathcal J|}))
\]
is described as
\[
j^q(f)(x)
= \left(x,\bigg(\raisedunderop\sum{|\alpha|\le q}\frac1{\alpha!}\frac{\partial^{|\alpha|}f_j}{\partial x^\alpha}(x)X^\alpha\bigg)_{j\in|\mathcal J|}\right)\,.
\]
In particular, it is easily verified that the following composition is continuous:
\[
\varphi:P^r(\mathcal I,\mathcal J)\times U
\xrightarrow{\rho_!\times U} C^\infty(U\cap\mathbb H^{\mathcal I},\mathbb R^{\mathcal J})\times U
\xrightarrow{j^q} J^q(U,\mathbb R^{|\mathcal J|})
\]
Now suppose $\rho_!(b(x))\in M(\Phi)$, or equivalently $\varphi(\{b(x)\}\times U)\subset \Phi$.
Let us denote by $K:=\operatorname{supp}\rho$.
Then we have $\varphi(\{b(x)\}\times K)\subset \Phi$, and since $\varphi$ is continuous and $K$ is compact, there is an open neighborhood $B$ of $b(x)$ on $P^r(\mathcal I,\mathcal J)$ such that $\varphi(B\times K)\subset \Phi$.
Note that if $x_0\in U\setminus K$, we have $\varphi(b(x),x_0)=0$, which implies we also have
\[
\varphi(B\times(U\setminus K)) = \varphi(\{b(x)\}\times (U\setminus K)) \subset\Phi\,.
\]
Therefore, we obtain $\varphi(B\times U)\subset \Phi$, which implies $B$ is a required neighborhood of $b(x)\in P^r(\mathcal I,\mathcal J)$.
\end{proof}

\begin{lemma}
\label{lem:poly-perturb}
Let $\mathcal X$ and $\mathcal Y$ be excellent arrangements of manifolds of shape $S$, and let $r\ge 0$ and $\kappa=(\kappa_0,\kappa_1)\in S^{[1]}$.
Suppose we are given the following data:
\begin{itemize}
  \item An $\mathcal X$-chart $(U,\varphi,\mathcal I)$ on $|\mathcal X|$ of scope $\kappa_0$.
  \item An $\mathcal Y_{\ge\kappa_0}$-chart $(V,\psi,\mathcal J)$ on $|\mathcal Y|$ of scope $\kappa_0$.
  \item A compact subset $K\subset J^r(U\cap\mathcal X_\kappa,V\cap\mathcal Y_\kappa)$.
\end{itemize}
Then, for every smooth map $F\in C^\infty(\mathcal X,\mathcal Y)$, there are an open subset $B\subset P^r(\mathcal I,\mathcal J)$ and a smooth map $G:U_F\times B\to V$, here $U_F:=U\cap F^{-1}(V)$, satisfying the following properties:
\begin{enumerate}[label={\rm(\roman*)}]
  \item\label{sublem:poly-perturb:conti} $G$ defines a continuous map $\widetilde G:B\to C^\infty(U_F\cap\mathcal X,V\cap\mathcal Y)$.
  \item\label{sublem:poly-perturb:zero} The closure $\overline B\subset P^r(\mathcal I,\mathcal J)$ contains $0$, and $\widetilde G$ extends continuously to $0\in\overline B$ by $\widetilde G(0)=F$.
  \item\label{sublem:poly-perturb:origin} There is a compact subset of $U_F$ except where the smooth map $\widetilde G(b):U_F\to V$ agrees with $F$ for each $b\in B$.
  \item\label{sublem:poly-perturb:subm} The smooth map
\[
j_G: U_F\times B \ni (p,b) \mapsto j^r(\widetilde G(b)|_{\mathcal X(\kappa_0)})(p)\in J^r(U_F\cap\mathcal X_\kappa,V\cap\mathcal Y_\kappa)
\]
is a submersion on $j^{-1}_G(K)$.
\end{enumerate}
\end{lemma}
\begin{proof}
We may assume $\varphi(U)\subset\mathbb H^{|\mathcal I|}$ is a bounded open subset.
Hence, by \Cref{lem:mkH-preserve}, there is an open subset $B'\subset P^r(\mathcal I,\mathcal J)$ with $B'\subset P^r_{\varphi(U)}(\mathcal I,\mathcal J)$ and $0\in\overline{B'}$.
We write $E:=J^r(U\cap\mathcal X_\kappa,V\cap\mathcal Y_\kappa)$, and set
\[
\begin{gathered}
\alpha:E\to U\cap\mathcal X(\kappa_0)\hookrightarrow U \\
\beta:E\to V\cap\mathcal Y(\kappa_0)\hookrightarrow V
\end{gathered}
\]
to be the projection maps.
Take smooth functions $\rho_U:U\to[0,1]$ and $\rho_V:V\to[0,1]$ whoose supports are compact and which value identically $1$ near compact subsets $\alpha(K)\subset U$ and $\beta(K)\subset V$ respectively.
We define a smooth function
\[
G':U_F\times B'\to\mathbb H^{|\mathcal J|}
\]
by
\begin{equation}
\label{eq:prf:poly-perturb:perturb}
G'(p,b) = \psi F(p) + \rho_U(p)\rho_V(F(p)) b(\varphi(p))\,.
\end{equation}
We define a subset $A\subset U_F$ to be the support of the composition
\[
U_F
\xrightarrow{(\mathrm{id},F)} U\times V
\xrightarrow{\rho_U\cdot\rho_V} [0,1]\,.
\]
Since the map $(\mathrm{id},F):U_F\to U\times V$ is a closed embedding, $A$ is compact.
We have
\[
G'(A \times \{0\}) = \psi F(A) \subset \psi(V)\,,
\]
so that we can choose an open subset $B''\subset B'\subset P^r(\mathcal I,\mathcal J)$ such that $0\in\overline{B''}$ and $G'(A\times B'')\subset\psi(V)$.
Notice that we also have
\[
G'((U_F\setminus A)\times B'')
\subset \psi F(U_F)
\subset \psi(V)\,.
\]
Hence, $G'$ restricts to a smooth function
\[
G'':U_F\times B'' \to \psi(V)\,.
\]
Finally, take a sufficiently small $\varepsilon>0$ so that $\rho_V\psi^{-1}:\psi(V)\to[0,1]$ values $1$ identically on the $\varepsilon$-neighborhood of $\psi\beta(K)\subset \psi(V)\subset\mathbb H^{|\mathcal J|}$, and define
\[
\begin{gathered}
B := \left\{b\in B''\mid \forall x\in\varphi(A):|b(x)|<\varepsilon \right\} \\
G := \psi^{-1}G''|_{U_F\times B} :U_F\times B\to V\,.
\end{gathered}
\]

We show that $B$ and $G$ satisfy the required properties.
First, it is easily verified that $G$ induces a well-defined map
\[
\widetilde G:B\to C^\infty(U_F\cap\mathcal X,V\cap\mathcal Y)\,.
\]
Notice that, by \Cref{prop:whit-comp}, the maps
\[
\begin{split}
C^\infty(U_F\cap\mathcal X,V\cap\mathcal Y)
&\simeq C^\infty(\varphi(U_F)\cap\mathbb H^{\mathcal I},\psi(V)\cap\mathbb H^{\mathcal J}) \\
&\hookrightarrow C^\infty(\varphi(U_F),\mathbb R^{|\mathcal J|})\,,
\end{split}
\]
are topological embeddings.
Thus, to see $\widetilde G$ is continuous (that is \ref{sublem:poly-perturb:conti}), it suffices to show the composition
\[
B
\xrightarrow{\widetilde G} C^\infty(U_F\cap\mathcal X,V\cap\mathcal Y)
\hookrightarrow C^\infty(\varphi(U_F),\mathbb R^{|\mathcal J|})
\]
is continuous.
Since the addition on $C^\infty(\varphi(U_F),\mathbb R^{|\mathcal J|})$ is a continuous operation, this follows from the formula \eqref{eq:prf:poly-perturb:perturb} and \Cref{lem:conti-perturb}.
In the same point of view, we can also verify \eqref{sublem:poly-perturb:zero} while \ref{sublem:poly-perturb:origin} immediately follow from the construction of $B$ and $G$.

Finally, we verify \ref{sublem:poly-perturb:subm}.
Take an open neighborhood $N_U\subset U$ of $\alpha(K)$ so that we have $\rho_U|_{N_U}\equiv 1$.
We also set
\[
N_V := \{q\in V\mid d(\psi(q),\psi\beta(K)) < \varepsilon \}\,,
\]
where $\varepsilon>0$ is the number appearing in the construction of $B$, and $d$ is the standard metric on $\mathbb H^{|\mathcal J|}$.
Notice that, by the definition of $\varepsilon$, we have $\rho_V|_{N_V}\equiv 1$.
Moreover, by the definition of $B$, $G(p,b)\in\beta(K)$ implies $F(p)\in N_V$; indeed, we have
\[
d(\psi G(p,b),\psi F(p)) \le |b(p)| < \varepsilon
\]
if $p\in A$, and $G(p,b)= F(p)$ otherwise.
It follows that for the map
\[
j_G:U_F\times B \ni (p,b) \mapsto j^r(\widetilde G(b))(p)\in  J^r(U_F\cap\mathcal X_\kappa, V\cap\mathcal Y_\kappa)\,,
\]
we have
\[
j_G^{-1}(K)\subset (\alpha(K)\cap F^{-1}(N_V))\times B \subset (N_U\cap F^{-1}(N_V))\times B\,.
\]
Now, for each $(p,b)\in (N_U\cap F^{-1}(N_V))\times B$, the map $G$ is given by
\[
G(p,b) = \psi^{-1}(\psi F(p) + h(\varphi(p)))\,,
\]
so that \ref{sublem:poly-perturb:subm} follows from \Cref{lem:poly-perturb-subm}.
\end{proof}

\subsection{Relative transversality theorem}
\label{sec:rel-transv-thm}

In this section, we prove the following theorem:

\thmJetTransversal

We first prove the last part of the theorem.
The following lemma is a generalization of Lemma II.4.14 in \cite{GG73}.

\begin{lemma}
\label{lem:multijet-open}
In the setting of \Cref{thm:multijet-transv}, suppose $W\subset J^r_{\mathbf n}(\mathcal X,\mathcal Y)$ is a submanifold whose image under $J^r_{\mathbf n}(\mathcal X,\mathcal Y)\twoheadrightarrow \mathcal X^{(\mathbf n)}$ is compact.
Then the subset
\[
\mathcal T_W :=\{F\in C^\infty(\mathcal X,\mathcal Y)\mid j^r_{\mathbf n}F\pitchfork W\} \subset C^\infty(\mathcal X,\mathcal Y)
\]
is open.
\end{lemma}
\begin{proof}
For each $\vec p=(p^\nu)_{\nu\in|\mathbf n|}\in \mathcal X^{(\mathbf n)}$, choose a family
\[
\left\{ U(\nu)\ \middle|\ \nu\in|\mathbf n| \right\}
\]
of open neighborhoods $U(\nu)\subset |\mathcal X|$ of $p^\nu$.
In particular, since points $p^\nu$ are all distinct, we can choose the family so that $U(\nu)\cap U(\nu')=\varnothing$ whenever $\nu\neq\nu'$.
Also, choose an open neighborhood $U'(\nu)\subset\mathcal X\double(\kappa_0(\nu)\double)$ of each $p^\nu$ so that $U'(\nu)$ has the compact closure in $U(\nu)$.
We set
\[
A_{\vec p} := \raisedunderop\prod{\nu\in|\mathbf n|} \overline{U'(\nu)} \subset\mathcal X^{(\mathbf{n})}\,.
\]
Clearly $A_{\vec p}$ is a compact subset of $\mathcal X^{\mathbf n}$, and the interiors of $A_{\vec p}$ cover the whole $\mathcal X^{(\mathbf{n})}$ when $\vec p$ runs all over $\mathcal X^{(\mathbf n)}$.
Let us denote by $\alpha:J^r_{\mathbf{n}}(\mathcal X,\mathcal Y)\to\mathcal X^{(\mathbf n)}$ the source map.
Since $\alpha(W)$ is compact by the assumption, we can choose a finite sequence $\vec p_1,\dots,\vec p_k\in \mathcal X^{(\mathbf n)}$ such that
\[
\alpha(W) \subset \bigcup_{i=1}^k \operatorname{int} A_{p_i}\,.
\]
We put $A_i := A_{\vec p_i}$ for $1\le i\le k$, and define
\[
\mathcal T_{W,i} := \left\{F\in C^\infty(\mathcal X,\mathcal Y)\mid j^r_{\mathbf n}F\pitchfork W\,\text{on}\, W\cap\alpha^{-1}(A_i) \right\}\,.
\]
Then, since $W \subset \bigcup_i\alpha^{-1}(A_i)$, we obtain
\[
\mathcal T_W = \bigcap_{i=1}^k \mathcal T_{W,i}\,.
\]
It follows that, in order to obtain the result, it suffices to show each $\mathcal T_{W,i}$ is open in $C^\infty(\mathcal X,\mathcal Y)$.

We shall show $\mathcal T_{W,i}$ is open.
Consider the set
\[
\mathcal T'_{W,i} := \left\{ G\in C^\infty(\mathcal X^{\mathbf n},\mathcal J^r(\mathcal X,\mathcal Y)^{\mathbf n})\ \middle|\ G\pitchfork W\,\text{on}\, W\cap\alpha'^{-1}(A_i)\right\}\,,
\]
where $\alpha':\mathcal J^r(\mathcal X,\mathcal Y)^{\mathbf n}\to \mathcal X^{\mathbf n}$ is the canonical surjection.
Then since $A_p$ is a compact subset of $\mathcal X^{\mathbf n}$, $\mathcal T'_{W,i}$ is open in $C^\infty(\mathcal X^{\mathbf n},\mathcal J^r(\mathcal X,\mathcal Y)^{\mathbf n})$ by \Cref{prop:transv-open}.
Recall that we have a continuous map
\[
j'^r_{\mathbf n}:C^\infty(\mathcal X,\mathcal Y) \to C^\infty(\mathcal X^{\mathbf n},\mathcal J^r(\mathcal X,\mathcal Y)^{\mathbf n})
\]
defined in \Cref{lem:multijet-conti}.
Since $\alpha'^{-1}(A_i)$ lies in $J^r_{\mathbf n}(\mathcal X,\mathcal Y)\subset \mathcal J^r(\mathcal X,\mathcal Y)^{\mathbf n}$, we obtain
\[
T_{W,i} = (j'^r_{\mathbf n})^{-1}(T'_{W,i})\,,
\]
which is open.
\end{proof}

\begin{proof}[Proof of \Cref{thm:multijet-transv}]
Let us denote by $\pi:J^r_{\mathbf n}(\mathcal X,\mathcal Y)\to\mathcal X^{(\mathbf n)}\times\mathcal Y^{\mathbf n}$ the projection of the fibration.
Since $\mathcal X$ and $\mathcal Y$ be excellent arrangements, we can choose data
\begin{itemize}
  \item a countable open cover $\{W_\mu\}_\mu$ of $W$ with compact closure $\overline W_\mu$ in $J^r_{\mathbf n}(\mathcal X,\mathcal Y)$;
  \item a family $\left\{(U^\nu_\mu,\varphi^\nu_\mu,\mathcal I^\nu_\mu)\ \middle|\ \nu\in|\mathbf n|,\mu\in\mathbb N\right\}$ of $\mathcal X$-charts such that each $(U^\nu_\mu,\varphi^\nu_\mu,\mathcal I^\nu_\mu)$ is of scope $\kappa_0(\nu)$;
  \item a family $\left\{(V^\nu_\mu,\psi^\nu_\mu,\mathcal J^\nu_\mu)\ \middle|\ \nu\in|\mathbf n|,\mu\in\mathbb N\right\}$ of $\mathcal Y$-charts such that each $(V^\nu_\mu,\psi^\nu_\mu,\mathcal J^\nu_\mu)$ is of scope at most $\kappa_0(\nu)$;
\end{itemize}
so as to satisfy the following properties:
\begin{enumerate}[label={\rm(\alph*)}]
  \item\label{cond:prf:mjtransv:disj} $U^\nu_\mu\cap U^{\nu'}_\mu=\varnothing$ unless $\nu=\nu'\in|\mathbf n|$.
  \item\label{cond:prf:mjtransv:prod} $\overline W_\mu\subset W\cap\pi^{-1}(U^{\mathbf n}_\mu\times V^{\mathbf n}_\mu)$, here we write
\[
\begin{split}
U^{\mathbf n}_\mu
&:=\raisedunderop\prod\nu(U^\nu_\mu\cap \mathcal X(\kappa_0(\nu)))\subset\mathcal X^{(\mathbf n)}\,,\\
V^{\mathbf n}_\mu
&:=\raisedunderop\prod\nu(V^\nu_\mu\cap \mathcal Y(\kappa_0(\nu)))\subset\mathcal Y^{\mathbf n}\,;
\end{split}
\]
  \item\label{cond:prf:mjtransv:cpt} The closure $\overline U^\nu_\mu\subset |\mathcal X|$ is compact.
\end{enumerate}
Put
\[
\mathcal T_\mu
:= \left\{f\in C^\infty(\mathcal X,\mathcal Y)\mid j^r_{\mathbf n}f\pitchfork W \,\text{on}\, \overline W_\mu \right\}\,.
\]
Then, since we have $\mathcal T_W = \bigcap_\mu \mathcal T_\mu$, to prove $\mathcal T_W$ is residual, it suffices to show each $\mathcal T_\mu$ is open and dense.
Notice that if $W$ is compact, we can choose $\{W_\mu\}_\mu$ to be a finite open covering of $W$, from which the last assertion follows.

Since \Cref{lem:multijet-open} implies $\mathcal T_\mu$ is open, it remains to show $\mathcal T_\mu$ is dense.
Suppose we are given $f\in C^\infty(\mathcal X,\mathcal Y)$.
We show $f$ can be approximated by functions $g$ with $g\pitchfork W$ on $\overline W_\mu$.
Let us denote by $\overline W^\nu_\mu$ the image of $\overline W_\mu$ under the projection $J^r_{\mathbf n}(\mathcal X,\mathcal Y)\twoheadrightarrow J^r(\mathcal X_{\kappa(\nu)},\mathcal Y_{\kappa(\nu)})$ to the component of $\nu\in|\mathbf n|$.
Then, by \Cref{lem:poly-perturb}, we obtain an open subset $B^\nu_\mu\subset P^r(\mathcal I^\nu_\mu,\mathcal J^\nu_\mu)$ and a smooth map
\[
G^\nu_\mu:(U^\nu_\mu\cap f^{-1}V^\nu_\mu)\times B^\nu_\mu \to V^\nu_\mu
\]
satisfying the properties in \Cref{lem:poly-perturb} for $K=\overline W^\nu_\mu\subset J^r(U^\nu_\mu\cap\mathcal X_{\ge\kappa_0(\nu)},V^\nu_\mu\cap\mathcal Y_{\ge\kappa_0(\nu)})$.
We put $B^{\mathbf n}_\mu:=\prod_\nu B^\nu_\mu$, and define a map $G_\mu:|\mathcal X|\times B^{\mathbf n}_\mu \to |\mathcal Y|$ by
\[
G_\mu(p,\vec{b}):=
\begin{cases}
G^\nu_\mu(p,b^\nu) &\quad\text{when}\ p\in U^\nu_\mu\cap f^{-1}(V^\nu_\mu) \\
f(p) &\quad\text{otherwise}.
\end{cases}
\]
The map $G_\mu$ is well-defined by \ref{cond:prf:mjtransv:disj} above and smooth by the property \ref{sublem:poly-perturb:origin} in \Cref{lem:poly-perturb}.
Moreover, it induces a well-defined map $\widetilde G_\mu:B^{\mathbf n}_\mu\to C^\infty(\mathcal X,\mathcal Y)$ by \ref{sublem:poly-perturb:conti}.
Note that the property \ref{sublem:poly-perturb:conti} also guarantees that $\widetilde G_\mu$ is continuous.
Indeed, for each $\vec h_0=(h^\nu_0)_\nu\in B^{\mathbf n}_\mu$ and an open subset $\Phi\subset J^q(|\mathcal X|,|\mathcal Y|)$ with $\widetilde G_\mu(\vec h_0)\in M(\Phi)\subset C^\infty(\mathcal X,\mathcal Y)$, we can consider an open subset
\[
\Phi^\nu := \Phi\cap J^q(U^\nu_\mu,|\mathcal Y|)\subset J^r(|\mathcal X|,|\mathcal Y|)
\]
for each $\nu\in|\mathbf n|$.
Since the function $\widetilde G^\nu_\mu:B^\nu_\mu\to C^\infty(U^\nu_\mu\cap\mathcal X,V^\nu_\mu\cap\mathcal Y)$ in \ref{sublem:poly-perturb:conti} of \Cref{lem:poly-perturb} is continuous, we can choose a neighborhood $N^\nu$ of $h^\nu_0$ on $B^\nu_\mu$ so that $G^\nu_\mu(N^\nu)\subset M(\Phi^\nu)$.
We set $N^{\mathbf n}:=\prod_\nu N^\nu$, then we have
\[
\widetilde G_\mu(N^{\mathbf n})
\subset\left\{ g\in C^\infty(\mathcal X,\mathcal Y)\ \middle|\ g|_{U^\nu_\mu}\in M(\Phi^\nu)\,, g|_{\text{else}}\equiv f|_{\text{else}}\equiv \widetilde G_\mu(\vec b_0)|_{\text{else}}\right\}\,.
\]
This directly implies $N^{\mathbf n}$ is a neighborhood of $\vec h_0$ on $B^{\mathbf n}_\mu$ with $\widetilde G_\mu(N^{\mathbf n})\subset \widetilde M(\Phi)$ so that $\widetilde G_\mu$ is continuous.
Note also that $B^{\mathbf n}_\mu$ has $0$ as an accumulation point, and $\widetilde G_\mu$ extends to $B^{\mathbf n}_\mu\cup\{0\}$ by $\widetilde G_\mu(0):= f$.

In order to construct an approximation of $f$ in $\mathcal T_\mu$, we want to use \Cref{theo:param-transv}.
To do this, we define a map $j_G:\mathcal X^{(\mathbf n)}\times B^{\mathbf n}_\mu\to J^r_{\mathbf n}(\mathcal X,\mathcal Y)$ by $j_G(\vec p,\vec b)=\left(j^\nu_G(\vec p,\vec b)\right)_\nu$ with
\[
j^\nu_G(\vec p,\vec b) := \begin{cases}
j^r(\widetilde G^\nu_\mu(b^\nu))(p^\nu) &\quad\text{when}\ p^\nu\in U^\nu_\mu\cap f^{-1}(V^\nu_\mu) \\
j^r(f)(p^\nu) &\quad\text{otherwise}\,.
\end{cases}
\]
Clearly we have $j_G^{-1}(\overline W_\mu)\subset (U^{\mathbf n}_\mu\cap f^{-1}(V^{\mathbf n}_\mu))\times B^{\mathbf n}$.
The restriction $j_G|_{U^{\mathbf n}_\mu\times B^{\mathbf n}}$ is obtained as the product of maps
\[
j_{G^\nu_\mu}:(U^\nu_\mu\cap f^{-1}(V^\nu_\mu))\times B^\nu \to J^r(U^\nu_\mu\cap f^{-1}(V^\nu_\mu)\cap\mathcal X_{\kappa(\nu)},V^\nu_\mu\cap\mathcal Y_{\kappa(\nu)}\,.
\]
Hence, the property \ref{sublem:poly-perturb:subm} in \Cref{lem:poly-perturb} implies that $j_G$ is a submersion on $j_G^{-1}(\overline W_\mu)$.
Now, $j_G$ is the adjoint map of the composition
\[
B^{\mathbf n}
\xrightarrow{\widetilde G} C^\infty(\mathcal X,\mathcal Y)
\xrightarrow{j^r_{\mathbf n}} C^\infty(\mathcal X^{\mathbf n},J^r_{\mathbf n}(\mathcal X,\mathcal Y))\,.
\]
Thus, \Cref{theo:param-transv} implies that the subset
\[
\mathcal B:=\left\{\vec b\in B^{\mathbf n}\ \middle|\ j^r_{\mathbf n}\widetilde G(\vec b)\pitchfork W\ \text{on}\ \overline W_\mu\right\}
\]
is dense in $B^{\mathbf n}$.
Since $0$ is an accumulation point of $B$ and $\widetilde G:B\cup\{0\}\to C^\infty(\mathcal X,\mathcal Y)$ is continuous with $\widetilde G(0)=f$, it follows that $\widetilde G(0)=f$ is arbitrarily approximated by functions in $\{\widetilde G(\vec b)\mid\vec b\in\mathcal B\}\subset\mathcal T_\mu$ as required.
\end{proof}

\section{Application: embedding theorem}
\label{sec:appemb}

In this last section, we discuss an application of \Cref{thm:multijet-transv}; the embedding theorem of manifolds with faces.
This was considered and proved for compact $\langle n\rangle$-manifolds in Proposition 2.1.7 in the paper \cite{Lau00}; the proof is due to a construction of actual embeddings by hand.
In contrast, we take a more systematic approach to the theorem.
We will investigate the space $C^\infty(X,\mathbb H^I)$ and, using \Cref{thm:multijet-transv}, deduce that it contains the nonempty subset of embeddings.
As a result, we will obtain more or less general version of the embedding theorem; non-compact cases and general edgings will be covered.

\subsection{Admissible maps}
\label{sec:adm-map}

Before stating our embedding theorem, we need to discuss an additional condition on maps along edgings.
This is because it is not easy to require maps directly to be immersive around corners in terms of transversality.
For example, take $X=Y=I^2$, the unit square in the Euclidean plane.
We have a canonical identification:
\[
\begin{array}{ccc}
\operatorname{bd} I^2 & \cong & \{(i,a)\mid i=1,2\ \text{and}\ a=0,1\} \\
\{(x_1,x_2)\in I^2\mid x_i=a\} & \leftrightarrow & (i,a)\,.
\end{array}
\]
Put $A:=\{(1,a)\mid a=0,1\}\subset\operatorname{bd}I^2$, and consider the following edging $\beta$ of $X=I^2$ with $I^2$:
\[
\beta:\operatorname{bd}X=\operatorname{bd}I^2\supset \{(1,a)\mid a=0,1\}
\hookrightarrow\operatorname{bd}I^2=\operatorname{bd}Y
\]
Then, a map $F:X\to Y$ is along $\beta$ if and only if it preserves ``the left and right edges'' of the square.
Let $F:X\to Y$ be a smooth map defined below:
\[
F(x_1,x_2) := \left(x_1^2, \frac{x_1}2 (1-x_2) + x_2\right)
\]
Clearly, $F$ is along $\beta$ above, and the image of $F$ is depicted in Figure.~\ref{fig:nonadmis}:
\begin{figure}
\centering
\includegraphics[width=4cm]{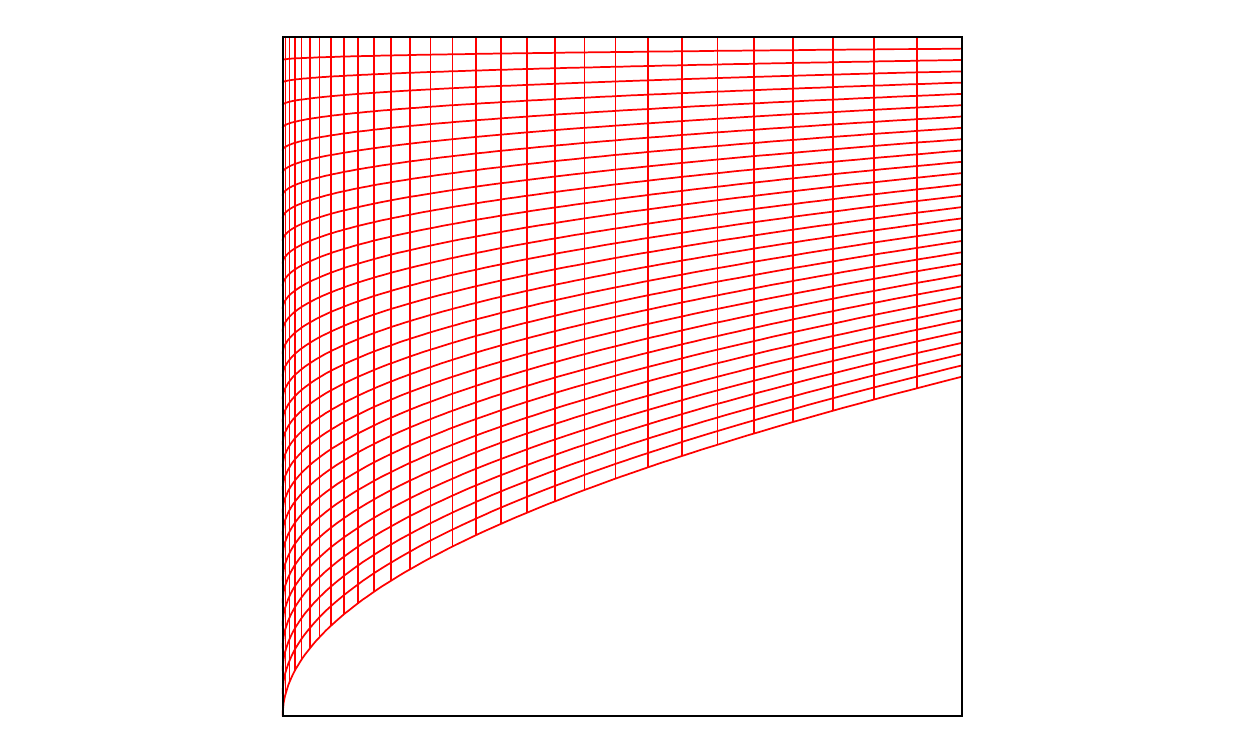} 
\caption{The image of the map $F:I^2\to I^2$}
\label{fig:nonadmis}
\end{figure}
The problem is that $F$ itself is not immersive at $(0,0)\in X=I^2$ while the restriction $F|_{0\times I}:0\times I\to 0\times I$ is.
Indeed, recall that the immersivity of a map involves its first jet.
In particular, since we think of $F$ as a relative map, we need to consider relative jets rather than the usual ones.
Put $C:=0\times I\in\operatorname{bd}I^2$ and $\kappa:=[C,I^2]\in(\Gamma_{I^2})^{[1]}$.
Then, the immersivity of $F$ at $0\times I$ is investigate by seeing the first relative jet
\[
j^1_\kappa F:0\times I\to J^1_\kappa(\widetilde X^\beta_{\sparallel},\widetilde Y)\,.
\]
The difficulty arise here; the fiber of the bundle is not so simple that it is not obvious that jets of a specified corank form a submanifold of $J^1_\kappa(\widetilde X^\beta_{\sparallel},\widetilde Y)$.
This is an obstruction to use a parallel argument of the classical embedding theorem.

To avoid it, we consider an additional condition which guarantees maps ``behave well'' around corners.

\begin{definition}
Let $X$ be a manifold with finite faces equipped with an edging $\beta$ with $Y$.
Then, a smooth map $F:X\to Y$ is said to be admissible along $\beta$ if it satisfies the following conditions:
\begin{enumerate}[label={\rm(\roman*)}]
  \item $F$ is along $\beta$; i.e. $F\in\mathcal F^\beta(X,Y)$.
  \item For each $\sigma\in\Gamma^\beta_X$ and for every $\tau\in\Gamma_Y$, we have $(F|_{\overline\partial_\sigma X})\pitchfork\overline\partial_\tau Y$ in $Y$.
\end{enumerate}
\end{definition}

As expected, for an admissible map, we can obtain its corank by seeing its restriction to any corners.

\begin{lemma}
\label{lem:adm-bdcorank}
Let $X$ and $Y$ be manifolds with finite faces, and let $\beta$ be an edging of $X$ with $Y$.
Suppose we are given an admissible map $F:X\to Y$ along $\beta$.
Then, for each $p\in X$, the corank, i.e. the minimum of the dimensions of the kernel and the cokernel, of the map
\begin{equation}
\label{eq:adm-bdcorank}
d_p(F|_{\overline\partial^\beta_\tau X}):T_p\overline\partial^\beta_\tau X\to T_{F(p)}\overline\partial_\tau Y
\end{equation}
does not depend on $\tau\in\Gamma_Y$ with $p\in\overline\partial^\beta_\tau X$.
\end{lemma}
\begin{proof}
We show the corank of the map \eqref{eq:adm-bdcorank} equals that of the map
\[
d_pF:T_pX\to T_{F(p)}Y\,.
\]
Note that, in general, the corank of a linear map $f:V\to W$ between vector spaces equals the minimum of the dimensions of $\kernel f$ and $\coker f$.
Hence, it suffices to prove that the two maps have isomorphic kernels and cokernels.
We have the following map of short exact sequences:
\[
\xymatrix{
  T_p\overline\partial^\beta_\tau X \ar[d]_{d_p(F|_{\overline\partial^\beta_\tau X})} \ar@{>->}[]+R+(1,0);[r] & T_p X \ar@{->>}[r] \ar[d]_{d_pF} & T_p X/T_p\overline\partial^\beta_\tau X \ar[d] \\
  T_{F(p)}\overline\partial_\tau Y \ar@{>->}[]+R+(1,0);[r] & T_{F(p)}Y \ar@{->>}[r] & T_{F(p)}Y/T_{F(p)}\overline\partial_\tau Y }
\]
The right vertical map is an epimorphism since $F$ is admissible.
Moreover, it is actually an isomorphism; indeed, one can verify two embeddings $\overline\partial^\beta_\tau X\hookrightarrow X$ and $\overline\partial_\tau Y\hookrightarrow Y$ are of the same codimension.
Therefore, the result follows from the snake lemma.
\end{proof}

Here are typical examples of admissible maps:

\begin{example}
Let $X$ be a manifold with finite faces, and consider an edging $\beta$ of $X$ with the unit interval $I$.
Then, a function $f:X\to I$ along $\beta$ is admissible if and only if it has no critical value on the boundary $\partial I$.
\end{example}

\begin{example}
\label{ex:proj-admis}
Let $X$ and $Y$ be two manifolds with finite faces.
Define a partial map $\beta_Y:\operatorname{bd}(X\times Y)\to\operatorname{bd}Y$ to be the composition
\[
\beta_Y:\operatorname{bd}(X\times Y)
\cong (\operatorname{bd}X)\amalg(\operatorname{bd}Y)
\supset\operatorname{bd}Y
\xrightarrow{=}\operatorname{bd}Y\,,
\]
which defines an edging of $X\times Y$ with $Y$.
Then, the projection $X\times Y\to Y$ is admissible along $\beta_Y$.
\end{example}

For admissibility, we have the following useful criterion:

\begin{proposition}
\label{prop:corner-transv}
Let $X$ be a manifold with finite faces equipped with an edging $\beta$ with $Y$.
Then, a smooth map $F:X\to Y$ is admissible along $\beta$ if and only if it satisfies the following conditions:
\begin{enumerate}[label={\rm(\roman*)}]
  \item\label{cond:corner-transv:transv} For each connected face $D\in\operatorname{bd}Y$, we have $F\pitchfork D$.
  \item\label{cond:corner-transv:prescodim} $F$ carries $\partial^\beta_\tau X$ into $\partial_\tau Y$ for each $\tau\in\Gamma_Y$.
\end{enumerate}
\end{proposition}
\begin{proof}
Obviously, $\beta$-admissibility implies the condition \ref{cond:corner-transv:transv}.
Thus, to see the required equivalence, it suffices to show that if \ref{cond:corner-transv:transv} holds, the following two are equivalent for each $p\in X$:
\begin{enumerate}[label={\rm(\alph*)}]
  \item\label{cond:prf:corner-transv:admis} For each $\sigma\in\Gamma^\beta_X$ with $p\in\overline\partial_\sigma X$, and for every $\tau\in\Gamma_Y$, $(F|_{\overline\partial_\sigma X})\pitchfork\overline\partial_\tau Y$ at $p$.
  \item\label{cond:prf:corner-transv:codim} $p\in\partial^\beta_\tau X$ implies $F(p)\in\partial_\tau Y$ for $\tau\in\Gamma_Y$.
\end{enumerate}

First, we prove the equivalence in the case $D(\beta)=\operatorname{bd}X$, or equivalently $\Gamma^\beta_X=\{X\}$.
For $p\in X$, take coordinate charts $\varphi:U\to\mathbb H^{\langle m|k\rangle}$ and $\psi:V\to\mathbb H^{\langle n|l\rangle}$ on $X$ and $Y$ centered at $p$ and $F(p)$ respectively.
Write $\psi F\varphi^{-1}=:(F_1,\dots,F_n)$.
Then, the condition \ref{cond:prf:corner-transv:admis} is satisfied if and only if the elements
\[
dF_{l+1},\dots,dF_n\in T^\ast_0\mathbb H^{\langle m|k\rangle}\simeq T^\ast_p X
\]
are linearly independent.
Since $F_j(0)=0$ is the minimal value of $F_j$ for $j=l+1,\dots,n$, we have $\frac{\partial F_j}{\partial x_i}(0)=0$ for $i=1,\dots,k$ and $j=l+1,\dots,n$.
Hence, the condition \ref{cond:prf:corner-transv:admis} is equivalent to that the matrix
\begin{equation}
\label{eq:prf:corner-transv:Fmat}
\begin{pmatrix}
\dfrac{\partial F_{l+1}}{\partial x_{k+1}}(0) & \cdots & \dfrac{\partial F_n}{\partial x_{k+1}}(0) \\
\vdots & \ddots & \vdots \\
\dfrac{\partial F_{l+1}}{\partial x_m}(0) & \cdots & \dfrac{\partial F_n}{\partial x_m}(0)
\end{pmatrix}
\end{equation}
has rank $n-l$.
Now, one can find an edging $\beta_p:\operatorname{bd}\mathbb H^{\langle m|k\rangle}\to\operatorname{bd}\mathbb H^{\langle n|l\rangle}$ so that the following square is commutative:
\[
\xymatrix{
  \operatorname{bd}\mathbb H^{\langle m|k\rangle} \ar[d]_{\varphi^\ast} \ar[r]^{\beta_p} & \operatorname{bd}\mathbb H^{\langle n|l\rangle} \ar[d]^{\psi^\ast} \\
  \operatorname{bd} X \ar[r]^\beta & \operatorname{bd} Y }
\]
Since we assumed $D(\beta)=\operatorname{bd}X$, we have $D(\beta_p)=\operatorname{bd}\mathbb H^{\langle m|k\rangle}$.
By the requirement for $\beta_p$ to be an edging, applying appropriate coordinate permutations, we may assume $\beta_p(\{x_{k+s}=0\})=\{y_{l+s}=0\}$.
On the other hand, since $F$ is along $\beta$, the smooth map $\mathbb H^{\langle m|k\rangle}\supset\varphi(U)\xrightarrow{(F_1,\dots,F_n)}\psi(V)\subset\mathbb H^{\langle n|l\rangle}$ is along $\beta_p$.
This implies that all the components of the matrix \ref{eq:prf:corner-transv:Fmat} is zero except for the diagonals which are non-zero under the condition \ref{cond:corner-transv:transv}.
Therefore, we conclude that, under \ref{cond:corner-transv:transv}, the condition \ref{cond:prf:corner-transv:admis} is satisfied if and only if $m-k=n-l$; in other words, we have
\[
\psi F\varphi^{-1}(\varphi(U)\cap\partial_{m-k}\mathbb H^{\langle m|k\rangle})
\subset \partial_{n-l}\mathbb H^{\langle n|l\rangle}
\]
which is precisely the condition \ref{cond:prf:corner-transv:codim}.

For general cases, put
\[
X_\sigma := \overline\partial_\sigma X \setminus \bigcup_{\sigma>\sigma'\in\Gamma^\beta_X}\overline\partial_{\sigma'}X
\]
for each $\sigma\in\Gamma^\beta_X$.
Then, the edging $\beta$ of $X$ is canonically restricted to that of $X_\sigma$, say $\beta_\sigma$, and we have $D(\beta_\sigma)=\operatorname{bd} X_\sigma$ (cf. \Cref{ex:edge-face}).
One should notice that $F$ is admissible if and only if so is its restriction
\[
F|_{X_\sigma}:X_\sigma\to Y
\]
for each $\sigma\in\Gamma^\beta_X$.
Hence, the result follows from the special case above.
\end{proof}

\begin{remark}
The two conditions in \Cref{prop:corner-transv} is described in view of the equality $\mathcal F^\beta(X,Y)=C^\infty(\widetilde X^\beta_{\sparallel},\widetilde Y)$.
On the other hand, according to \Cref{lem:edge-alongLan}, we have another equality $\mathcal F^\beta(X,Y)=C^\infty(\widetilde X,\widetilde Y_{\widetilde\beta})$.
In this point of view, the two conditions can be described in the following forms:
\begin{enumerate}[label={\rm(\roman*')}]
  \item For each $C\in D(\beta)\subset\operatorname{bd}X$, we have $F\pitchfork \beta(C)$ on $C$.
  \item $F$ carries $\partial_\sigma X$ into $\partial_{\widetilde\beta(\sigma)}Y$ for each $\sigma\in\Gamma_X$.
\end{enumerate}
\end{remark}

In practice, to verify the condition \ref{cond:corner-transv:transv}, the following notion is sometimes convenient:

\begin{definition}[cf. boundary defining functions defined in \cite{Joy09}]
Let $X$ be a manifold with faces.
Then, a non-negative smooth function $f:X\to\mathbb R_+$ is said to recognize a connected face $C\in\operatorname{bd}X$ if it satisfies the following conditions:
\begin{enumerate}[label={\rm(\roman*)}]
  \item\label{cond:func-spec:zero} The map $f$ vanishes on $C$.
  \item\label{cond:func-spec:exact} The $1$-form $df$ does not vanish on $C$.
\end{enumerate}
\end{definition}

Thanks to the Collar Neighborhood Theorem (\Cref{thm:collaring}), every connected face admits a non-negative function recognizing it.
Moreover, if $f:X\to\mathbb R_+$ recognizes a connected face $C$, then for every $C$-collaring vector field $\xi$ on $X$, the smooth function $\xi(f):X\to\mathbb R$ is positive on $C$.
Note that, for each $\sigma\in\Gamma_X$ with $\sigma\not\le C$ and for each $p\in\overline\partial_\sigma X\cap C$, since $\xi_p\in T_pX$ belongs to the image of $T_p\overline\partial_\sigma X$, we have
\[
\xi_p(d_pf|_{\overline\partial_\sigma X}) = \xi_p(d_pf) = \xi_p(f) > 0\,.
\]
This implies that the restriction $f|_{\overline\partial_\sigma X}:\overline\partial_\sigma X\to\mathbb R_+$ recognizes faces contained in $\overline\partial_\sigma X\cap C$.

\begin{proposition}
\label{prop:transv-recog}
Let $X$ and $Y$ be manifolds with finite faces, and let $\beta$ be an edging of $X$ with $Y$.
Suppose $F:X\to Y$ be a smooth map along $\beta$.
Then, for every connected face $D\in\operatorname{bd}Y$ of $Y$, the following two conditions are equivalent:
\begin{enumerate}[label={\rm(\alph*)}]
  \item\label{cond:transv-recog:transv} $F$ intersects transversally to $D$ on any $C$ with $C\in D(\beta)\subset\operatorname{bd}X$ and $\beta(C)=D$.
  \item\label{cond:transv-recog:recog} The map $F$ pulls back non-negative functions on $Y$ recognizing $D$ to those on $X$ recognizing every $C$ with $C\in D(\beta)\subset\operatorname{bd}X$ and $\beta(C)=D$.
\end{enumerate}
\end{proposition}
\begin{proof}
First, suppose \ref{cond:transv-recog:transv} is satisfied, and let $g:Y\to\mathbb R_+$ be an arbitrary smooth function recognizing $D$.
It is obvoius that the composition $gF$ vanishes on any connected faces $C\in D(\beta)$ with $\beta(C)=D$.
Moreover, $g$ gives rise to a short exact sequence
\[
\mathbb R\cdot d_qg
\hookrightarrow T^\ast_qY
\twoheadrightarrow T^\ast_qD
\]
for each $q\in D$.
Since we have $F\pitchfork D$ by the condition \ref{cond:transv-recog:transv}, for every $p\in X$ with $F(p)\in D$, the composition
\[
\mathbb R\cdot d_{F(p)}g
\hookrightarrow T^\ast_{F(p)}Y
\xrightarrow{F^\ast} T^\ast_p X
\]
is a monomorphism.
In particular, the pullback $d(gF)=F^\ast(dg)$ does not vanish on any $C$ with $\beta(C)=D$, so that $F$ satisfies \ref{cond:transv-recog:recog}.

Conversely, suppose \ref{cond:transv-recog:recog}.
Take a smooth function $g:Y\to\mathbb R_+$ recognizing $D$, which is done by using \Cref{thm:collaring}.
Let $C$ be any connected face of $X$ with $C\in D(\beta)$ and $\beta(C)=D$.
Then, since $gF$ recognizes $C$ by the assumption, for each $p\in C$, there are isomorphisms
\[
\begin{gathered}
T^\ast_{F(p)}Y \cong T^\ast_{F(p)} D\oplus\mathbb R\cdot d_{F(p)}g \\
T^\ast_pX \cong T^\ast_p C\oplus\mathbb R\cdot d_p(gF)\mathrlap{\,.}
\end{gathered}
\]
We obtain a morphism between short exact sequences:
\[
\xymatrix{
  \mathbb R\cdot d_{F(p)}g \ar[r] \ar[d]_{F^\ast} & T^\ast_{F(p)}Y \ar[r] \ar[d]_{F^\ast} & T^\ast_{F(p)}D \ar[d]^{F^\ast} \\
  \mathbb R\cdot d_p(gF) \ar[r] & T^\ast_pX \ar[r] & T^\ast_pC }
\]
Since the left arrow is clearly a monomorphism, the condition \ref{cond:transv-recog:transv} immediately follows.
\end{proof}

\begin{corollary}
\label{cor:admcompose}
Let $X$ be a manifold with faces equipped with an edging $\beta$ with $Y$.
Suppose, in addition, we also have an edging $\gamma$ of $Y$ with $Z$.
If smooth maps $F:X\to Y$ and $G:Y\to Z$ are admissible along $\beta$ and $\gamma$ respectively, then the composition $GF$ is again admissible along $\gamma\beta$.
\end{corollary}

\begin{example}
Let $\mathcal X$ be a neat arrangement of shape $S$.
Note that, for each $s\in S$, there is a unique edging $\eta_s$ of $\mathcal X(s)$ with $|\mathcal X|$ so that the embedding $\mathcal X(s)\hookrightarrow |\mathcal X|$ is along it.
It is easily verified that the embedding is even admissible.
Now, suppose we have an edging $\beta$ of $\mathcal X$ with $Y$, so that we also have an edging $\beta\eta_s$ of $\mathcal X(s)$ for each $s\in S$.
Then, by \Cref{cor:admcompose}, if a map $F:|\mathcal X|\to Y$ is admissible along $\beta$, the composition
\[
\mathcal X(s)
\hookrightarrow |\mathcal X|
\xrightarrow{F} Y
\]
is also admissible along $\beta\eta_s$.
This is a reason why we did not discuss admissibility of maps on arrangements.
\end{example}

\begin{corollary}
\label{cor:admpaste}
Let $X$, $Y_1$, and $Y_2$ be manifolds with finite faces.
Suppose we have two edging $\beta_1$ and $\beta_2$ of $X$ with $Y_1$ and $Y_2$ respectively such that $D(\beta_1)\cap D(\beta_2)=\varnothing\subset\operatorname{bd}X$, so we have an edging $\beta=\beta_1\amalg\beta_2$ of $X$ with $Y_1\times Y_2$ (see \Cref{ex:edge-paste}).
Then, a smooth map $F=(F_1,F_2):X\to Y_1\times Y_2$ along $\beta$ (by \Cref{lem:edgefunc-prod}, this is equivalent to say $F_i$ is along $\beta_i$ for $i=1,2$) is admissible if and only if both $F_1$ and $F_2$ are admissible.
\end{corollary}
\begin{proof}
Let us denote by $\pi_i:Y_1\times Y_2\to Y_i$ the projection for $i=1,2$, each of which is admissible along the canonical edging of the projection by \Cref{ex:proj-admis}.
Since $F_i=\pi_i F$, \Cref{cor:admcompose} implies that if $F$ is admissible, then so are $F_1$ and $F_2$.
We show the converse using \Cref{prop:corner-transv}.
Note that each connected face of $Y_1\times Y_2$ is of either form
\[
D_1\times\alpha_2
\quad\text{or}\quad \alpha_1\times D_2
\]
for $D_i\in\operatorname{bd}Y_i$ and $\alpha_i\in\pi_0Y_i$.
Hence, we have commutative diagrams below:
\[
\begin{gathered}
\xymatrix{
  T_pX \ar[r]^F \ar[dr]_{F_1} & T_{F(p)}(Y_1\times Y_2) \ar@{->>}[r] \ar[d]^{\pi_1} & T_{F(p)}(Y_1\times Y_2)/T_{F(p)}(D_1\times\alpha_2) \ar[d]^\cong \\
  & T_{F_1(p)}Y_1 \ar@{->>}[r] & T_{F_1(p)}Y_1/T_{F_1(p)}D_1 }
\\
\xymatrix{
  T_pX \ar[r]^F \ar[dr]_{F_2} & T_{F(p)}(Y_1\times Y_2) \ar@{->>}[r] \ar[d]^{\pi_2} & T_{F(p)}(Y_1\times Y_2)/T_{F(p)}(\alpha_1\times D_2) \ar[d]^\cong \\
  & T_{F_2(p)}Y_2 \ar@{->>}[r] & T_{F_2(p)}Y_2/T_{F_2(p)}D_2 }
\end{gathered}
\]
In addition, one can verify the formula
\[
\partial_\tau(Y_1\times Y_2)
= \partial_{\tau_1}Y_1\times\partial_{\tau_2}Y_2
\subset Y_1\times Y_2
\]
for each $\tau=(\tau_1,\tau_2)\in\Gamma_{Y_1\times Y_2}\cong\Gamma_{Y_1}\times\Gamma_{Y_2}$.
Therefore, thanks to \Cref{prop:corner-transv}, the map $F=(F_1,F_2):X\to Y_1\times Y_2$ is admissible as soon as so are $F_1$ and $F_2$.
\end{proof}

As a consequence of \Cref{cor:admpaste}, we can construct an admissible map $X\to\mathbb R_+^{\operatorname{bd}X}$ for every manifold $X$ with finite faces in the following way:
For each $C\in\operatorname{bd}X$, choose a non-negative smooth function $f_C:X\to\mathbb R_+$ recognizing the face $C$.
Notice that, in the moment, $f_C$ is admissible along the edging
\[
\operatorname{bd}X\supset \{C\}\to \{\mathrm{pt}\}=\operatorname{bd}\mathbb R_+\,.
\]
Therefore, by \Cref{cor:admpaste}, the map
\[
F:=(f_C)_C : X\to\mathbb R^{\operatorname{bd}X}_+
\]
is admissible along the identity edging.
The same result will, however, be proved in the next section in more sophisticated way.

\subsection{Genericity of admissible maps}
\label{sec:generic-adm}

We continue the discussion on admissible maps.
In this section, we prove that maps in $\mathcal F^\beta(X,Y)$ are generically admissible.
More precisely, the goal of this section is to prove the following theorem:

\begin{theorem}[cf. Theorem 1.7 in \cite{Ish98}]
\label{thm:adm-residual}
Let $X$ and $Y$ be manifolds with finite faces, and let $\beta$ be an edging of $X$ with $Y$.
Then admissible maps form an open and dense subset in $\mathcal F^\beta(X,Y)$.
\end{theorem} 

We are going to prove \Cref{thm:adm-residual} by using \Cref{thm:multijet-transv}.
To do this, we want to describe two conditions in the remark following \Cref{prop:corner-transv} in terms of transversality.
There is, however, a difficulty on the first condition.
Let $X$, $Y$, and $\beta$ be as in \Cref{thm:adm-residual}, so that \Cref{lem:edge-alongLan} implies $\mathcal F^\beta(X,Y)=C^\infty(\widetilde X,\widetilde Y_{\widetilde\beta})$.
Then, for $\sigma\in\Gamma_X$ and $D\in\operatorname{bd}Y$ with $\widetilde\beta(\sigma)\le D$, the condition $F\pitchfork D$ on $\partial_\sigma X$ is not actually a condition on relative $0$-th jets.
Indeed, the behaviors of $F$ on $\partial_\sigma X$ should be controlled by its relative jets of the form $j^r_{[\sigma,X]}F:\partial_\sigma X\to J^r_{[\sigma,X]}(\widetilde X,\widetilde Y_{\widetilde\beta})$, while we have $J^0_{[\sigma,X]}(\widetilde X,\widetilde Y_{\widetilde\beta})\cong\partial_\sigma X\times\overline\partial_{\widetilde\beta(\sigma)}Y\subset\partial_\sigma X\times D$.
Hence, we need more observations on transversality to faces.

Let $X$, $Y$, and $\beta$ be as above, and suppose we are given $D\in\operatorname{bd}Y$ and for $\sigma\in\Gamma_X$ with $\widetilde\beta(\sigma)\le D$.
We have the first relative jet bundles
\[
\begin{gathered}
J^1_{[\sigma,X]}(\widetilde X,\widetilde Y_{\widetilde\beta})
\to \partial_\sigma X\times\overline\partial_{\widetilde\beta(\sigma)} Y\,, \\
J^1_{[\sigma,X]}(\widetilde X,\widetilde Y_{\widetilde\beta}\cap D)
\to \partial_\sigma X\times\overline\partial_{\widetilde\beta(\sigma)} Y\,.
\end{gathered}
\]
The embedding $\widetilde Y_{\widetilde\beta}\cap D\hookrightarrow\widetilde Y_{\widetilde\beta}$ induces an embedding
\begin{equation}
\label{eq:transedge}
J^1_{[\sigma,X]}(\widetilde X,\widetilde Y_{\widetilde\beta(\sigma)}\cap D)
\hookrightarrow J^1_{[\sigma,X]}(\widetilde X,\widetilde Y_{\widetilde\beta(\sigma)})\,,
\end{equation}
which is also a bundle map over $\partial_\sigma X\times\overline\partial_{\widetilde\beta(\sigma)} Y$.
For brevity, we denote by $N_{D,\sigma}$ the image of the embedding \eqref{eq:transedge}.
Note that the fiber of the bundle $J^1_{[\sigma,X]}(\widetilde X,\widetilde Y_{\widetilde\beta})$ over a point $(p,q)$ can be identified with the space of linear maps $g:T_pX\to T_qY$ which maps $T_p\overline\partial_{\sigma'}X\subset T_pX$ into $T_q\overline\partial_{\widetilde\beta(\sigma')}Y\subset T_qY$ for each $\sigma\le\sigma'\le X\in\Gamma_X$.
Then, the map \eqref{eq:transedge} is determined by the restriction of the map
\[
\mathrm{Hom}_{\mathbb R}(T_pX,T_qD)
\hookrightarrow \mathrm{Hom}_{\mathbb R}(T_pX,T_qY)\,.
\]
In particular, elements of $N_{D,\sigma}$ are jets lying over $\partial_\sigma X$ and transverse to $D\subset Y$.
The next lemma gives a convenient description for $N_{D,\sigma}$:

\begin{lemma}
\label{lem:ND-pb}
In the situation above, suppose $g:Y\to\mathbb R_+$ is a smooth function such that it recognizes the face $D$ and vanishes precisely on $D$.
We denote by $\widetilde{\mathbb R}^{(D)}_+$ the arrangement of shape $\Gamma_Y$ so that for $D'\in\operatorname{bd}Y$,
\[
\widetilde{\mathbb R}^{(D)}_+(D')=
\begin{cases}
\{0\} \quad & \text{if $D=D'$,} \\
\mathbb R_+ \quad & \text{otherwise.}
\end{cases}
\]
Then, the following square is a transverse pullback; that is, the right and bottom arrows intersect transversally, and the square is a pullback.
\begin{equation}
\label{eq:ND-pb:req}
\vcenter{
\xymatrix{
  N_{D,\sigma} \ar@{^(->}[]+R+(1,0);[r] \ar[d] & J^1_{[\sigma,X]}(\widetilde X,\widetilde Y_{\widetilde\beta}) \ar[d]^{g_\ast} \\
  \partial_\sigma X \ar[r] & J^1_{[\sigma,X]}(\widetilde X,(\widetilde{\mathbb R}^{(D)}_+)_{\widetilde\beta}) }}
\end{equation}
\end{lemma}
\begin{proof}
Note first that the assumption on the function $g:Y\to\mathbb R_+$ implies that the square below is a transverse pullback:
\[
\xymatrix{
  D \ar@{^(->}[]+R+(1,0);[r] \ar[d] & Y \ar[d]^g \\
  \{0\} \ar[r] & \mathbb R_+ }
\]
This gives rise to the following transverse pullback square
\begin{equation}
\label{eq:prf:ND-pb:pbtot}
\vcenter{
\xymatrix{
  J^1(X,D) \ar@{^(->}[]+R+(1,0);[r] \ar[d] & J^1(X,Y) \ar[d]^{g_\ast} \\
  X \ar[r] & J^1(X,\mathbb R_+) }}
\end{equation}
Then, it is fiberwisely verified that the square \eqref{eq:prf:ND-pb:pbtot} restricts to the required transverse pullback square \eqref{eq:ND-pb:req}.
\end{proof}

Using the description in \Cref{lem:ND-pb}, we can describe the transversality to faces as the condition on first relative jets:

\begin{lemma}
\label{lem:transedge-eqv}
In the situation above, for a map $F\in\mathcal F^\beta(X,Y)$, the following are all equivalent:
\begin{enumerate}[label={\rm(\alph*)}]
  \item\label{cond:transedge-eqv:transbd} The map $F$ intersects transversally to $D$ on $\partial_\sigma X$.
  \item\label{cond:transedge-eqv:disj} The image $j^1_{[\sigma,X]} F(\partial_\sigma X)$ does not intersect with $N_{D,\sigma}\subset\overline\partial_{\widetilde\beta(\sigma)}Y$.
  \item\label{cond:transedge-eqv:transjet} The first relative jet
\[
j^1_{[\sigma,X]} F:\partial_\sigma X\to J^1_{[\sigma,X]}(\widetilde X,\widetilde Y_{\widetilde\beta})
\]
intersects transversally to $N_{D,\sigma}$.
\end{enumerate}
\end{lemma}
\begin{proof}
If $\partial_\sigma X=\varnothing$ or $\dim Y\le 0$, there is nothing to prove, so we assume $\partial_\sigma X\neq\varnothing$, $Y\neq\varnothing$, and $\dim Y\ge 1$.
Take a smooth function $g:Y\to\mathbb R_+$ recognizing the face $D$ such that
\[
\{q\in Y\mid g(q)=0\} = D\,.
\]
Then, by \Cref{lem:ND-pb}, the square below is a transverse pullback:
\[
\xymatrix{
  N_{D,\sigma} \ar@{^(->}[]+R+(1,0);[r] \ar[d] \ar@{}[dr]|(.4)\pbcorner & J^1_{[\sigma,X]}(\widetilde X,\widetilde Y_{\widetilde\beta}) \ar[d]^{g_\ast} \\
  \mathllap{\partial_\sigma X}\cong J^1_{[\sigma,X]}(\widetilde X,\{0\}) \ar@{^(->}[]+R+(1,0);[r] & J^1_{[\sigma,X]}(\widetilde X,(\mathbb R_+)_{\widetilde\gamma_D\widetilde\beta}) }
\]
Thus, the problem reduces to the case $Y=\mathbb R_+$ and $D=\{0\}\in\operatorname{bd}\mathbb R_+$.

Assume $Y=\mathbb R_+$ and $D=\{0\}$, so that $J^1_{[\sigma,X]}(\widetilde X,(\widetilde{\mathcal R_+})_{\widetilde\beta})$ is a real vector bundle over $\partial_\sigma X$ for each $\sigma\in\Gamma_X$ rather than just a fiber bundle.
If $\widetilde\beta(\sigma)=\{0\}\in\Gamma_Y$ and $\overline\partial_\sigma X\neq\varnothing$, by \Cref{lem:edge-sliceisom}, there is a unique connected face $C\in\operatorname{bd}X$ such that $\beta(C)=\{0\}\subset\mathbb R_+$ and $\sigma\le C\in\Gamma_X$.
Then, for every $p\in\overline\partial_\sigma X$, we have a short exact sequence
\[
J^1_{[\sigma,X]}(\widetilde X,(\widetilde{\mathbb R_+})_{\widetilde\beta})_p
\hookrightarrow T^\ast_p X
\twoheadrightarrow T^\ast_p C
\]
of vector bundles over $\partial_\sigma X$.
In particular, $J^1_{[\sigma,X]}(\widetilde X,(\widetilde{\mathbb R_+})_{\widetilde\beta})$ is a subbundle of $T^\ast X|_{\partial_\sigma X}$.
More precisely, take a smooth function $f:X\to\mathbb R_+$ such that
\begin{itemize}
  \item $\sigma\le C$;
  \item $C\in D(\beta)$ and $\beta(C)=\{0\}\in\operatorname{bd}\mathbb R_+$;
  \item $C=\{p\in X\mid f(p)=0\}$;
  \item $d_pf\neq 0$ for every $p\in C$.
\end{itemize}
Put $E^\sigma:=\mathbb R\cdot df|_{\partial_\sigma X}\subset T^\ast X|_{\partial_\sigma X}$, then we obtain a canonical a canonical isomorphism
\[
T^\ast X|_{\partial_\sigma X}
\cong (T^\ast C|_{\partial_\sigma X})\oplus E^\sigma
\]
together with an identification
\[
J^1_{[\sigma,X]}(\widetilde X,(\widetilde{\mathbb R_+})_{\widetilde\beta})
\cong E^\sigma\,.
\]
For a map $F:X\to\mathbb R_+$ along $\beta$, its first relative jet $j^1_{[\sigma,X]}F$ is identified with the composition
\begin{equation}
\label{eq:prf:transedge-eqv:reljet}
j^1_{[\sigma,X]}\partial_\sigma X
\xrightarrow{dF} T^\ast X|_{\partial_\sigma X}
\xrightarrow{\mathrm{proj.}} E^\sigma\,.
\end{equation}
Now, since we can regard $N_{\{0\},\sigma}$ as the image of the zero-section $\partial_\sigma X\hookrightarrow E^\sigma$, two conditions \ref{cond:transedge-eqv:transbd} and \ref{cond:transedge-eqv:disj} are clearly equivalent.
To see the condition \ref{cond:transedge-eqv:transjet} is also equivalent, we see that the relative jet $j^1_{[\sigma,X]}F$ factors through a submanifold of $E^\sigma$ with boundaries.
Let $f:X\to\mathbb R_+$ be as above, and set
\[
E^{\sigma+}
:= R_+\cdot d_f
\subset E^\sigma\,.
\]
Then, $E^{\sigma+}$ is a submanifold of $E^\sigma$ of codimension $0$ whose boundary can be identified with $N_{\{0\},\sigma}\hookrightarrow E^\sigma$.
It is easily verified that $j^1_{[\sigma,X]}F$, which is identified with the map \eqref{eq:prf:transedge-eqv:reljet}, factors through $E^{\sigma+}$.
Since $E^{\sigma+}$ has nonempty boundaries while $\partial_\sigma X$ does not, the smooth map $j^1_{[\sigma,X]}F$ cannot intersects to the boundary $\partial E^{\sigma+}=N_{\{0\},\sigma}$.
Hence, the condition \ref{cond:transedge-eqv:transjet} is also equivalent to the others.
\end{proof}

\begin{proof}[Proof of \Cref{thm:adm-residual}]
Since the lattice $\Gamma_X$ and the set $\operatorname{bd}Y$ are finite, in view of \Cref{prop:corner-transv} (and the remark following it), it suffices to show that the following subsets of $\mathcal F^\beta(X,Y)=C^\infty(\widetilde X,\widetilde Y_{\widetilde\beta})$ are open and dense:
\[
\begin{gathered}
\mathcal B^\partial_\sigma
:=\left\{F\in\mathcal F^\beta(X,Y)\mid F(\partial_\sigma X)\subset\partial_{\widetilde\beta(\sigma)}Y\right\} \\
\mathcal B^\pitchfork_\sigma
:=\left\{F\in\mathcal F^\beta(X,Y)\mid F|_{\partial_\sigma X}\pitchfork D\right\}\,,
\end{gathered}
\]
where $\sigma\in\Gamma_X$ and $D\in\operatorname{bd}Y$ with $\widetilde\beta(\sigma)\le D\in\Gamma_Y$.

We first show $\mathcal B^\partial_\sigma$ is open and dense.
Recall that we have a canonical diffeomorphism
\[
J^0_{[\sigma,\sigma]}(\widetilde X,\widetilde Y_{\widetilde\beta})
\cong \partial_\sigma X\times\overline\partial_{\widetilde\beta(\sigma)}Y\,.
\]
Hence, a smooth map $F:X\to Y$ along $\beta$ belongs to $\mathcal B^\partial_\sigma$ if and only if its $0$th relative jet
\[
j^0_{[\sigma,\sigma]}F:\partial_\sigma X
\to\partial_\sigma X\times\overline\partial_{\widetilde\beta(\sigma)}Y
\ ;\quad p \mapsto (p,F(p))
\]
does not intersect with $\partial_\sigma X\times\overline\partial_\tau Y$ for any $\tau\in\Gamma_Y$ with $\tau<\widetilde\beta(\sigma)$.
Since $\partial_\sigma X$ is a manifold without boundaries and $\partial_\sigma X\times\overline\partial_\tau Y$ is a set of corners of positive codimensions, one can prove that the following three conditions are equivalent:
\begin{enumerate}[label={\rm(\roman*)}]
  \item $F\in\mathcal B^\partial_\sigma$;
  \item $j^0_{[\sigma,\sigma]}F(\partial_\sigma X)\subset\partial_\sigma X\times\partial_{\widetilde\beta(\sigma)}Y=\partial_\sigma X\times\left(\overline\partial_{\widetilde\beta(\sigma)}Y\setminus\bigcup_{\tau<\widetilde\beta(\sigma)}\overline\partial_\tau Y\right)$;
  \item $j^0_{[\sigma,\sigma]}F\pitchfork(\partial_\sigma X\times\overline\partial_\tau Y)$ for every $\tau\in\Gamma_Y$ with $\tau<\widetilde\beta(\sigma)$.
\end{enumerate}
The second condition implies $\mathcal B^\partial_\sigma\subset\mathcal F^\beta(X,Y)$ is open, and the thrid, \Cref{thm:multijet-transv}, and \Cref{cor:edgefunc-Baire} imply it is dense.
Using \Cref{lem:transedge-eqv}, one can also prove in a really similar way that $\mathcal B^\pitchfork_{\sigma,D}$ is open and dense.
\end{proof}

\subsection{Embedding theorem}
\label{sec:emb-theorem}

In this last section, we prove the embedding theorem.

\thmImmersion
\begin{proof}
By the assumption on dimensions, in view of \Cref{lem:adm-bdcorank}, an admissible map $F:X\to Y$ along $\beta$ is an immersion if and only if the induced map
\begin{equation}
\label{eq:prf:imm-thm:tangent}
d_p(F|_{\partial^\beta_\tau X}):T_p\partial^\beta_\tau X\to T_{F(p)}\overline\partial_\tau Y
\end{equation}
is of corank $0$ for each $\tau\in\Gamma_Y$ and $p\in\partial^\beta_\tau$.
Since we have residually many admissible maps in $\mathcal F^\beta(X,Y)$ by \Cref{thm:adm-residual}, it suffices to show that maps satisfying the condition above form a residual subset in $\mathcal F^\beta(X,Y)$.

Notice that, for each $\tau\in\Gamma_Y$, we have canonical identifications
\begin{equation}
\label{eq:prf:imm-thm:bdl}
J^1_{[\tau,\tau]}(\widetilde X^\tau,\widetilde Y)
\cong J^1(\partial^\beta_\tau X,\overline\partial_\tau Y)
\cong \mathpzc{Hom}_{\mathbb R}(T\partial^\beta_\tau X,T\overline\partial_\tau Y)\,.
\end{equation}
Under this identification, the first relative jet
\[
j^1_{[\tau,\tau]}F:\partial^\beta_\tau X\to J^1_{[\tau,\tau]}(\widetilde X^\beta,\widetilde Y)
\]
for $F\in\mathcal F^\beta(X,Y)$ is identified with the map given by
\[
j^1_{[\tau,\tau]}F(p) = d_p(F|_{\partial^\beta_\tau X})
\in\mathrm{Hom}_{\mathbb R}(T_p\partial^\beta_\tau X,T_{F(p)}\overline\partial_\tau Y)\,.
\]
For each non-negative integer $r\ge0$, jets inducing homomorphisms \eqref{eq:prf:imm-thm:tangent} of corank $r$ forms a subbundle of the bundle \eqref{eq:prf:imm-thm:bdl} over $\partial^\beta_\tau X\times\overline\partial_\tau Y$, which we write
\[
\mathcal L^r_\tau
\subset J^1_{[\tau,\tau]}(\widetilde X^\beta,\widetilde Y)\,.
\]
It is verified that the corank of this embedding is computed by the formula below:
\[
r(r+|\dim \overline\partial_\tau Y-\dim \partial^\beta_\tau X|)\,.
\]
Hence, using the equation $\dim\overline\partial_\tau Y-\dim\partial^\beta_\tau X=\dim Y-\dim X$, we obtain
\[
\codim\mathcal L^r_\tau
= r(r+\dim Y-\dim X)
\ge r(r+\dim X)
> r(1+\dim\partial^\beta_\tau X)\,.
\]
If $r>0$, The last number is strictly greater than the dimension $\partial^\beta_\tau X$ of the domain of the first relative jet $j^1_{[\tau,\tau]}F$.
It follows from \Cref{thm:multijet-transv} that, for residually many $F\in\mathcal F^\beta(X,Y)=C^\infty(\widetilde X^\beta,\widetilde Y)$, $j^1_{[\tau,\tau]}F$ does not intersect with $\mathcal L^r_\tau$ for any $r>0$.
In other words, for those $F$, the map \eqref{eq:prf:imm-thm:tangent} is of corank $0$.
Thus, we obtain the result.
\end{proof}

\begin{theorem}
\label{thm:inj-residual}
Let $X$ and $Y$ be two manifolds with finite faces, and let $\beta$ be an edging of $X$ with $Y$.
Assume we have $2\cdot \dim X+1\le\dim Y$.
Then, injective maps $X\to Y$ along $\beta$ form a residual and, hence, dense subset of the space $\mathcal F^\beta(X,Y)$.
\end{theorem}
\begin{proof}
If a map $F:X\to Y$ along $\beta$ is admissible, \Cref{prop:corner-transv} implies $F(\partial^\beta_\tau X)\subset\partial_\tau Y$ for each $\tau$.
Notice $\partial_\tau Y\cap\partial_{\tau'}Y=\varnothing$ for every $\tau\neq\tau'\in\Gamma_Y$.
It follows that an admissible map $F$ is injective if and only if the restriction
\[
F_\tau:=F|_{\partial^\beta_\tau X}:\partial^\beta_\tau X\to\overline\partial_\tau Y
\]
is injective.
Since admissible maps form a residual subset of $\mathcal F^\beta(X,Y)$ by \Cref{thm:adm-residual}, to obtain the result, it suffices to show that maps $F:X\to Y$ along $\beta$ with $F_\tau$ injective for each $\tau\in\Gamma_Y$ form a residual subset in $\mathcal F^\beta(X,Y)$.

For an element $\tau\in\Gamma_Y$, define a map $\mathbf n_\tau:(\Gamma_Y)^{[1]}\to\mathbb Z_{\ge0}$ by
\[
\mathbf n_\tau(\kappa):=
\begin{cases}
2 \quad &\text{if $\kappa=[\tau,\tau]$,} \\
0 \quad &\text{otherwise.}
\end{cases}
\]
Then, we have a canonical identification
\begin{equation}
\label{eq:prf:inj-residual:J0}
J^0_{\mathbf n_\tau}(\widetilde X^\beta,\widetilde Y)
\cong (\partial^\beta_\tau X)^{(2)}\times(\overline\partial_\tau Y)^{\times2}\,.
\end{equation}
We denote by $\Delta_\tau\subset J^0_{\mathbf n_\tau}(\widetilde X^\beta,\widetilde Y)$ the image of the embedding
\[
(\partial^\beta_\tau X)^{(2)}\times\overline\partial_\tau Y
\xrightarrow{\mathrm{id}\times\Delta}(\partial^\beta_\tau X)^{(2)}\times(\overline\partial_\tau Y)^{\times2}\,.
\]
Note that, under the identification \eqref{eq:prf:inj-residual:J0}, the $0$-th $\mathbf n_\tau$-multijet of $F$ is given by
\[
\begin{array}{rccc}
j^0_{\mathbf n_\tau}F: & (\partial^\beta_\tau X)^{(2)} & \to & J^0_{\mathbf n_\tau}(\widetilde X^\beta,\widetilde Y) \\
& (p,p') & \mapsto & (p,p',F(p),F(p'))\,,
\end{array}
\]
and its intersection with $\Delta_\tau$ corresponds to pairs $(p,p')\in(\partial^\beta_\tau X)^{(2)}$ of distinct points such that $F(p)=F(p')$.
Thus, the restriction $F_\tau:\partial^\beta_\tau X\to\overline\partial_\tau Y$ is injective if and only if $j^0_{\mathbf n_\tau}F$ does not intersect with $\Delta_\tau$.
This is, moreover, equivalent to the condition  $j^0_{\mathbf n_\tau}F\pitchfork\Delta_\tau$; indeed, we have
\[
\begin{multlined}
\codim \Delta_\tau
= \dim \overline\partial_\tau Y
= \dim Y - \codim\overline\partial_\tau Y
= \dim Y - \codim\partial^\beta_\tau X \\
> 2\cdot\dim X -\codim\partial^\beta_\tau X
\ge \dim X
\ge \dim\partial^\beta_\tau X\,.
\end{multlined}
\]
Therefore, the result follows form \Cref{thm:multijet-transv}.
\end{proof}

Using two theorems above, we obtain the existence result of embeddings in compact cases.

\begin{corollary}[cf. Proposition 2.1.7 in \cite{Lau00}]
\label{cor:cptemb}
Let $X$, $Y$, and $\beta$ be as in \Cref{thm:inj-residual}.
Then, if $X$ is compact, admissible embeddings along $\beta$ form a residual subset in the space $\mathcal F^\beta(X,Y)$.
In particular, there is an embedding $F:X\to Y$ along $\beta$ such that
\begin{enumerate}[label={\rm(\roman*)}]
  \item $F\pitchfork D$ for each connected face $D\in\operatorname{bd}Y$;
  \item $F(\partial^\beta_\tau X)\subset\partial_\tau Y$ for each $\tau\in\Gamma_Y$;
\end{enumerate}
provided $\mathcal F^\beta(X,Y)$ is nonempty.
\end{corollary}
\begin{proof}
Note that, for a compact manifold $X$, a smooth map $F:X\to Y$ is an embedding if and only if it is an injective immersion.
Thus, the result follows from \Cref{thm:imm-thm} and \Cref{thm:inj-residual}.
\end{proof}

\begin{remark}
Actually, the embeddings satisfying the properties in \Cref{cor:cptemb} are sometimes called neat embeddings.
It is known that, for a neat embedding $X\hookrightarrow Y$, each point $p\in X$ admits a chart $\varphi:U\to\mathbb R^m\times\mathbb R_+^k$ on $Y$ centered at $p$ so that the following square is a pullback for some $n\le m$:
\[
\xymatrix{
  X\cap U \ar[d]_\varphi \ar@{^(->}[]+R+(1,0);[r] \ar@{}[dr]|(.4)\pbcorner & U \ar[d]^\varphi \\
  0\times\mathbb R^n\times\mathbb R_+^k \ar@{^(->}[]+R+(1,0);[r] & \mathbb R^m\times\mathbb R^k_+ }
\]
\end{remark}

\begin{example}
For positive integers $k\le n$, a $k$-fold $n$-dimensional bordism is a compact manifold $W$ with faces which is equipped with a full edging $\beta$ with $I^2$; i.e. $D(\beta)=\operatorname{bd}W$.
In this case, by virtue of \Cref{prop:edgefunc-polyhedra}, the space $\mathcal F^\beta(W,I^2\times\mathbb R^r)$ is non-empty for every non-negative integer $r$.
Then, \Cref{cor:cptemb} asserts that, for a sufficiently large $r>0$, there is an admissible embedding $W\hookrightarrow I^2\times\mathbb R^r$.
This gives rise to an embedding of classical bordisms (with corners) into the ``$(\infty,k)$-category of bordisms'' defined in \cite{CalaqueScheimbauer2015}.
\end{example}

More generally, the observation used in the proof of \Cref{cor:cptemb} is, in non-compact cases, valid for proper maps.
For this, we need to make more observation on proper maps.
We first discuss the existence.

\begin{proposition}
\label{prop:proper-exist}
Every manifold $X$ with corners admits a proper function $f:X\to\mathbb R$.
\end{proposition}
\begin{proof}
See Proposition I.5.11 in \cite{GG73}.
\end{proof}

\begin{corollary}
\label{cor:proper-nonemp}
Let $X$ and $Y$ be a manifold with finite faces, and let $\beta$ be an edging of $X$ with $Y$.
Suppose the space $\mathcal F^\beta(X,Y)$ is non-empty.
By abuse of notation, we regard the edging $\beta$ also as an edgin with $Y\times\mathbb R$ since there is a canonical identification $\Gamma_Y\cong\Gamma_{Y\times\mathbb R}$.
Then, there is a proper map $F:X\to Y\times\mathbb R$ along $\beta$.
In other words, the subset of proper maps of the space $\mathcal F^\beta(X,Y\times\mathbb R)$ is non-empty.
\end{corollary}
\begin{proof}
Take a smooth map $F':X\to Y$ along $\beta$ and a proper function $f:X\to\mathbb R$ by \Cref{prop:proper-exist}.
Define a function $F:X\to Y\times \mathbb R$ by $F(p):=(F'(p),f(p))$, then $F$ is obviously proper and along $\beta$.
\end{proof}

We are going to obtain proper embeddings by perturbing general proper maps.
To do this, thanks to \Cref{thm:imm-thm} and \Cref{thm:inj-residual}, it will suffice to prove that proper maps form an open subset.
We use the following consequence of \Cref{prop:proper-exist}:

\begin{lemma}[Lemma I.5.10 in \cite{GG73}]
\label{lem:compl-metric}
Every manifold $X$ with corners admits a complete metric $d_X:X\times X\to\mathbb R_+$.
\end{lemma}
\begin{proof}
Take a proper function $f:X\to\mathbb R$, and consider the map $d_f:X\times X\to\mathbb R_+$ given by
\[
d_f(p,p'):= |f(p)-f(p')|\,.
\]
It is easily verified that $d_f$ is a pseudo-metric and induces a weaker topology than the original one on $X$.
Choosing a metric $d'_X:X\times X\to\mathbb R_+$ compatible with the topology on $X$, we define a map $d_X:X\times X\to\mathbb R_+$ by
\[
d_X(p,p') := d'_X(p,p') + d_f(p,p')\,.
\]
Then, $d_X$ is a metric compatible with the original topology.
It remains to show that $d_X$ is complete.
According to Heine-Borel Theorem, it suffices to show that for each $p\in X$ and $\delta>0$, the subset
\[
\overline B_X(p;\delta)
:= \{p'\in X\mid d_X(p,p')\le\delta\} \subset X
\]
is compact.
This follows from the observation that $\overline B_X(p;\delta)$ is a closed subset of another subset
\[
\overline B_f(p;\delta)
:= \{p'\in X\mid d_f(p,p')\le\delta\}
= f^{-1}([f(p)-\delta,f(p)+\delta])
\subset X
\]
which is compact since $f$ is proper.
\end{proof}

Now, let $X$ and $Y$ be manifolds with corners.
By \Cref{lem:compl-metric}, we can choose a complete metric $d_Y:Y\times Y\to\mathbb R_+$ on $Y$, and define a map $\rho:C^\infty(X,Y)\times C^\infty(X,Y)\to\mathbb R_+$ by
\[
\rho(F,G)
:= \sup_{p\in X} \frac{d_Y(F(p),G(p))}{1+d_Y(F(p),G(p))}\,.
\]
It is verified that $\rho$ can be given as the map $\rho_0$ defined in section \ref{sec:whitney-top}, so that, by \Cref{lem:whitney-metric}, $\rho$ is a metric on $C^\infty(X,Y)$ inducing a weaker topology than Whitney $C^\infty$-topology.
In other words, a subset $B\subset C^\infty(X,Y)$ is open as soon as it is open with respect to the metric $\rho$.
We use this metric to prove the next result.

\begin{proposition}[cf. Lemma II.5.10 in \cite{GG73}]
\label{prop:proper-open}
Let $X$ and $Y$ be manifolds with corners.
Then, proper maps $X\to Y$ form a clopen (i.e. both closed and open) subset of $C^\infty(X,Y)$ (in Whitney $C^\infty$-topology).
\end{proposition}
\begin{proof}
Choose a complete metric $d_Y$ on $Y$, and we obtain a metric $\rho$ on $C^\infty(X,Y)$ as above.
By the argument above, it suffices to show that the subset of proper maps is clopen in $C^\infty(X,Y)$ with respect to the metric $\rho$.
More precisely, we will show that, for two maps $F,G\in C^\infty(X,Y)$ with $\rho(F,G)\le1/2$, $F$ is proper if and only if so is $G$.
Indeed, this implies that any proper (resp. non-proper) map $F$ admits an open neighborhood
\[
B_\rho\bigl(F;\frac12\bigr) := \left\{G\in C^\infty(X,Y)\ \middle| \rho(F,G)<\frac12\right\}
\]
consisting of proper (resp. non-proper) maps.

For each point $q\in Y$ and each positive number $\delta>0$, we put
\[
\overline B_Y(q;\delta):=\{q'\in Y\mid d_Y(q,q')\le \delta\}\,.
\]
Since the metric $d_Y$ is complete, by Heine-Borel Theorem for complete metric space, a subset of $Y$ is compact if and only if it is closed and contained in a subset of the form $\overline B_Y(q;\delta)$.
Hence, a smooth map $F:X\to Y$ is proper if and only if the subset
\[
F^{-1}(\overline B_Y(q;\delta)) \subset X
\]
is compact for every $q\in Y$ and $\delta>0$.
Now, suppose we have two maps $F,G:X\to Y$ with $\rho(F,G)\le1/2$.
Note that the condition $\rho(F,G)\le1/2$ is equivalent to that, for each $p\in X$, we have $d_Y(F(p),G(p))\le1$.
Thus, for each $q\in Y$ and $\delta>0$, the triangle inequality implies
\[
\begin{split}
G^{-1}(\overline B_Y(q;\delta)) &\subset F^{-1}(\overline B_Y(q;\delta+1))\mathrlap{\,,} \\
F^{-1}(\overline B_Y(q;\delta)) &\subset G^{-1}(\overline B_Y(q;\delta+1))\mathrlap{\,.}
\end{split}
\]
Since the left-hand sides are closed in the right-hand sides respectively, the formers are compact as soon as so are the latters.
It follows that $F$ is proper if and only if so is $G$, as required.
\end{proof}

\begin{corollary}
\label{cor:propemb-approx}
Let $X$ and $Y$ be as in \Cref{prop:proper-open}, and let $\beta$ be an eding of $X$ with $Y$.
Assume $2\cdot\dim X+1\ge\dim Y$.
Then, every proper map $F:X\to Y$ along $\beta$ can be arbitrarily approximated by proper admissible embeddings in the space $\mathcal F^\beta(X,Y)$.
\end{corollary}
\begin{proof}
The result immediately follows from \Cref{thm:adm-residual}, \Cref{thm:imm-thm}, \Cref{thm:inj-residual}, and \Cref{prop:proper-open}.
\end{proof}

Finally, we obtain our main result in this section:

\thmEmbedding

\begin{corollary}
\label{cor:emb-polyhedra}
Let $X$ be a manifold with faces equipped with an edging $\beta$ with a manifold $K$ realized as a convex polyhedron in the Euclidean space $\mathbb R^n$.
Then, for any sufficiently large integer $n>0$, there is a proper embedding $F:X\to K\times\mathbb R^n$ such that
\begin{enumerate}[label={\rm(\roman*)}]
  \item for each $\tau\in\Gamma_K$, $F(\partial^\beta_\tau X)\subset\partial_\tau K$;
  \item for each connected face $D\in\operatorname{bd}K$ of $Y$, $F\pitchfork D$.
\end{enumerate}
\end{corollary}

\begin{remark}
\Cref{cor:emb-polyhedra} is a generalization of Proposition 2.1.7 \cite{Lau00}.
Indeed, Laures proved the result for compact $\langle n\rangle$-manifolds $X$; see \Cref{ex:<n>-mfd}.
In particular, it does not covers non-compact cases while our result does.
\end{remark}

\bibliographystyle{plain}
\bibliography{mybiblio}
\end{document}